\documentclass[12pt]{amsproc}

\usepackage{fullpage}
\usepackage[utf8]{inputenc}
\usepackage[T1]{fontenc}
\usepackage{graphicx,subfigure}
\usepackage{amsmath,amsfonts,amssymb,mathtools}
\usepackage{stmaryrd}
\usepackage{mathrsfs,dsfont}
\usepackage{amsthm}
\usepackage{mathabx}
\usepackage{tabularx}
\usepackage{float}
\usepackage{amscd}

\usepackage{hyperref}

\usepackage{enumerate}

\usepackage{color}

\newcommand{\E}{\mathbb{E}}
\newcommand{\R}{\mathbb{R}}
\newcommand{\N}{\mathbb{N}}
\newcommand{\F}{\mathcal{F}}

\newcommand{\IA}{\mathcal{A}}
\newcommand{\IB}{\mathcal{B}}
\newcommand{\IL}{\Lambda}

\newcommand{\IY}{\mathcal{Y}}

\newcommand{\IP}{\mathcal{P}}

\newcommand{\XX}{\mathbb{X}}
\newcommand{\YY}{\mathbb{Y}}

\newcommand{\HH}{\mathcal{H}}

\newtheorem{theo}{Theorem}[section]

\newtheorem{rem}[theo]{Remark}

\newtheorem{propo}[theo]{Proposition}
\newtheorem{lemma}[theo]{Lemma}

\newtheorem{ass}{Assumption}

\usepackage{algorithm}
\usepackage{algorithmic}

\begin{document}

\title{Uniform weak error estimates for an asymptotic preserving scheme applied to a class of slow-fast parabolic semilinear SPDEs}

\author{Charles-Edouard Br\'ehier}

\address{Univ Lyon, Université Claude Bernard Lyon 1, CNRS UMR 5208, Institut Camille Jordan, 43 blvd. du 11 novembre 1918, F-69622 Villeurbanne cedex, France}
\email{brehier@math.univ-lyon1.fr}


\begin{abstract}
We study an asymptotic preserving scheme for the temporal discretization of a system of parabolic semilinear SPDEs with two time scales. Owing to the averaging principle, when the time scale separation $\epsilon$ vanishes, the slow component converges to the solution of a limiting evolution equation, which is captured when the time-step size $\Delta t$ vanishes by a limiting scheme. The objective of this work is to prove weak error estimates which are uniform with respect to $\epsilon$, in terms of $\Delta t$: the scheme satisfies a uniform accuracy property. This is a non trivial generalization of the recent article~\cite{BR} in an infinite dimensional framework. The fast component is discretized using the modified Euler scheme for SPDEs introduced in the recent work~\cite{B}. Proving the weak error estimates requires delicate analysis of the regularity properties of solutions of infinite dimensional Kolmogorov equations.
\end{abstract}

\maketitle

\section{Introduction}

Applied mathematicians need to face many challenges when they study multiscale stochastic systems, which appear in all fields of science and engineering, whether one is interested in theoretical understanding of the behavior of such systems, in their effective numerical approximation, or  in their applications for concrete models. We refer for instance to the monograph~\cite{PavliotisStuart} for a presentation of averaging and homogenization techniques applied to multiscale stochastic systems, and to~\cite{BerglundGentz,Kuehn} for a description of possible dynamical behaviors in such problems.

In this work, we study a class of systems of parabolic semilinear stochastic partial differential equations (SPDEs) of type
\begin{equation}\label{eq:SPDE-introfield}
\left\lbrace
\begin{aligned}
\partial_t\XX^\epsilon(t,\xi)&=\partial_\xi\bigl(a(\xi)\partial_\xi \XX^\epsilon(t,\xi)\bigr)+f(\XX^\epsilon(t,\xi),\YY^\epsilon(t,\xi))\\
\partial_t\YY^\epsilon(t,\xi)&=\frac{1}{\epsilon}\partial_\xi\bigl(a(\xi)\partial_\xi \YY^\epsilon(t,\xi)\bigr)+\sqrt{\frac{{2}}{{\epsilon}}}\dot{W}(t,\xi)
\end{aligned}
\right.
\end{equation}
where $t\in(0,\infty)$, $\xi\in(0,1)$, $\dot{W}$ is space-time white noise, and the mappings $a:[0,1]\to(0,\infty)$ and $f:\R^2\to\R$ are assumed to be sufficiently smooth. In addition, homogeneous Dirichlet boundary conditions are applied, and (deterministic) initial values $\XX^\epsilon(0,\cdot)=x_0^\epsilon(\cdot)$, $\YY^\epsilon(0,\cdot)=y_0^\epsilon(\cdot)$ are given. The time scale separation parameter is denoted by $\epsilon$.

Instead of considering the system~\eqref{eq:SPDE-introfield}, where the unknowns are random fields $\XX^\epsilon,\YY^\epsilon:[0,T]\times(0,1)\to \R$, in the sequel, we consider systems of stochastic evolution equations (SEEs) (see~\cite{DPZ}) of type
\begin{equation}\label{eq:SPDE-intro}
\left\lbrace
\begin{aligned}
d\XX^\epsilon(t)&=-\IL\XX^\epsilon(t)dt+F(\XX^\epsilon(t),\YY^\epsilon(t))dt\\
d\YY^\epsilon(t)&=-\frac{1}{\epsilon}\IL\YY^\epsilon(t)dt+\sqrt{\frac{{2}}{{\epsilon}}}dW(t),
\end{aligned}
\right.
\end{equation}
with initial values $\XX^\epsilon(0)=x_0^\epsilon$ and $\YY^\epsilon(0)=y_0^\epsilon$, where the unknowns $\XX^\epsilon,\YY^\epsilon:[0,T]\to H$ take values in an infinite dimensional Hilbert space $H$ (with $H=L^2(0,1)$ to consider the system~\eqref{eq:SPDE-introfield}). We refer to Section~\ref{sec:assumptions} below for precise assumptions on the linear operator $\IL$ and the nonlinearity $F$. The second component in the system~\eqref{eq:SPDE-intro} is driven by a cylindrical Wiener process.

When the time scale separation parameter $\epsilon$ vanishes, the slow component $\XX^\epsilon$ converges (in a suitable sense, under appropriate conditions) to the solution $\overline{\XX}$ of a deterministic evolution equation
\begin{equation}\label{eq:averaged-intro}
d\overline{\XX}(t)=-\IL\overline{\XX}(t)dt+\overline{F}(\overline{\XX}(t)),
\end{equation}
with initial value $\overline{\XX}(0)=x_0=\underset{\epsilon\to 0}\lim~x_0^\epsilon$, where the effect of the fast component $\YY^\epsilon$ is averaged out:
\[
\overline{F}(x)=\int F(x,y)d\nu(y)
\]
where $\nu$ is a Gaussian distribution. This result, known as the averaging principle, has been proved for SPDE systems~\eqref{eq:SPDE-intro} for the first time in~\cite{MR2480788}. We also refer to~\cite{MR2537194,MR2854919} for similar results, and to~\cite{B:2012,B:2020} for results on the rate of convergence when $\epsilon\to 0$ (in strong and weak senses). This list of references on the averaging principle for SPDE systems is not exhaustive. The system~\eqref{eq:SPDE-intro} considered in this work has a simplified structure compared with the systems treated in the literature, which is crucial in the analysis performed in this article. First, the evolution of the fast component $\YY^\epsilon$ does not depend on the slow component $\XX^\epsilon$: one can write $\YY^\epsilon(t)=\YY(t/\epsilon)$ (equality being understood in distribution), where $\YY$ is solution of a stochastic evolution equation which does not depend on $\epsilon$. Second, $\YY^\epsilon$ is an infinite dimensional Ornstein--Uhlenbeck process, in particular $\YY^\epsilon(t)$ is an $H$-valued Gaussian random variable for all $t\ge 0$. Note that the second condition is crucial for the arguments described below, however the first condition may be relaxed by introducing coefficients depending on the slow component in the evolution of the fast component. This generalization would require extra technical arguments in the analysis and in the proof of the error estimates, and is left for future work.

The objective of this article is to introduce and study an effective numerical scheme which allows to approximate the slow component $\XX^\epsilon$ in regimes where the time-scale separation parameter $\epsilon$ either vanishes, or has a fixed value. We only focus on the temporal discretization, even if in practice the approximation of solutions of SPDEs also needs a spatial discretization procedure (for instance using finite differences). Since the fast component $\YY^\epsilon$ evolves at the time scale $t/\epsilon$, a careful construction is required to be able to choose a time-step size $\Delta t$ which is independent of $\epsilon$. If one is interested only in the regime where $\epsilon$ vanishes, a popular method is the Heterogeneous Multiscale Method (HMM): see~\cite{MR2916381} for a general overview of this method, \cite{MR2165382} for its description for the approximation of multiscale stochastic differential equations, and~\cite{B:2012,B:2020,MR2881027} for its analysis and application to multiscale SPDE systems. The idea of HMM is to discretize slow and fast components using coarse and fine integrators respectively, depending on different time-step sizes. In addition, in HMM the coarse discretization of the slow component is inspired by the averaging principle, where the unknown averaged nonlinearity is approximated using the fine scheme. As a result, the HMM scheme is efficient when $\epsilon$ is small, but not in the regime where the time scale separation parameter does not vanish. In this article, we are interested in a different methodology, which allows to cover all regimes by a single numerical scheme, and where the time-step size $\Delta t$ can be chosen independently of $\epsilon$.

We propose to discretize the system~\eqref{eq:SPDE-intro} by the following numerical scheme
\begin{equation}\label{eq:scheme-intro}
\left\lbrace
\begin{aligned}
\XX_{n+1}^{\epsilon,\Delta t}&=\IA_{\Delta t}\bigl(\XX_n^{\epsilon,\Delta t}+\Delta t F(\XX_n^{\epsilon,\Delta t},\YY_{n+1}^{\epsilon,\Delta t})\bigr)\\
\YY_{n+1}^{\epsilon,\Delta t}&=\IA_{\frac{\Delta t}{\epsilon}}\YY_n^{\epsilon,\Delta t}+\sqrt{\frac{2\Delta t}{\epsilon}}\IB_{\frac{\Delta t}{\epsilon},1}\Gamma_{n,1}+\sqrt{\frac{2\Delta t}{\epsilon}}\IB_{\frac{\Delta t}{\epsilon},2}\Gamma_{n,2},
\end{aligned}
\right.
\end{equation}
where $\IA_{\Delta t}=(I+\Delta t\IL)^{-1}$, the linear operators $\IB_{\Delta t/\epsilon,1}$ and $\IB_{\Delta t/\epsilon,2}$ are chosen to satisfy~\eqref{eq:operators}, and $\Gamma_{n,1},\Gamma_{n,2}$ are independent cylindrical Gaussian random variables. We refer to Section~\ref{sec:scheme} for details on the construction of the scheme~\eqref{eq:scheme-intro}. On the one hand, the slow component $\XX^\epsilon$ is discretized using a semi-implicit Euler scheme. On the other hand, the fast component $\XX^\epsilon$ is discretized using the modified Euler scheme for parabolic semilinear SPDEs introduced in the recent work~\cite{B}. The construction of the scheme allows us to check that for any fixed value of $\epsilon$, $\XX_N^{\epsilon,\Delta t}$ converges to $\XX^\epsilon(T)$ (with $T=N\Delta t$) when $\Delta t\to 0$, and that the scheme is asymptotic preserving, in the following sense. First, for any value of the time-step size $\Delta t$ and all integers $n\ge 1$, one has $\XX_n^{\epsilon,\Delta t}\to \XX_n^{\Delta t}$ when $\epsilon\to 0$, determined by the limiting scheme
\begin{equation}\label{eq:limitingscheme-intro}
\XX_{n+1}^{\Delta t}=\IA_{\Delta t}\XX_n^{\Delta t}+\Delta t\IA_{\Delta t}F(\XX_n^{\Delta t},\IL^{-\frac12}\Gamma_{n})
\end{equation}
with initial value $\XX_0^{\Delta t}=x_0=\underset{\epsilon\to 0}\lim~\XX_0^{\epsilon,\Delta t}$. The last but not the least, the limiting scheme~\eqref{eq:limitingscheme-intro} is consistent with the limiting evolution equation~\eqref{eq:averaged-intro}: one has $\XX_N^{\Delta t}\to \overline{\XX}(T)$ when $\Delta t=T/N\to 0$. Note that the choice of the modified Euler scheme to discretize the fast component in the scheme~\eqref{eq:scheme-intro} is essential to obtain the last property: it is not satisfied when the standard Euler scheme is used. We refer to Section~\ref{sec:main} for rigorous statements of the properties above, in particular about the need to consider convergence in distribution.

The asymptotic preserving property is written as the fact that the diagram
\[
\begin{CD}
\XX_N^{\epsilon,\Delta t}     @>{N \to \infty}>> \XX^\epsilon(T) \\
@VV{\epsilon\to 0}V        @VV{\epsilon\to 0}V\\
\XX_N^{\Delta t}     @>{N \to \infty}>> \overline{\XX}(T)
\end{CD}
\]
is commutative. Asymptotic preserving methods are popular in the field of numerical analysis of multiscale PDEs, see for instance the recent review~\cite{Jin} and the references therein. In recent years, they have been studied for stochastic systems, for instance in~\cite{AyiFaou,BR}.

The asymptotic preserving property for the scheme~\eqref{eq:scheme} is also proved in the manuscript~\cite{B} in which the modified Euler scheme (used here to discretize the fast component $\YY^\epsilon$) has been introduced and studied. In this article, we make a further major step in the analysis and prove a form of uniform accuracy property. The main result of this manuscript, Theorem~\ref{theo:UA}, can be written as follows: under appropriate regularity and growth conditions, one has the uniform weak error estimates
\begin{equation}\label{eq:UA-intro}
\underset{\epsilon\in(0,\epsilon_0)}\sup~\big|\E[\varphi(\XX_N^{\epsilon,\Delta t})]-\E[\varphi(\XX^\epsilon(T))]\big|\le C_\kappa(T,\varphi,x_0) \Delta t^{\frac13-\kappa}
\end{equation}
where $\kappa\in(0,\frac13)$ is an arbitrarily small positive auxiliary parameter, $\varphi:H\to\R$ is a mapping of class $\mathcal{C}^3$ with bounded derivatives, and $C_\kappa(T,\varphi,x_0)\in(0,\infty)$.

The order of convergence $1/3$ appearing in the right-hand side of the uniform weak error estimates~\eqref{eq:UA-intro} may not be optimal. Indeed, for a fixed value of $\epsilon$, the order of convergence of the scheme~\eqref{eq:scheme-intro} is at least $1/2$ (see Proposition~\ref{propo:error_fixedepsilon}), and for the limiting scheme the order of convergence is $1$ (see Proposition~\ref{propo:error_limitingscheme-averagedequation}). The reduction of the order of convergence is due to the strategy of the proof, which consists in obtaining two different error estimates based on the commutative diagram above. Obtaining a positive order of convergence is already a non trivial challenge which is solved in this manuscript for the first time in the context of stochastic PDEs. In order to prove the uniform weak error estimate~\eqref{eq:UA-intro}, we follow the same strategy as in~\cite{BR} (where a reduction of the order of convergence is also obtained), which deals with finite dimensional stochastic differential equations. Substantial modifications due to the infinite dimensional setting are required. Precisely, the main difficulties appear for the proof of a direct error estimate for the weak error, see Proposition~\ref{propo:error_fixedepsilon}. Compared with~\cite{BR}, additional arguments concerning the regularity properties of the solutions of the associated Kolmogorov equation need to be studied carefully, see Lemma~\ref{lem:uepsilon} and Section~\ref{sec:proof-UA2}.

The manuscript is organized as follows. First, the setting is described in Section~\ref{sec:setting}. Preliminary results on the SPDE system are recalled in Section~\ref{sec:SPDE} and the averaging principle is discussed in Section~\ref{sec:averaging}. The numerical scheme  studied in this work is presented in Section~\ref{sec:scheme}. Then the main results of this work are stated in Section~\ref{sec:main}: the asymptotic preserving property is studied in Section~\ref{sec:AP} and the main result, Theorem~\ref{theo:UA}, is stated in Section~\ref{sec:UA}. The auxiliary error estimates required to prove Theorem~\ref{theo:UA} are stated in Section~\ref{sec:auxiliary-error}. Section~\ref{sec:Kolmogorov} provides the regularity properties of solutions of infinite dimensional Kolmogorov equations, see Lemma~\ref{lem:uepsilon} and Lemma~\ref{lem:ubarDeltat}.  The reminder of the manuscript is devoted to proving the error estimates, in Section~\ref{sec:proof-AP} (results from Section~\ref{sec:AP}) and Section~\ref{sec:proof-UA} (results from Section~\ref{sec:auxiliary-error}).

\section{Setting}\label{sec:setting}

\subsection{Notation}\label{sec:notation}

The set of integers is denoted by $\N=\{1,\ldots\}$.

Let $H$ be a separable infinite dimensional Hilbert space, equipped with inner product and norm denoted by $\langle\cdot,\cdot\rangle$ respectively. Let also $\HH=H^2$ be the Hilbert space, with inner product and norm as follows: for all $(x_1,y_1),(x_2,y_2),(x,y)\in\HH$, set
\begin{align*}
\langle (x_1,y_1),(x_2,y_2)\rangle_{\HH}&=\langle x_1,x_2\rangle+\langle y_1,y_2\rangle\\
\|(x,y)\|_{\HH}^2&=\|x\|^2+\|y\|^2.
\end{align*}
The set of bounded linear operators from $H$ to $H$ is denoted by $\mathcal{L}(H)$, this set is a Banach space, with the norm $\|\cdot\|_{\mathcal{L}(H)}$ defined by
\[
\|L\|_{\mathcal{L}(H)}=\underset{x\in H\subset\{0\}}\sup~\frac{|Lx|}{|x|}.
\]

The random variables and the stochastic processes considered in this article are defined on a probability space denoted by $(\Omega,\mathcal{F},\mathbb{P})$. This probability space is equipped with a filtration $\bigl(\mathcal{F}_t\bigr)_{t\ge 0}$ which is assumed to satisfy the usual conditions. The expectation operator is denoted by $\E[\cdot]$.

Let $\bigl(\beta_j\bigr)_{j\in\N}$ denote a sequence of independent standard real-valued Wiener processes, adapted to the filtration $\bigl(\mathcal{F}_t\bigr)_{t\ge 0}$. The cylindrical Wiener process $\bigl(W(t)\bigr)_{t\ge 0}$ on $H$ is formally defined as
\begin{equation}\label{eq:Wiener}
W(t)=\sum_{j\in\N}\beta_j(t)e_j
\end{equation}
where $\bigl(e_j\bigr)_{j\in\N}$ is an arbitrary complete orthonormal system of $H$.

The following terminology is used in the sequel: a random variable $\Gamma$ is called a cylindrical Gaussian random variable if
\[
\Gamma=\sum_{j\in\N}\gamma_j e_j,
\]
where $\bigl(\gamma_j\bigr)_{j\in\N}$ is a sequence of independent standard real-valued Gaussian random variables ($\gamma_j\sim\mathcal{N}(0,1)$ for all $j\in\N$).

Some of the proofs below require tools from Malliavin calculus~\cite{Nualart}. We do not give precise definitions, instead let us state the notation used in this article and quote the most useful results. If $\Theta$ is an $H$-valued random variable, $\mathcal{D}_s^h\Theta\in H$ is the Malliavin derivative of $\Theta$ at time $s$ in direction $h\in H$. For instance, this means that
\[
\mathcal{D}_s^h\bigl(\int_0^TL(t)dW(t)\bigr)=L(s)h
\]
if $t\in[0,T]\mapsto L(t)\in\mathcal{L}(H)$ is an adapted process. In addition, if $\Theta$ is $\mathcal{F}_t$-measurable, then $\mathcal{D}_s^h\Theta=0$ for all $s>t$. The Malliavin derivative satisfies a chain rule property: if $\Phi:H\to H$ is of class $\mathcal{C}^1$ with bounded derivative, then for all $s\ge 0$ and $h\in H$ one has
\[
\mathcal{D}_s^h\phi(\Theta)=D\phi(\Theta).\mathcal{D}_s^h\Theta.
\]
The same type of notation and results are satisfied for $\R$-valued random variables $\theta$. Finally, one has the following integration by parts formula, which is essential for the proof of weak error estimates, see~\cite{Debussche:11}: if $\theta$ is $\R$-valued random variable and if $(t,s)\mapsto\phi(t,s)\in\R$ is a given deterministic function, for all $j\in\N$, one has
\begin{equation}\label{eq:MalliavinIBP}
\E\bigl[\theta\int_0^t\phi(t,s)d\beta_j(s)\bigr]=\int_0^t\E[\mathcal{D}_s^{e_j}\theta \phi(t,s)d\beta_j(s)].
\end{equation}

Finally, introduce the following notation. If $\varphi:H\to\R$ is a mapping of class $\mathcal{C}^{3}$ with bounded derivatives of order $1,2,3$, set
\begin{align*}
\vvvert\varphi\vvvert_1&=\underset{x,h\in H}\sup~\frac{|D\varphi(x).h|}{|h|},\\
\vvvert\varphi\vvvert_2&=\vvvert\varphi\vvvert_1+\underset{x,h_1,h_2\in H}\sup~\frac{|D^2\varphi(x).(h_1,h_2)|}{|h_1||h_2|},\\
\vvvert\varphi\vvvert_3&=\vvvert\varphi\vvvert_2+\underset{x,h_1,h_2,h_3\in H}\sup~\frac{|D^3\varphi(x).(h_1,h_2,h_3)|}{|h_1||h_2||h_3|}.
\end{align*}

If $\phi:\HH\to\R$ is a function of class $\mathcal{C}^2$, for all $h,k\in H$ and $(x,y)\in\HH$, the following notation is used below:
\begin{align*}
D_x\phi(x,y).h&=D\phi(x,y).(h,0)\\
D_y\phi(x,y).h&=D\phi(x,y).(0,h)\\
D_x^2\phi(x,y).(h,k)&=D^2\phi(x,y).\bigl((h,0),(k,0)\bigr)\\
D_y^2\phi(x,y).(h,k)&=D^2\phi(x,y).\bigl((0,h),(0,k)\bigr)\\
D_xD_y\phi(x,y).(h,k)&=D^2\phi(x,y).\bigl((h,0),(0,k)\bigr).
\end{align*}
Similar notation is used for third order derivatives below.

In this work, the values of constants $C\in(0,\infty)$ (which may depend on auxiliary parameters) appearing in the error estimates may vary from line to line. All the constants are independent of the time scale separation parameter $\epsilon$ and of the time-step size $\Delta t$.

\subsection{Assumptions}\label{sec:assumptions}

The stochastic evolution equations considered in this work are driven by an unbounded self-adjoint linear operator $-\IL:D(\IL)\subset H\to H$, which is assumed to satisfy the following conditions.

\begin{ass}\label{ass:L}
There exists a complete orthonormal system $\bigl(e_j\bigr)_{j\in\N}$ of $H$ and a non-decreasing sequence $\bigl(\lambda_j\bigr)_{j\in\N}$ of positive real numbers, such that
\[
\IL e_j=\lambda_je_j
\]
for all $j\in\N$. In addition, it is assumed that there exists $c_{\IL}\in(0,\infty)$ that $\lambda_j\sim c_{\IL}j^2$ when $j\to\infty$.
\end{ass}
The self-adjoint unbounded linear operator $-\IL$ generates a semigroup which is denoted by $\bigl(e^{-t\IL}\bigr)_{t\ge 0}$. Precisely, for all $t\ge 0$ and $x\in H$, set
\[
e^{-t\IL}x=\sum_{j\in\N}e^{-t\lambda_j}\langle x,e_j\rangle e_j.
\]
In addition, for all $\alpha\in[-1,1]$, define the self-adjoint linear operators $\IL^{\alpha}$ such that
\[
\IL^\alpha e_j=\lambda_j^\alpha e_j
\]
for all $j\in\N$. Equivalently,
\[
\IL^\alpha x=\sum_{j\in\N}\lambda_j^\alpha \langle x,e_j\rangle e_j.
\]
If $\alpha\in[-1,0]$, $\IL^\alpha$ is a bounded linear operator from $H$ to $H$ and the expression above is well-defined for all $x\in H$. For all $\alpha\in[0,1]$, introduce the notation
\[
|x|_\alpha=\bigl(\sum_{j\in\N}\lambda_j^{2\alpha}\langle x,e_j\rangle^2\bigr)^{\frac12}\in[0,\infty],
\]
then $\IL^\alpha$ is an unbounded self-adjoint linear operator with domain $D(\IL^\alpha)=H^\alpha$, defined by
\[
H^\alpha=\{x\in H;~|x|_\alpha<\infty\}.
\]
Finally, let us recall two standard inequalities: for all $\alpha\in[0,1]$, there exists $C_\alpha\in(0,\infty)$ such that for all $t\in(0,\infty)$ and all $x\in H$, one has
\begin{equation}\label{eq:smoothing}
|e^{-t\IL}x|\le C_\alpha t^{-\alpha}|\IL^{-\alpha}x|
\end{equation}
and for all $x\in H^\alpha$ one has
\begin{equation}\label{eq:regularity}
|\bigl(e^{-t\IL}-I\bigr)x|\le C_\alpha t^\alpha|x|_\alpha.
\end{equation}

Let us now state the regularity and growth assumptions for the nonlinear operator $F$.
\begin{ass}\label{ass:F}
Let $F:\HH\to H$ be a mapping of class $\mathcal{C}^3$, with bounded derivatives of order $1,2,3$.
\end{ass}
Note that in particular $F$ is globally Lipschitz continuous: there exists $C_F\in(0,\infty)$ such that for all $(x_1,y_1),(x_2,y_2)\in\HH$, one has
\[
|F(x_2,y_2)-F(x_1,y_1)|\le C_F\bigl(|x_2-x_1|+|y_2-y_1|\bigr).
\]

\begin{ass}\label{ass:init}
For all $\epsilon\in(0,\epsilon_0)$, let $x_0^\epsilon\in H$ and $y_0^\epsilon\in H$, satisfying the following conditions: there exists $x_0\in H$ such that 
\begin{equation}
|x_0^\epsilon-x_0|\underset{\epsilon\to 0}\to 0.
\end{equation}
Moreover, there exists $\kappa_0\in(0,1)$, such that for all $\kappa\in[0,\kappa_0]$, there exists $C_\kappa\in(0,\infty)$ such that $x_0\in H^{\frac{\kappa}{2}}$ and
\begin{equation}
\underset{\epsilon\in(0,\epsilon_0)}\sup~\bigl(|\IL^{\frac{\kappa}{2}}x_0^\epsilon|+|\IL^{\frac{\kappa}{2}}y_0^\epsilon|\bigr)\le C_\kappa\bigl(1+|\IL^{\frac{\kappa}{2}}x_0|\bigr).
\end{equation}
\end{ass}

\subsection{SPDE system}\label{sec:SPDE}

In this work, we study the behavior of a class of numerical methods applied to the following stochastic evolution equations system, both for fixed $\epsilon\in(0,\epsilon_0)$ and in the regime $\epsilon\to 0$:
\begin{equation}\label{eq:SPDE}
\left\lbrace
\begin{aligned}
d\XX^\epsilon(t)&=-\IL\XX^\epsilon(t)dt+F(\XX^\epsilon(t),\YY^\epsilon(t))dt\\
d\YY^\epsilon(t)&=-\frac{1}{\epsilon}\IL\YY^\epsilon(t)dt+\sqrt{\frac{{2}}{{\epsilon}}}dW(t),
\end{aligned}
\right.
\end{equation}
with initial values $\XX^\epsilon(0)=x_0^\epsilon$ and $\YY^\epsilon(0)=y_0^\epsilon$, which satisfy Assumption~\ref{ass:init}. In the system above, the linear operator $\IL$ satisfies Assumption~\ref{ass:L}, the nonlinear operator $F$ satisfies Assumption~\ref{ass:F} and the cylindrical Wiener process $\bigl(W(t)\bigr)_{t\ge 0}$ is given by~\eqref{eq:Wiener}.

The following result is a standard well-posedness statement for the system~\eqref{eq:SPDE}, see for instance~\cite{DPZ}.
\begin{propo}\label{propo:SPDE}
Let Assumptions~\ref{ass:L} and~\ref{ass:F} be satisfied. For all $\epsilon\in(0,\epsilon_0)$, the system~\eqref{eq:SPDE} admits a unique global mild solution $\bigl(\XX^\epsilon(t),\YY^\epsilon(t)\bigr)_{t\ge 0}$, with initial values $\XX^\epsilon(0)=x_0^\epsilon$ and $\YY^\epsilon(0)=y_0^\epsilon$: for all $t\ge 0$,
\begin{equation}\label{eq:mild}
\left\lbrace
\begin{aligned}
\XX^\epsilon(t)&=e^{-t\IL}x_0^\epsilon+\int_0^t e^{-(t-s)\IL}F(\XX^\epsilon(s),\YY^\epsilon(s))ds\\
\YY^\epsilon(t)&=e^{-\frac{t}{\epsilon}\IL}y_0^\epsilon+\sqrt{\frac{{2}}{{\epsilon}}}\int_0^te^{-\frac{t-s}{\epsilon}\IL}dW(s).
\end{aligned}
\right.
\end{equation}

If Assumption~\ref{ass:init} is satisfied, one has the following moment bounds properties, uniformly with respect to $\epsilon\in(0,\epsilon_0)$: for all $T\in(0,\infty)$,
\begin{equation}\label{eq:momentboundmild}
\underset{\epsilon\in(0,\epsilon_0)}\sup~\Bigl(\underset{0\le t\le T}\sup~\E[|\XX^\epsilon(t)|^2]+\underset{t\ge 0}\sup~\E[|\YY^\epsilon(t)|^2]\Bigr)<\infty.
\end{equation}
\end{propo}

\subsection{The averaging principle}\label{sec:averaging}

Let us describe the behavior of the solution of the system~\eqref{eq:SPDE} when $\epsilon\to 0$. Note that the parameter $\epsilon$ introduces a time-scale separation in the evolution of the two components. On the one hand, the fast component $\YY^\epsilon$ is an $H$-valued Ornstein--Uhlenbeck process, and one has the equality in distribution
\[
\bigl(\YY^\epsilon(t)\bigr)_{t\ge 0}=\bigl(e^{-\frac{t}{\epsilon}\IL}y_0^\epsilon+\YY(\frac{t}{\epsilon})\bigr)_{t\ge 0},
\]
where the Ornstein--Uhlenbeck process $\YY$ is the solution of the stochastic evolution equation
\[
d\YY(t)=-\IL\YY(t)dt+\sqrt{2}dW(t),
\]
with initial value $\YY(0)=0$. It is straightforward to check that the $H$-valued process $\YY$ is ergodic and that its unique invariant distribution is the Gaussian distribution $\nu=\mathcal{N}(0,\IL^{-1})$. In addition, for all $t\in(0,\infty)$, $\YY^\epsilon(t)$ converges in distribution to $\nu$ when $\epsilon\to 0$. On the other hand, the component $\XX^\epsilon$ evolves slowly, and the behavior of the fast component implies that an averaging principle holds: when $\epsilon\to 0$, $\XX^\epsilon$ converges (in various suitable senses) to the solution $\overline{\XX}$ of an evolution equation where the effect of the fast component has been averaged out, with a nonlinearity depending on the the Gaussian distribution $\nu$.

In order to state a rigorous version of the averaging principle, introduce the nonlinear operator $\overline{F}:H\to H$ defined as follows: for all $x\in H$, set
\begin{equation}\label{eq:Fbar}
\overline{F}(x)=\int F(x,y)d\nu(y)=\E_{\nu}[F(x,\YY)]=\E[F(x,\IL^{-\frac12}\Gamma)],
\end{equation}
where $\Gamma$ is a cylindrical Gaussian random variable. Observe that if $F$ satisfies Assumption~\ref{ass:F}, then the mapping $\overline{F}:H\to H$ is of class $\mathcal{C}^3$, with bounded derivatives of order $1,2,3$. In particular, $\overline{F}$ is globally Lipschitz continuous.

The asymptotic behavior of the slow component $\XX^\epsilon$ in~\eqref{eq:SPDE} is described by the solution of the averaged equation:
\begin{equation}\label{eq:averaged}
\frac{d\overline{\XX}(t)}{dt}=-\IL\overline{\XX}(t)+\overline{F}(\overline{\XX}(t)),
\end{equation}
with initial value $\overline{\XX}(0)=x_0=\underset{\epsilon\to 0}\lim~x_0^\epsilon$ (see Assumption~\ref{ass:init}). The deterministic evolution equation~\eqref{eq:averaged} admits a unique global mild solution: for all $t\ge 0$,
\[
\overline{\XX}(t)=e^{-t\IL}x_0+\int_0^t e^{-(t-s)\IL}\overline{F}(\overline{\XX}(s))ds.
\]

One has the following convergence result.
\begin{propo}\label{propo:averaging}
Let Assumptions~\ref{ass:L},~\ref{ass:F} and~\ref{ass:init} be satisfied. Then for all $t\ge 0$, $\XX^\epsilon(t)$ converges to $\overline{\XX}(t)$ when $\epsilon\to 0$, in the mean-square sense.

In addition, one has the following weak error estimate: for all $\kappa\in(0,\kappa_0)$ and $T\in(0,\infty)$, there exists $C_\kappa(T)\in(0,\infty)$ such that for any function $\varphi:H\to\R$ of class $\mathcal{C}^3$ with bounded derivatives of order $1,2,3$ and all $\epsilon\in(0,\epsilon_0)$ one has
\begin{equation}\label{eq:error_averaging}
\big|\E[\varphi(\XX^\epsilon(T))]-\E[\varphi(\overline{\XX}(T))]\big|\le C_\kappa(T)\vvvert\varphi\vvvert_3 \epsilon^{1-\kappa}\bigl(1+|\IL^{\frac{\kappa}{2}}x_0|^2\bigr).
\end{equation}
\end{propo}

For the mean-square convergence result, see for instance~\cite{B:2012,MR2480788}. The weak error estimate~\eqref{eq:error_averaging} is obtained in~\cite{B:2012,B:2020} (under weaker regularity conditions on $\varphi$). It is also retrieved in this article as a consequence of the auxiliary error estimates stated and proved below.

\subsection{Numerical scheme}\label{sec:scheme}

Let us introduce the notation required to define the numerical scheme studied in this work. The time-step size is denoted by $\Delta t$. Without loss of generality, it is assumed that there exists a fixed time $T\in(0,\infty)$ such that $\Delta t=T/N$ for some integer $N\in\N$. In the sequel, the limit $\Delta t\to 0$ is considered by letting $N\to\infty$ with $T$ fixed. In addition, it is assumed that $\Delta t\in(0,\Delta t_0)$. To simplify the notation, let $\tau=\frac{\Delta t}{\epsilon}$.

Let $\IA_{\Delta t}=(I+\Delta t\IL)^{-1}$, and introduce also linear operators $\IA_\tau$, $\IB_{\tau,1}$ and $\IB_{\tau,2}$ assumed to satisfy (see~\cite[Section~2]{B}) 
\begin{equation}\label{eq:operators}
\IA_\tau=(I+\tau\IL)^{-1},\quad \IB_{\tau,1}=\frac{1}{\sqrt{2}}(I+\tau\IL)^{-1},\quad \IB_{\tau,2}\IB_{\tau,2}^\star=\frac12(I+\tau\IL)^{-1},
\end{equation}
where $L^\star$ is the adjoint of a linear operator $L$.

Note that $\|\IA_{\Delta t}\|_{\mathcal{L}(H)}\le 1$ for all $\Delta t\in(0,\Delta t_0)$. In addition, one has the following property: for all $\alpha\in[0,1]$, there exists $C_\alpha\in(0,\infty)$ such that for all $\Delta t\in(0,\Delta t_0)$ and $n\in\{1,\ldots,N\}$, one has
\begin{equation}\label{eq:smoothingnum}
\|\IL^\alpha \IA_{\Delta t}^n\|_{\mathcal{L}(H)}\le \frac{C_\alpha}{(n\Delta t)^\alpha}.
\end{equation}

Let $\bigl(\Gamma_{n,1}\bigr)_{n\ge 0}$ and $\bigl(\Gamma_{n,2}\bigr)_{n\ge 0}$ be two independent sequences of independent cylindrical Gaussian random variables. Then the scheme is defined as follows: for all $n\ge 0$, set
\begin{equation}\label{eq:scheme}
\left\lbrace
\begin{aligned}
\XX_{n+1}^{\epsilon,\Delta t}&=\IA_{\Delta t}\bigl(\XX_n^{\epsilon,\Delta t}+\Delta t F(\XX_n^{\epsilon,\Delta t},\YY_{n+1}^{\epsilon,\Delta t})\bigr)\\
\YY_{n+1}^{\epsilon,\Delta t}&=\IA_{\frac{\Delta t}{\epsilon}}\YY_n^{\epsilon,\Delta t}+\sqrt{\frac{2\Delta t}{\epsilon}}\IB_{\frac{\Delta t}{\epsilon},1}\Gamma_{n,1}+\sqrt{\frac{2\Delta t}{\epsilon}}\IB_{\frac{\Delta t}{\epsilon},2}\Gamma_{n,2},
\end{aligned}
\right.
\end{equation}
with initial values $\XX_0^{\epsilon,\Delta t}=x_0^\epsilon$ and $\YY_0^{\epsilon,\Delta t}=y_0^\epsilon$.

On the one hand, in the scheme~\eqref{eq:scheme}, the slow component of~\eqref{eq:SPDE} is discretized using a linear implicit Euler scheme: the definition can be rewritten as
\[
\XX_{n+1}^{\epsilon,\Delta t}=\XX_{n}^{\epsilon,\Delta t}-\tau\IL \XX_{n+1}^{\epsilon,\Delta t}+\tau F(\XX_n^{\epsilon,\Delta t},\YY_{n+1}^{\epsilon,\Delta t}),
\]
which means that the linear part is discretized implicitly, whereas the nonlinearity part is discretized explicitly with respect to the slow component $\XX^\epsilon$ and implicitly with respect to the fast component $\YY^\epsilon$. This choice is motivated by the analysis of the scheme when $\epsilon\to 0$.

On the other hand, the fast component is discretized using the modified Euler scheme introduced in the recent work~\cite{B}: we refer to this preprint for the construction and the properties of this scheme, below we only recall the notation required in the analysis of the scheme~\eqref{eq:scheme}. Some properties of the scheme~\eqref{eq:scheme} are in fact already studied in~\cite[Section~9.1]{B}. As explained in~\cite{BR} (SDE case) and~\cite{B} SPDE case), discretizing the fast component using the standard linear implicit Euler scheme would not be appropriate in the regime $\epsilon\to 0$. The main advantage of the modified Euler scheme is the preservation of the invariant distribution $\nu$, for any choice of the time-step size $\Delta t$. The main tool to analyze the modified Euler scheme is the interpretation as the accelerated exponential Euler scheme applied to a modified stochastic evolution equation (see~\cite[Section~3.3]{B}): using the notation $\tau=\frac{\Delta t}{\epsilon}$ and $t_n^\tau=n\tau$, one has the equality in distribution
\[
\bigl(\YY_n^{\epsilon,\Delta t}\bigr)_{n\ge 0}=\bigl(\IY_n^{\tau}\bigr)_{n\ge 0},
\]
where
\begin{equation}\label{eq:scheme-interp}
\IY_{n+1}^\tau=e^{-\tau\IL_{\tau}}\IY_n^\tau+\sqrt{2}\int_{t_n^\tau}^{t_{n+1}^\tau}e^{-(t_{n+1}^{\tau}-t)\IL_\tau}Q_\tau^{\frac12}dW(s),
\end{equation}
with initial value $\IY_0^\tau=y_0^\epsilon=\YY_0^{\epsilon,\Delta t}$. The linear operators $\IL_\tau$, $Q_\tau$ and $Q_\tau^{\frac12}$ are defined by the following expressions:
\begin{equation}\label{eq:linearoperators}
\left\lbrace
\begin{aligned}
\IL_\tau x&=\sum_{j\in\N}\lambda_{\tau,j}\langle x,e_j\rangle e_j\\
Q_\tau x&=\sum_{j\in\N}q_{\tau,j}\langle x,e_j\rangle e_j\\
Q_\tau^{\frac12}x&=\sum_{j\in\N}\sqrt{q_{\tau,j}}\langle x,e_j\rangle e_j,
\end{aligned}
\right.
\end{equation}
where the eigenvalues are defined for all $j\in\N$ and $\tau\in(0,\infty)$ by
\begin{equation}\label{eq:eigenvalues}
\left\lbrace
\begin{aligned}
\lambda_{\tau,j}&=\frac{\log(1+\tau\lambda_j)}{\tau}\in(0,\lambda_j)\\
q_{\tau,j}&=\frac{\log(1+\tau\lambda_j)}{\tau\lambda_j}\in(0,1).
\end{aligned}
\right.
\end{equation}
The auxiliary process defined by~\eqref{eq:scheme-interp} satisfies $\IY_n^\tau=\IY^\tau(t_n^\tau)$ for all $n\ge 0$, where the process $\bigl(\IY^\tau(t)\bigr)_{t\ge 0}$ is the mild solution of the modified stochastic evolution equation
\begin{equation}\label{eq:auxOU}
d\IY^\tau(t)=-\IL_\tau\IY^\tau(t)dt+\sqrt{2}Q_\tau^{\frac12}dW(t).
\end{equation}

Let $\alpha\in[0,1]$ and set
\[
C_\alpha=\underset{z\in(0,\infty)}\sup~z^{-\alpha}|\frac{\log(1+z)}{z}-1|<\infty.
\]
One then obtains the following bonds (see~\cite[Section~5.1]{B}): for all $\tau\in(0,\infty)$ and $j\in\N$,
\begin{equation}\label{eq:error-eigenvalues}
\left\lbrace
\begin{aligned}
&0\le \lambda_{j}-\lambda_{\tau,j}\le C_\alpha \tau^\alpha \lambda_j^{1+\alpha}\\
&0\le 1-q_{\tau,j}\le C_\alpha\tau^{\alpha}\lambda_j^\alpha,
\end{aligned}
\right.
\end{equation}
which are used below to analyze the error.

Let us provide two results on the numerical scheme~\eqref{eq:scheme} which are used below to prove the main result of this article.

\begin{lemma}\label{lem:momentboundsscheme-epsilonDeltat}
Let Assumption~\ref{ass:init} be satisfied. Then one has
\begin{equation}\label{eq:boundYnepsilonDeltat}
\underset{\epsilon\in(0,\epsilon_0),\Delta t\in(0,\Delta t_0)}\sup~\underset{n\ge 0}\sup~\E[|\YY_n^{\epsilon,\Delta t}|^2]<\infty.
\end{equation}
Moreover, for all $T\in(0,\infty)$, there exists $C(T)\in(0,\infty)$ such that for all $\epsilon\in(0,\epsilon_0)$ and $\Delta t=T/N\in(0,\Delta t_0)$, one has
\begin{equation}\label{eq:boundXnepsilonDeltat}
\underset{0\le n\le N}\sup~\E[|\XX_n^{\epsilon,\Delta t}|^2]\le C(T)(1+|x_0|^2).
\end{equation}
\end{lemma}

\begin{proof}[Proof of Lemma~\ref{lem:momentboundsscheme-epsilonDeltat}]
Let us first prove the inequality~\eqref{eq:boundYnepsilonDeltat}. One has, for all $n\ge 0$, the equality in distribution
\[
\YY_{n}^{\epsilon,\Delta t}=\IY_{n}^\tau=\IY^\tau(t_n^\tau)=e^{-t_n^\tau\IL_\tau}y_0^\epsilon+\sqrt{2}\int_0^{t_n^\tau}e^{-(t_n^\tau-t)\IL_\tau}Q_\tau^{\frac12}dW(t).
\]
On the one hand, for all $n\ge 0$ and $\Delta t\in(0,\Delta t_0)$, one has
\[
\underset{\epsilon\in(0,\epsilon_0)}\sup~|e^{-t_n^\tau\IL_\tau}y_0^\epsilon|\le \underset{\epsilon\in(0,\epsilon_0)}\sup~|y_0^\epsilon|<\infty,
\]
owing to Assumption~\ref{ass:init}, since $\lambda_{\tau,j}\ge 0$ for all $j\in\N$ and $\tau\in(0,\infty)$.

On the other hand, using It\^o's isometry formula, it is straightforward to check that one has
\[
\E[\big|\sqrt{2}\int_0^{t_n^\tau}e^{-(t_n^\tau-t)\IL_\tau}Q_\tau^{\frac12}dW(s)|^2]\le \int|y|^2d\nu(y)=\sum_{j\in\N}\frac{1}{\lambda_j}<\infty.
\]
This concludes the proof of the inequality~\eqref{eq:boundYnepsilonDeltat}. Let us now prove the inequality~\eqref{eq:boundXnepsilonDeltat}. Since $F$ is globally Lipschitz continuous (Assumption~\ref{ass:F}) and since $\|\IA_{\Delta t}\|_{\mathcal{L}(H)}\le 1$, for all $n\ge 0$, one has
\begin{align*}
|\XX_{n+1}^{\epsilon,\Delta t}|&\le |\XX_n^{\epsilon,\Delta t}|+\Delta t|F(\XX_n^{\epsilon,\Delta t},\YY_{n+1}^{\epsilon,\Delta t})|\\
&\le (1+C\Delta t)|\XX_n^{\epsilon,\Delta t}|+C\Delta t(1+|\YY_{n+1}^{\epsilon,\Delta t}|).
\end{align*}
Since $\underset{\epsilon\in(0,\epsilon_0)}\sup~|x_0^\epsilon|\le C(1+|x_0|)$ owing to Assumption~\ref{ass:init}, the inequality~\eqref{eq:boundXnepsilonDeltat} is then obtained by a straightforward argument, using the inequality~\eqref{eq:boundYnepsilonDeltat} proved above. The proof of Lemma~\ref{lem:momentboundsscheme-epsilonDeltat} is thus completed.
\end{proof}

\begin{lemma}\label{lem:incrementsscheme}
For all $\kappa\in(0,1)$ and $T\in(0,\infty)$, there exists $C_\kappa(T)\in(0,\infty)$ such that for all $\epsilon\in(0,\epsilon_0)$ and $\Delta t=T/N\in(0,\Delta t_0)$, for all $n\in\{1,\ldots,N\}$, one has
\begin{equation}\label{eq:momentschemepower}
\bigl(\E[|\IL^{1-\kappa}\XX_n^{\epsilon,\Delta t}|^2]\bigr)^{\frac12}\le \frac{C_\kappa(T)}{(n\Delta t)^{1-\kappa}}(1+|x_0|).
\end{equation}
Moreover, all $\kappa\in(0,\kappa_0)$ and $T\in(0,\infty)$, there exists $C_\kappa(T)\in(0,\infty)$ such that for all $\epsilon\in(0,\epsilon_0)$ and $\Delta t=T/N\in(0,\Delta t_0)$, one has
\begin{equation}\label{eq:momentschemepowerbis}
\underset{0\le n\le N}\sup~\bigl(\E[|\IL^{\frac{\kappa}{2}}\XX_n^{\epsilon,\Delta t}|^2]\bigr)^{\frac12}\le C_\kappa(T)(1+|\IL^{\frac{\kappa}{2}}x_0|).
\end{equation}

Finally, for all $n\in\{1,\ldots,N\}$, one has
\begin{equation}\label{eq:incrementsscheme}
\bigl(\E[|\XX_{n+1}^{\epsilon,\Delta t}-\XX_n^{\epsilon,\Delta t}|^2]\bigr)^{\frac12}\le \frac{C_\kappa(T)}{(n\Delta t)^{1-\kappa}}\Delta t^{1-\kappa}(1+|x_0|)
\end{equation}
and for all $n\in\{0,\ldots,N\}$, one has
\begin{equation}\label{eq:incrementsscheme-bis}
\bigl(\E[|\IL^{-1+\kappa}\bigl(\XX_{n+1}^{\epsilon,\Delta t}-\XX_n^{\epsilon,\Delta t}\bigr)|^2]\bigr)^{\frac12}\le C_\kappa(T)\Delta t^{1-\kappa}(1+|x_0|).
\end{equation}
\end{lemma}

\begin{proof}[Proof of Lemma~\ref{lem:incrementsscheme}]
Using the inequality~\eqref{eq:smoothingnum}, the identity
\[
\XX_{n}^{\epsilon,\Delta t}=\IA_{\Delta t}^n x_0^\epsilon+\Delta t\sum_{\ell=0}^{n-1}\IA_{\Delta t}^{n-\ell}F(\XX_\ell^{\epsilon,\Delta t},\YY_\ell^{\epsilon,\Delta t}),
\]
the moment bound~\eqref{eq:momentschemepower} for $\IL^{1-\kappa}\XX_n^{\epsilon,\Delta t}$ is obtained as follows: for all $n\in\{1,\ldots,N\}$, one has
\begin{align*}
\bigl(\E[|\IL^{1-\kappa}\XX_n^{\epsilon,\Delta t}|^2]\bigr)^{\frac12}&\le \frac{C_\kappa}{(n\Delta t)^{1-\kappa}}|x_0|+\Delta t\sum_{\ell=0}^{n-1}\frac{C_\kappa}{\bigl((n-\ell)\Delta t\bigr)^{1-\kappa}}\bigl(\E[|F(\XX_\ell^{\epsilon,\Delta t},\YY_\ell^{\epsilon,\Delta t})|^2]\bigr)^{\frac12}\\
&\le \frac{C_\kappa}{(n\Delta t)^{1-\kappa}}|x_0|+\Delta t\sum_{\ell=0}^{n-1}\frac{C_\kappa}{\bigl((n-\ell)\Delta t\bigr)^{1-\kappa}} (1+|x_0|)\\
&\le \frac{C_\kappa(T)}{(n\Delta t)^{1-\kappa}}(1+|x_0|).
\end{align*}
The moment bound~\eqref{eq:momentschemepowerbis} is proved using in addition the condition $|\IL^{\frac{\kappa}{2}}x_0|\le |\IL^{\frac{\kappa_0}{2}}x_0|<\infty$ from Assumption~\ref{ass:init} for all $n\in\{0,\ldots,N\}$, one has
\begin{align*}
\bigl(\E[|\IL^{\frac{\kappa}{2}}\XX_n^{\epsilon,\Delta t}|^2]\bigr)^{\frac12}&\le |\IL^{\frac{\kappa}{2}}x_0^\epsilon|+\Delta t\sum_{\ell=0}^{n-1}\frac{C_\kappa}{\bigl((n-\ell)\Delta t\bigr)^{\frac{\kappa}{2}}}\bigl(\E[|F(\XX_\ell^{\epsilon,\Delta t},\YY_\ell^{\epsilon,\Delta t})|^2]\bigr)^{\frac12}\\
&\le |\IL^{\frac{\kappa}{2}}x_0^\epsilon|+\Delta t\sum_{\ell=0}^{n-1}\frac{C_\kappa}{\bigl((n-\ell)\Delta t\bigr)^{\frac{\kappa}{2}}} (1+|x_0|)\\
&\le C_\kappa(T)(1+|\IL^{\frac{\kappa}{2}}x_0|).
\end{align*}
Let us now prove the inequality~\eqref{eq:incrementsscheme}. Using the inequality
\[
\|\IL^{-\alpha}(\IA_{\Delta t}-I)\|_{\mathcal{L}(H)}\le C_\alpha\Delta t^\alpha,
\]
with $C_\alpha\in(0,\infty)$, and the definition~\eqref{eq:scheme} of the scheme, one has
\begin{align*}
|\XX_{n+1}^{\epsilon,\Delta t}-\XX_n^{\epsilon,\Delta t}|&=\big|(\IA_{\Delta t}-I)\XX_n^{\epsilon,\Delta t}+\Delta t\IA_{\Delta t}F(\XX_n^{\epsilon,\Delta t},\YY_{n+1}^{\epsilon,\Delta t}\big|\\
&\le |(\IA_{\Delta t}-I)\XX_n^{\epsilon,\Delta t}|+\Delta t|F(\XX_n^{\epsilon,\Delta t},\YY_{n+1}^{\epsilon,\Delta t})|\\
&\le C_\kappa\Delta t^{1-\kappa}|\IL^{1-\kappa}\XX_n^{\epsilon,\Delta t}|+\Delta t\bigl(1+|\XX_n^{\epsilon,\Delta t}|+|\YY_{n+1}^{\epsilon,\Delta t}|\bigr).
\end{align*}
Using the moment bounds~\eqref{eq:boundXnepsilonDeltat} and~\eqref{eq:boundYnepsilonDeltat} from Lemma~\ref{lem:momentboundsscheme-epsilonDeltat} and the moment bound~\eqref{eq:momentschemepower} proved above, one then obtains the inequality~\eqref{eq:incrementsscheme}.

Finally, to obtain the inequality~\eqref{eq:incrementsscheme-bis}, note that
\begin{align*}
|\IL^{-1+\kappa}\bigl(\XX_{n+1}^{\epsilon,\Delta t}-\XX_n^{\epsilon,\Delta t}\bigr)|&\le |\IL^{-1+\kappa}(\IA_{\Delta t}-I)\XX_n^{\epsilon,\Delta t}|+\Delta t|F(\XX_n^{\epsilon,\Delta t},\YY_{n+1}^{\epsilon,\Delta t})|\\
&\le C_\kappa\Delta t^{1-\kappa}|\XX_n^{\epsilon,\Delta t}|+\Delta t\bigl(1+|\XX_n^{\epsilon,\Delta t}|+|\YY_{n+1}^{\epsilon,\Delta t}|\bigr).
\end{align*}
Using the moment bounds~\eqref{eq:boundXnepsilonDeltat} and~\eqref{eq:boundYnepsilonDeltat} from Lemma~\ref{lem:momentboundsscheme-epsilonDeltat}, one then obtains the inequality~\eqref{eq:incrementsscheme-bis}.

The proof of Lemma~\ref{lem:incrementsscheme} is completed.
\end{proof}

\section{Main results}\label{sec:main}

\subsection{Asymptotic preserving property}\label{sec:AP}

Introduce the limiting scheme defined as follows: for all $\Delta t=T/N\in(0,\Delta t_0)$ and $n\in\{0,\ldots,N-1\}$, set
\begin{equation}\label{eq:limitingscheme}
\XX_{n+1}^{\Delta t}=\IA_{\Delta t}\XX_n^{\Delta t}+\Delta t\IA_{\Delta t}F(\XX_n^{\Delta t},\IL^{-\frac12}\Gamma_{n})
\end{equation}
with initial value $\XX_0^{\Delta t}=x_0$.

\begin{lemma}\label{lem:momentboundscheme}
For all $T\in(0,\infty)$, there exists $C(T)\in(0,\infty)$ such that for all $\Delta t=T/N\in(0,\Delta t_0)$, one has
\begin{equation}\label{eq:momentboundscheme}
\underset{0\le n\le N}\sup~\E[|\XX_n^{\Delta t}|^2]\le C(T)(1+|x_0|^2).
\end{equation}
\end{lemma}

\begin{proof}[Proof of Lemma~\ref{lem:momentboundscheme}]
Since $F$ is globally Lipschitz continuous (Assumption~\ref{ass:F}), for all $n\ge 0$, one has
\begin{align*}
|\XX_{n+1}^{\Delta t}|&\le |\XX_n^{\Delta t}|+\Delta t|F(\XX_n^{\Delta t},\IL^{-\frac12}\Gamma_n)|\\
&\le (1+C\Delta t)|\XX_n^{\Delta t}|+C\Delta t(1+|\IL^{-\frac12}\Gamma_n|).
\end{align*}
Using the property
\[
\E[|\IL^{-\frac12}\Gamma_n|^2]=\sum_{j\in\N}\frac{1}{\lambda_j}<\infty,
\]
the inequality~\eqref{eq:momentboundscheme} is obtained by a straightforward argument. This concludes the proof of Lemma~\ref{lem:momentboundscheme}.
\end{proof}

The fact that~\eqref{eq:limitingscheme} defines the limiting scheme associated with the scheme~\eqref{eq:scheme} when $\epsilon\to 0$ for fixed time-step size $\Delta t$ is justified by Proposition~\ref{propo:cv_scheme-limitingscheme}.

\begin{propo}\label{propo:cv_scheme-limitingscheme}
Let $\varphi:H\to\R$ be a globally Lipschitz continuous function. For all $T\in(0,\infty)$, $\Delta t\in(0,\Delta t_0)$ and $n\in\{0,\ldots,N\}$, one has
\begin{equation}\label{eq:cvlimitingscheme}
\underset{\epsilon\to 0}\lim~\E[\varphi(\XX_{n}^{\epsilon,\Delta t})]=\E[\varphi(\XX_{n}^{\Delta t})].
\end{equation}
\end{propo}

In addition, the limiting scheme~\eqref{eq:limitingscheme} is consistent with the limiting evolution equation~\eqref{eq:averaged}, as justified by Proposition~\ref{propo:error_limitingscheme-averagedequation} below.
\begin{propo}\label{propo:error_limitingscheme-averagedequation}
For all $\kappa\in(0,1)$ and $T\in(0,\infty)$, there exists $C_\kappa(T)\in(0,\infty)$ such that for any function $\varphi:H\to\R$ of class $\mathcal{C}^2$ with bounded first and second order derivatives, for all $\Delta t\in(0,\Delta t_0)$, one has
\begin{equation}\label{eq:error_limitingscheme-averagedequation}
\big|\E[\varphi(\XX_N^{\Delta t})]-\varphi(\overline{\XX}(T))\big|\le C_\kappa(T)\vvvert\varphi\vvvert_2\Delta t^{1-\kappa}(1+|x_0|^2).
\end{equation}
\end{propo}

Combining Propositions~\ref{propo:cv_scheme-limitingscheme} and~\ref{propo:error_limitingscheme-averagedequation} shows that the scheme~\eqref{eq:scheme} is asymptotic preserving.

The proofs of Propositions~\ref{propo:cv_scheme-limitingscheme} and~\ref{propo:error_limitingscheme-averagedequation} are postponed to Section~\ref{sec:proof-AP1} and Section~\ref{sec:proof-AP2} respectively. In fact, those two results are reformulations of~\cite[Theorem~91]{B}, and the proofs are given to make the presentation of the analysis of the scheme~\eqref{eq:scheme} self-contained. In addition, Proposition~\ref{propo:error_limitingscheme-averagedequation} is employed in the proof of the main result of this article.

\subsection{Uniform weak error estimates}\label{sec:UA}

The main result of this article is Theorem~\ref{theo:UA}, which gives uniform weak error estimates for the numerical scheme.
\begin{theo}\label{theo:UA}
For all $\kappa\in(0,\kappa_0)$ and $T\in(0,\infty)$, there exists $C_\kappa(T)\in(0,\infty)$ such that for any function $\varphi:H\to\R$ of class $\mathcal{C}^3$ with bounded derivatives of order $1,2,3$, for all $\Delta t=T/N\in(0,\Delta t_0)$ and $\epsilon\in(0,\epsilon_0)$, one has
\begin{equation}\label{eq:error_UA}
\big|\E[\varphi(\XX_N^{\epsilon,\Delta t})]-\E[\varphi(\XX^\epsilon(T))]\big|\le C_\kappa(T)\Delta t^{\frac13-\kappa}\vvvert\varphi\vvvert_3\bigl(1+|\IL^{\frac{\kappa_0}{2}}x_0|^2\bigr).
\end{equation}
\end{theo}
The proof of Theorem~\ref{theo:UA} is given in Section~\ref{sec:proofUA} below, as a consequence of several auxiliary error estimates stated in Section~\ref{sec:auxiliary-error} below, which are proved in Section~\ref{sec:proof-UA} using non-trivial arguments and lengthy computations.

\subsection{Auxiliary error estimates}\label{sec:auxiliary-error}

The proof of Theorem~\ref{theo:UA} is based on using several auxiliary error estimates.

Let us first introduce the following auxiliary scheme: for all $\Delta t=T/N\in(0,\Delta t_0)$, $x\in H$ and $n\in\{0,\ldots,N-1\}$, set
\begin{equation}\label{eq:auxiliaryscheme}
\overline{\XX}_{n+1}^{\Delta t;x}=\IA_{\Delta t}\overline{\XX}_n^{\Delta t;x}+\Delta t\IA_{\Delta t}\overline{F}(\overline{\XX}_n^{\Delta t;x})
\end{equation}
with initial value $\overline{\XX}_0^{\Delta t;x}=x\in H$. The scheme~\eqref{eq:auxiliaryscheme} is the standard linear implicit Euler scheme applied to the limiting evolution equation~\eqref{eq:averaged}. One has the following convergence result.
\begin{propo}\label{propo:error_auxiliaryscheme}
For all $\kappa\in(0,\kappa_0)$ and $T\in(0,\infty)$, there exists $C_\kappa(T)\in(0,\infty)$ such that for all $\Delta t\in(0,\Delta t_0)$ one has
\begin{equation}\label{eq:error_auxiliaryscheme}
\big|\overline{\XX}_n^{\Delta t;x_0}-\overline{\XX}(n\Delta t)\big|\le C_\kappa(T)\Delta t^{1-\kappa}(1+\frac{1}{(n\Delta t)^{1-\kappa}}|x_0|).
\end{equation}
\end{propo}
Even if Proposition~\ref{propo:error_auxiliaryscheme} is a standard result in the numerical analysis of parabolic evolution equations, its proof is given in Section~\ref{sec:proof-UA1} for completeness. Note that the initial value $x_0$ is only assumed to satisfy $x_0\in H$ in this statement.

Proposition~\ref{propo:error_fixedepsilon} provides a weak error estimate where the right-hand side is allowed to depend on $\epsilon$. This result provides the consistency of the scheme~\eqref{eq:scheme} for the approximation of $\XX^\epsilon(T)$ for any value of $\epsilon\in(0,\epsilon_0)$. The order of convergence with respect to $\Delta t$ is equal to $1/2$.
\begin{propo}\label{propo:error_fixedepsilon}
For all $\kappa\in(0,\kappa_0)$ and $T\in(0,\infty)$, there exists $C_\kappa(T)\in(0,\infty)$ such that for any function $\varphi:H\to\R$ of class $\mathcal{C}^3$ with bounded derivatives of order $1,2,3$, for all $\Delta t=T/N\in(0,\Delta t_0)$ and $\epsilon\in(0,\epsilon_0)$, one has
\begin{equation}\label{eq:error_fixedepsilon}
\big|\E[\varphi(\XX_N^{\epsilon,\Delta t})]-\E[\varphi(\XX^\epsilon(T))]\big|\le C_\kappa(T)\Bigl(\bigl(\frac{\Delta t}{\epsilon}\bigr)^{\frac12-\kappa}+\frac{\Delta t}{\epsilon}\Bigr)\vvvert\varphi\vvvert_3\bigl(1+|\IL^{\frac{\kappa}{2}}x_0|^2\bigr).
\end{equation}
\end{propo}
The proof of Proposition~\ref{propo:error_fixedepsilon} is the most delicate part of the analysis in this article.

Finally, Proposition~\ref{propo:error_scheme-limitingscheme} is a variant of Proposition~\ref{propo:averaging} in discrete-time, and is related to Proposition~\ref{propo:cv_scheme-limitingscheme} above.
\begin{propo}\label{propo:error_scheme-limitingscheme}
For all $\kappa\in(0,\kappa_0)$ and $T\in(0,\infty)$, there exists $C_\kappa(T)\in(0,\infty)$ such that for any function $\varphi:H\to\R$ of class $\mathcal{C}^2$ with bounded first and second order derivatives, for all $\Delta t=T/N\in(0,\Delta t_0)$ and $\epsilon\in(0,\epsilon_0)$, one has
\begin{equation}\label{eq:error_scheme-limitingscheme}
\big|\E[\varphi(\XX_N^{\epsilon,\Delta t})]-\E[\varphi(\XX_N^{\Delta t})]\big|\le C_\kappa(T)\bigl(\frac{\epsilon}{\Delta t^\kappa}+\Delta t^{1-\kappa}\bigr)\vvvert\varphi\vvvert_2(1+|x_0|^2).
\end{equation}
\end{propo}
See~\cite[Lemma~5.4]{BR} for a similar statement in the finite dimensional SDE case. Note that the right-hand side of~\eqref{eq:error_scheme-limitingscheme} goes to infinity when $\Delta t\to 0$, but the upper bound is sufficient for the proof of Theorem~\ref{theo:UA}. Having $\epsilon$ instead of $\frac{\epsilon}{\Delta t^\kappa}$ would not change the result. The presence of $\Delta t^{\kappa}$, with arbitrarily small $\kappa\in(0,\kappa_0)$ is due to arguments on the analysis of parabolic semilinear evolution equations.

\subsection{Proof of Theorem~\ref{theo:UA}}\label{sec:proofUA}

The proof of Theorem~\ref{theo:UA} is a straightforward consequence of auxiliary weak error estimates which have been stated above. Let us first obtain the weak error estimate~\eqref{eq:error_averaging} as a straightforward consequence of the results stated above.
\begin{proof}[Proof of the inequality~\eqref{eq:error_averaging}]
The weak error in the right-hand side of~\eqref{eq:error_averaging} can be decomposed as
\begin{align*}
\big|\E[\varphi(\XX^\epsilon(T))]-\E[\varphi(\overline{\XX}(T))]\big|&\le \big|\E[\varphi(\XX^\epsilon(T))]-\E[\varphi(\XX_N^{\epsilon,\Delta t})]\big|\\
&+\big|\E[\varphi(\XX_N^{\epsilon,\Delta t})]-\E[\varphi(\XX_N^{\Delta t})]\big|\\
&+\big|\E[\varphi(\XX_N^{\Delta t})]-\E[\varphi(\overline{\XX}(T))]\big|,
\end{align*}
where the value of $\Delta t=T/N$ in the right-hand side of the inequality above is arbitrary. Since the value of the left-hand side is independent of $\Delta t$, choosing $N=\epsilon^{-2}+1$ and using the inequalities~\eqref{eq:error_fixedepsilon},~\eqref{eq:error_scheme-limitingscheme} and~\eqref{eq:error_limitingscheme-averagedequation} from Propositions~\ref{propo:error_fixedepsilon},~\ref{propo:error_scheme-limitingscheme} and~\ref{propo:error_limitingscheme-averagedequation} respectively gives the inequality~\eqref{eq:error_averaging}.
\end{proof}

\begin{proof}[Proof of Theorem~\ref{theo:UA}]
The weak error $\E[\varphi(\XX_N^{\epsilon,\Delta t})]-\E[\varphi(\XX^\epsilon(T))]$ can be treated using two different strategies.

On the one hand, one has the inequality~\eqref{eq:error_fixedepsilon} from Proposition~\ref{propo:error_fixedepsilon}:
\[
\big|\E[\varphi(\XX_N^{\epsilon,\Delta t})]-\E[\varphi(\XX^\epsilon(T))]\big|\le C_\kappa(T)\Bigl(\bigl(\frac{\Delta t}{\epsilon}\bigr)^{\frac12-\kappa}+\frac{\Delta t}{\epsilon}\Bigr)\vvvert\varphi\vvvert_3\bigl(1+|\IL^{\frac{\kappa}{2}}x_0|^2\bigr).
\]
On the other hand, one has
\begin{align*}
\big|\E[\varphi(\XX_N^{\epsilon,\Delta t})]-\E[\varphi(\XX^\epsilon(T))]\big|&\le \big|\E[\varphi(\XX_N^{\epsilon,\Delta t})]-\E[\varphi(\XX_N^{\Delta t})]\big|\\
&+\big|\E[\varphi(\XX_N^{\Delta t})]-\varphi(\overline{\XX}(T))\big|\\
&+\big|\varphi(\overline{\XX}(T))-\E[\varphi(\XX^\epsilon(T))]\big|\\
&\le C_\kappa(T)\bigl(\frac{\epsilon}{\Delta t^\kappa}+\Delta t^{1-\kappa}\bigr)\vvvert\varphi\vvvert_2(1+|x_0|^2)\\
&+C_\kappa(T)\vvvert\varphi\vvvert_2\Delta t^{1-\kappa}(1+|x_0|^2)\\
&+C_\kappa(T)\vvvert\varphi\vvvert_3 \epsilon^{1-\kappa}\bigl(1+|\IL^{\frac{\kappa}{2}}x_0|^2\bigr),
\end{align*}
using the inequalities~\eqref{eq:error_scheme-limitingscheme},~\eqref{eq:error_limitingscheme-averagedequation} and~\eqref{eq:error_averaging} from Propositions~\ref{propo:error_scheme-limitingscheme},~\ref{propo:error_limitingscheme-averagedequation} and~\ref{propo:averaging} respectively.

Using the first inequality when $\Delta t^{\frac13}\le \epsilon$ and the second inequality when $\epsilon\le \Delta t^{\frac13}$, one then obtains the inequality~\eqref{eq:error_UA}. Since the parameter $\kappa\in(0,\kappa_0)$ is arbitrarily small, the proof of Theorem~\ref{theo:UA} is thus completed.
\end{proof}

\begin{rem}
If the fast component $\YY^\epsilon$ of the SPDE system~\eqref{eq:SPDE} is discretized using the accelerated exponential Euler scheme, one obtains the scheme
\begin{equation}\label{eq:scheme-expo}
\left\lbrace
\begin{aligned}
\XX_{n+1}^{\epsilon,\Delta t}&=\IA_{\Delta t}\bigl(\XX_n^{\epsilon,\Delta t}+\Delta t F(\XX_n^{\epsilon,\Delta t},\YY_{n+1}^{\epsilon,\Delta t})\bigr)\\
\YY_{n+1}^{\epsilon,\Delta t}&=e^{-\frac{\Delta t}{\epsilon}\IL}\YY_n^{\epsilon,\Delta t}+\sqrt{\frac{2}{\epsilon}}\int_{t_n}^{t_{n+1}}e^{-\frac{t_{n+1}-t}{\epsilon}}dW(s),
\end{aligned}
\right.
\end{equation}
with initial values $\XX_0^{\epsilon,\Delta t}=x_0^\epsilon$ and $\YY_0^{\epsilon,\Delta t}=y_0^\epsilon$.

The result of Theorem~\ref{theo:UA} is valid also for the scheme~\eqref{eq:scheme-expo}. In fact, the proof of Proposition~\ref{propo:error_fixedepsilon} would be simpler for that scheme: for instance the error terms $e_n^{1,\epsilon,\Delta t}$ and $e_n^{2,\epsilon,\Delta t}$ defined by~\eqref{eq:en1} and~\eqref{eq:en2} below would vanish. We thus focus only on the analysis of the scheme~\eqref{eq:scheme}.

Note that the scheme~\eqref{eq:scheme-expo} can be applied only if the eigenvalues $\lambda_j$ and eigenfunctions $e_j$ of the linear operator $\IL$ (see Assumption~\ref{ass:F}) are known (in which case it is appropriate to use a spectral Galerkin discretization in space). On the contrary, the scheme~\eqref{eq:scheme}, based on the modified Euler scheme introduced in~\cite{B}, can be applied without this knowledge, and it is appropriate to combine it with a finite difference discretization in space.
\end{rem}

It thus remains to establish all the auxiliary results used in the proof of Theorem~\ref{theo:UA} above.

\section{Regularity estimates for solutions of Kolmogorov equations}\label{sec:Kolmogorov}

Let $\varphi:H\to\R$ be a continuous mappping. The weak error analysis requires to study the regularity and growth properties of the auxiliary mappings $(t,x,y)\in[0,T]\times \HH\mapsto u^\epsilon(t,x,y)$ and $(n,x)\in\{0,\ldots,N\}\times H\mapsto \overline{u}_n^{\Delta t}(x)$ defined by
\begin{align}
u^\epsilon(t,x,y)&=\E_{x,y}[\varphi(\XX^\epsilon(t))],\label{eq:uepsilon}\\
\overline{u}_n^{\Delta t}(x)&=\varphi(\overline{\XX}_n^{\Delta t;x}),\label{eq:ubar}
\end{align}
where $\bigl(\XX^\epsilon(t),\YY^\epsilon(t)\bigr)_{t\in[0,T]}$ is the mild solution of~\eqref{eq:SPDE} with initial values $\XX^\epsilon(0)=x$ and $\YY^\epsilon(0)=y$ (this is the meaning of the notation $\E_{x,y}[\cdot]$ in~\eqref{eq:uepsilon}), and where $\bigl(\overline{\XX}_n^{\Delta t;x}\bigr)_{n=0,\ldots,N}$ is the solution of~\eqref{eq:auxiliaryscheme}.

The function $u^\epsilon$ is solution of the Kolmogorov equation
\begin{equation}\label{eq:Kolmogorov}
\begin{aligned}
\partial_tu^\epsilon(t,x,y)&=\langle D_xu^\epsilon(t,x,y),-\IL x+F(x,y)\rangle\\
&+\frac{1}{\epsilon}\Bigl(-\langle D_yu^\epsilon(t,x,y),\IL y\rangle+\sum_{j\in\N}D_y^2u^\epsilon(t,x,y).(e_j,e_j)\Bigr),
\end{aligned}
\end{equation}
with initial value $u^\epsilon(0,x,y)=\varphi(x)$. We refer to the monograph~\cite{Cerrai} for results on infinite dimensional Kolmogorov equations. In this section, it would be convenient to introduce a spectral Galerkin approximation procedure to justify all the computations. This is a standard tool, and to simplify the notation this is omitted in the sequel. All the upper bounds are understood to hold uniformly with respect to the auxiliary approximation parameter.

Let us first state regularity results for the mapping $u^\epsilon$.
\begin{lemma}\label{lem:uepsilon}
For all $T\in(0,\infty)$ and $\kappa\in(0,1]$, $\alpha\in[0,1)$, $\alpha_1,\alpha_2\in[0,1)$ such that $\alpha_1+\alpha_2<1$, there exist $C_\kappa(T),C_\alpha(T),C_{\alpha_1,\alpha_2}(T)\in(0,\infty)$ such that for all $\epsilon\in(0,\epsilon_0)$ and all $\varphi:H\to\R$ of class $\mathcal{C}^3$ with bounded derivatives of order $1,2,3$, one has the following inequalities.
\begin{enumerate}
\item For all $t\in(0,T]$, $x,y\in H$ and $h\in H$, one has
\begin{equation}\label{eq:lemuepsilon-1}
|\langle D_xu^\epsilon(t,x,y),h\rangle|+\frac{1}{\epsilon}|\langle D_yu^\epsilon(t,x,y),h\rangle|\le \frac{C_\alpha(T)}{t^{\alpha}}\vvvert\varphi\vvvert_1|\IL^{-\alpha}h|.
\end{equation}
\item For all $t\in(0,T]$, $x,y\in H$ and $h^1,h^2\in H$, one has
\begin{equation}\label{eq:lemuepsilon-2}
\begin{aligned}
|D_x^2u^\epsilon(t,x,y).(h^1,h^2)|&+\frac{1}{\epsilon^{1-\alpha_1}}|D_xD_yu^\epsilon(t,x,y).(h^1,h^2)|+\frac{1}{\epsilon}|D_y^2u^\epsilon(t,x,y).(h^1,h^2)|\\
&\le \frac{C_{\alpha_1,\alpha_2}(T)}{t^{\alpha_1+\alpha_2}}\vvvert\varphi\vvvert_2\|\IL^{-\alpha_1}h^1||\IL^{-\alpha_2}h^2|.
\end{aligned}
\end{equation}
\item For all $t\in(0,T]$, $x,y\in H^\kappa$ and $h\in H^\kappa$, one has
\begin{equation}\label{eq:lemuepsilon-3}
|\partial_t\langle D_xu^\epsilon(t,x,y),h\rangle|\le \frac{1}{\epsilon}\frac{C_\kappa(T)}{t^{1-\kappa}}\vvvert\varphi\vvvert_3(1+|\IL^\kappa x|+|\IL^\kappa y|)|\IL^\kappa h|.
\end{equation}
\end{enumerate}
\end{lemma}
Lemma~\ref{lem:uepsilon} is a variant of~\cite[Lemma~5.5]{BR} (SDE case), with a more precise analysis of the dependence with respect to the parameter $\epsilon$ of the derivatives with respect to the variable $y$. In addition, in order to obtain the optimal weak order of convergence with respect to $\Delta t$ (with fixed $\epsilon$) in Proposition~\ref{propo:error_fixedepsilon}, one needs to choose $\alpha,\alpha_1,\alpha_2>0$. The bounds of type~\eqref{eq:lemuepsilon-1} and~\eqref{eq:lemuepsilon-2} are specific to the parabolic semilinear evolution equations setting, and are related to the smoothing inequality~\eqref{eq:smoothing}. We refer for instance~\cite{BrehierDebussche} and~\cite{Debussche:11} for similar results (with fixed $\epsilon$) and their use to prove weak error estimates.

Let us now provide regularity results for the mappings $\overline{u}_n^{\Delta t}$ defined by~\eqref{eq:ubar}.
\begin{lemma}\label{lem:ubarDeltat}
For all $T\in(0,\infty)$ and $\kappa\in(0,1]$, there exists $C_\kappa(T)\in(0,\infty)$ such that for all $\Delta t\in(0,\Delta t_0)$, all $x,y\in H$, all $h,k\in H$, all $n\in\{1,\ldots,N\}$, and all $\varphi:H\to\R$ of class $\mathcal{C}^2$ with bounded derivatives of order $1,2$, one has
\begin{align}
|\langle D\overline{u}_n^{\Delta t}(x),h\rangle|&\le \frac{C_\kappa(T)}{(n\Delta t)^{1-\kappa}}\vvvert\varphi\vvvert_1 \bigl(\Delta t|h|+|\IL^{-1+\kappa}h|\bigr)\label{eq:lemubarDeltat_1}\\
|D^2\overline{u}_n^{\Delta t}(x).(h,k)|&\le \frac{C_\kappa(T)}{(n\Delta t)^{1-\kappa}}\vvvert\varphi\vvvert_2 \bigl(\Delta t|h|+|\IL^{-1+\kappa}h|\bigr)|k|\label{eq:lemubarDeltat_2}\\
|\langle D\overline{u}_{n+1}^{\Delta t}(x)-D\overline{u}_{n}^{\Delta t}(x),h\rangle|&\le \frac{C_\kappa(T)\Delta t^{1-\kappa}}{(n\Delta t)^{1-\kappa}}\vvvert\varphi\vvvert_2 (1+|x|)|h|.\label{eq:lemubarDeltat_3}
\end{align}
Note that with $\kappa=1$, the inequalities~\eqref{eq:lemubarDeltat_1} and~\eqref{eq:lemubarDeltat_2} provide the following result:
\begin{equation}\label{eq:lemubarDeltat_4}
\underset{0\le \Delta t\le \Delta t_0}\sup~\underset{0\le n\le N}\sup~\vvvert \overline{u}_n\vvvert_2\le C(T)\vvvert\varphi\vvvert_2.
\end{equation}
\end{lemma}
Lemma~\ref{lem:ubarDeltat} is a variant of~\cite[Lemma~5.7]{BR} (SDE case), where like in Lemma~\ref{lem:uepsilon} one needs $1-\kappa\neq 0$. The proof employs the discrete time version~\eqref{eq:smoothingnum} of the smoothing inequality~\eqref{eq:smoothing}. See also~\cite[Lemma~7.2]{B:2013} for a variant of Lemma~\ref{lem:ubarDeltat} (analysis of HMM schemes in the SPDE case).

The proof of Lemma~\ref{lem:uepsilon} is given in Section~\ref{sec:proof-uepsilon}, whereas the proof of Lemma~\ref{lem:ubarDeltat} is given in Section~\ref{sec:proof-ubarDeltat}.

\subsection{Proof of Lemma~\ref{lem:uepsilon}}\label{sec:proof-uepsilon}

Recall the notation $\HH=H\times H$. For all ${\bf h}=(h_x,h_y)\in \HH$, one has the following expression for the first-order derivatives:
\begin{align*}
Du^\epsilon(t,x,y).{\bf h}&=D_xu^\epsilon(t,x,y).h_x+D_yu^\epsilon(t,x,y).h_y\\
&=\E_{x,y}[D\varphi(\XX^\epsilon(t,x,y)).\eta_x^{\epsilon,{\bf h}}(t)]
\end{align*}
where $t\in[0,T]\mapsto \eta^{{\bf h}}(t)=(\eta_x^{\epsilon,{\bf h}}(t),\eta_y^{\epsilon,{\bf h}}(t))\in \HH$ is solution of
\[
\left\lbrace
\begin{aligned}
\frac{d\eta_x^{\epsilon,{\bf h}}(t)}{dt}&=-\IL\eta_x^{\epsilon,{\bf h}}(t)+D_xF(\XX^\epsilon(t),\YY^\epsilon(t)).\eta_x^{\epsilon,{\bf h}}(t)+D_yF(\XX^\epsilon(t),\YY^\epsilon(t)).\eta_y^{\epsilon,{\bf h}}(t)\\
\frac{d\eta_y^{\epsilon,{\bf h}}(t)}{dt}&=-\frac{1}{\epsilon}\IL\eta_y^{\epsilon,{\bf h}}(t),
\end{aligned}
\right.
\]
with initial values $\eta_x^{\epsilon,{\bf h}}(0)=h_x$ and $\eta_y^{\epsilon,{\bf h}}(0)=h_y$.

For all ${\bf h}^1=(h_x^1,h_y^1)\in\HH$ and ${\bf h}^2=(h_x^2,h_y^2)\in\HH$, one has the following expression for the second-order derivatives:
\begin{align*}
D^2u^\epsilon(t,x,y).({\bf h}^1,{\bf h}^2)&=D_x^2u^\epsilon(t,x,y).(h_x^1,h_x^2)+D_y^2u^\epsilon(t,x,y).(h_y^1,h_y^2)\\
&+D_xD_yu^\epsilon(t,x,y).(h_x^1,h_y^2)+D_yD_xu^\epsilon(t,x,y).(h_y^1,h_x^2)\\
&=\E_{x,y}[D\varphi(\XX^\epsilon(t)).\zeta_x^{\epsilon,{\bf h}^1,{\bf h}^2}(t)]+\E_{x,y}[D^2\varphi(\XX^\epsilon(t)).(\eta_x^{\epsilon,{\bf h}^1}(t),\eta_x^{\epsilon,{\bf h}^2}(t))],
\end{align*}
where $t\in[0,T]\mapsto \zeta^{\epsilon,{\bf h}^1,{\bf h}^2}(t)=(\zeta_x^{\epsilon,{\bf h}^1,{\bf h}^2}(t),\zeta_y^{\epsilon,{\bf h}^1,{\bf h}^2}(t))\in\HH$ is solution of
\[
\left\lbrace
\begin{aligned}
\frac{d\zeta_x^{\epsilon,{\bf h}^1,{\bf h}^2}(t)}{dt}&=-\IL\zeta_x^{\epsilon,{\bf h}^1,{\bf h}^2}(t)+D_xF(\XX^{\epsilon}(t),\YY^{\epsilon}(t)).\zeta_x^{\epsilon,{\bf h}^1,{\bf h}^2}(t)+D_yF(\XX^\epsilon(t),\YY^\epsilon(t)).\zeta_y^{\epsilon,{\bf h}^1,{\bf h}^2}(t)\\
&+D^2F(\XX^\epsilon(t),\YY^\epsilon(t)).(\eta^{\epsilon,{\bf h}^1}(t),\eta^{\epsilon,{\bf h}^2}(t))\\
\frac{d\zeta_y^{\epsilon,{\bf h}^1,{\bf h}^2}(t)}{dt}&=0,
\end{aligned}
\right.
\]
with initial values $\zeta_x^{\epsilon,{\bf h}^1,{\bf h}^2}(0)=\zeta_y^{\epsilon,{\bf h}^1,{\bf h}^2}(0)=0$. In the expressions above, the fact that the initial value $=u^\epsilon(0,x,y)=\varphi(x)$ is independent of $y$ is used.

\begin{proof}[Proof of the inequality~\eqref{eq:lemuepsilon-1}]
Let $\alpha\in[0,1)$.

Observe that for all $t\ge 0$, one has $\eta_y^{{\bf h}}(t)=e^{-\frac{t}{\epsilon}\IL}h_y$. As a consequence, using the semigroup property and the smoothing inequality~\eqref{eq:smoothing}, for all $t\in(0,\infty)$, one obtains
\begin{equation}\label{eq:eta-y_bound}
|\eta_y^{\epsilon,{\bf h}}(t)|\le C_\alpha \frac{\epsilon^{\alpha}}{t^{\alpha}}|e^{-\frac{t}{2\epsilon}\IL}h_y|\le C_\alpha e^{-\frac{\lambda_1 t}{2\epsilon}}\frac{\epsilon^{\alpha}}{t^{\alpha}}|\IL^{-\alpha}h_y|.
\end{equation}

Introduce an auxiliary process defined by $\tilde{\eta}_x^{\epsilon,{\bf h}}(t)=\eta_x^{\epsilon,{\bf h}}(t)-e^{-t\IL}h_x$ for all $t\ge 0$. Using the mild formulation
\begin{align*}
\eta_x^{\epsilon,{\bf h}}(t)=e^{-t\IL}h_x&+\int_0^te^{-(t-s)\IL}D_xF(\XX^\epsilon(s),\YY^\epsilon(s)).{\eta}_x^{\epsilon,{\bf h}}(s)ds\\
&+\int_{0}^{t}e^{-(t-s)\IL}D_yF(\XX^\epsilon(s),\YY^\epsilon(s)).{\eta}_y^{\epsilon,{\bf h}}(s)ds,
\end{align*}
one obtains, for all $t\in[0,T]$,
\begin{align*}
\tilde{\eta}_x^{\epsilon,{\bf h}}(t)&=\int_{0}^{t}e^{-(t-s)\IL}D_xF(\XX^\epsilon(s),\YY^\epsilon(s)).\tilde{\eta}_x^{\epsilon,{\bf h}}(s)ds\\
&+\int_{0}^{t}e^{-(t-s)\IL}D_xF(\XX^\epsilon(s),\YY^\epsilon(s)).e^{-s\IL}h_x ds\\
&+\int_{0}^{t}e^{-(t-s)\IL}D_yF(\XX^\epsilon(s),\YY^\epsilon(s)).e^{-\frac{s}{\epsilon}\IL}h_y ds.
\end{align*}
Since the mappings $D_xF$ and $D_yF$ are bounded (Assumption~\ref{ass:F}), using the smoothing inequality~\eqref{eq:smoothing} and the bound above, one then has
\begin{align*}
|\tilde{\eta}_x^{\epsilon,{\bf h}}(t)|&\le C\int_{0}^{t}|\tilde{\eta}_x^{\epsilon,{\bf h}}(s)|ds+C\int_{0}^{t}|e^{-s\IL}h_x|ds+C\int_{0}^{t}|e^{-\frac{s}{\epsilon}\IL}h_y|ds\\
&\le C\int_{0}^{t}|\tilde{\eta}_x^{{\bf h}}(s)|ds+C_\alpha\int_0^{t}s^{-\alpha}ds |\IL^{-\alpha}h_x|+C_\alpha\int_{0}^{t}e^{-\frac{\lambda_1 s}{2\epsilon}}\frac{\epsilon^{\alpha}}{s^{\alpha}}|\IL^{-\alpha}h_y|ds\\
&\le C\int_{0}^{t}|\tilde{\eta}_x^{{\bf h}}(s)|ds+C_\alpha(T)\bigl(|\IL^{-\alpha}h_x|+\epsilon|\IL^{-\alpha}h_y|\bigr)
\end{align*}
with $C_\alpha(T)=\int_0^{T}s^{-\alpha}ds+\int_{0}^{\infty}e^{-\frac{\lambda_1 s}{2}}s^{-\alpha}ds<\infty$, by a straightforward change of variables argument in the integral.

Applying Gronwall's inequality, one then obtains
\[
\underset{0\le t\le T}\sup~|\tilde{\eta}_x^{\epsilon,{\bf h}}(t)|\le C_\alpha(T)\bigl(|\IL^{-\alpha}h_x|+\epsilon|\IL^{-\alpha}h_y|\bigr),
\]
with $C_\alpha(T)\in(0,\infty)$, independent of $\epsilon\in(0,\epsilon_0)$. Therefore, for all $t\in(0,T]$, one obtains the inequality
\begin{equation}\label{eq:eta-x_bound}
|\eta_x^{\epsilon,{\bf h}}(t)|\le C_\alpha(T)\bigl(\frac{1}{t^{\alpha}}|\IL^{-\alpha}h_x|+\epsilon|\IL^{-\alpha}h_y|\bigr).
\end{equation}
Since $D\varphi$ is bounded, one finally obtains the inequality
\[
|Du^\epsilon(t,x,y).{\bf h}|\le C_\kappa(T)\vvvert\varphi\vvvert_1\bigl(\frac{1}{t^{\alpha}}|\IL^{-\alpha}h_x|+\epsilon|\IL^{-\alpha}h_y|\bigr),
\]
for all $t\in(0,T]$. Considering the cases ${\bf h}=(h_x,h_y)=(h,0)$ and ${\bf h}=(h_x,h_y)=(0,h)$ then concludes the proof of the inequality~\eqref{eq:lemuepsilon-1}.
\end{proof}

\begin{proof}[Proof of the inequality~\eqref{eq:lemuepsilon-2}]
Let $\alpha_1,\alpha_2\in[0,1)$ be such that $\alpha_1+\alpha_2<1$.

Observe that $\zeta_y^{{\bf h}^1,{\bf h}^2}(t)=0$ for all $t\ge 0$, and that, using a mild formulation, one has, for all $t\ge 0$,
\begin{align*}
\zeta_x^{\epsilon,{\bf h}^1,{\bf h}^2}(t)&=\int_{0}^{t}e^{-(t-s)\IL}D_xF(\XX^{\epsilon}(s),\YY^{\epsilon}(s)).\zeta_x^{\epsilon,{\bf h}^1,{\bf h}^2}(s)ds\\
&+\int_{0}^{t}e^{-(t-s)\IL}D^2F(\XX^\epsilon(s),\YY^\epsilon(s)).(\eta^{\epsilon,{\bf h}^1}(s),\eta^{\epsilon,{\bf h}^2}(s))ds.
\end{align*}
Using the inequalities~\eqref{eq:eta-x_bound} and~\eqref{eq:eta-y_bound}, one then obtains
\begin{align*}
|\zeta_x^{\epsilon,{\bf h}^1,{\bf h}^2}(t)|&\le C\int_{0}^{t}|\zeta_x^{\epsilon,{\bf h}^1,{\bf h}^2}(s)|ds+C\int_{0}^{t}|\eta^{\epsilon,{\bf h}^1}(s)| |\eta^{\epsilon,{\bf h}^2}(s)|ds\\
&\le C\int_{0}^{t}|\zeta_x^{\epsilon,{\bf h}^1,{\bf h}^2}(s)|ds\\
&+C_{\alpha_1,\alpha_2}(T)\int_0^t \Bigl(s^{-\alpha_1}|\IL^{-\alpha_1}h_x^1|+\bigl(\epsilon+\frac{\epsilon^{\alpha_1}}{s^{\alpha_1}}e^{-\frac{\lambda_1 s}{2\epsilon}}\bigr)|\IL^{-\alpha_1}h_y^1|\Bigr)\\
&\hspace{3cm}\Bigl(s^{-\alpha_2}|\IL^{-\alpha_2}h_x^2|+\bigl(\epsilon+\frac{\epsilon^{\alpha_2}}{s^{\alpha_2}}e^{-\frac{\lambda_1 s}{2\epsilon}}\bigr)|\IL^{-\alpha_2}h_y^2|\Bigr)ds\\
&\le C\int_{0}^{t}|\zeta_x^{\epsilon,{\bf h}^1,{\bf h}^2}(s)|ds\\
&+C_{\alpha_1,\alpha_2}(T)\Bigl(|\IL^{-\alpha_1}h_x^1||\IL^{-\alpha_2}h_x^2|+\epsilon|\IL^{-\alpha_1}h_y^1||\IL^{-\alpha_2}h_y^2|\Bigr)\\
&+C_{\alpha_1,\alpha_2}(T)\Bigl(\epsilon^{1-\alpha_1}|\IL^{-\alpha_1}h_x^1||\IL^{-\alpha_2}h_y^2|+\epsilon^{1-\alpha_2}|\IL^{-\alpha_1}h_y^1||\IL^{-\alpha_2}h_x^2|\Bigr),
\end{align*}
where $C_{\alpha_1,\alpha_2}(T)\in(0,\infty)$ is independent of $\epsilon\in(0,\epsilon_0)$, using change of variables arguments in the integrals, like in the proof of the inequality~\eqref{eq:lemuepsilon-1} above.

Applying Gronwall's lemma then yields the inequality
\begin{align*}
\underset{0\le t\le T}\sup~|\zeta_x^{\epsilon,{\bf h}^1,{\bf h}^2}(t)|&\le C_{\alpha_1,\alpha_2}(T)\Bigl(|\IL^{-\alpha_1}h_x^1||\IL^{-\alpha_2}h_x^2|+\epsilon|\IL^{-\alpha_1}h_y^1||\IL^{-\alpha_2}h_y^2|\Bigr)\\
&+C_{\alpha_1,\alpha_2}(T)\Bigl(\epsilon^{1-\alpha_1}|\IL^{-\alpha_1}h_x^1||\IL^{-\alpha_2}h_y^2|+\epsilon^{1-\alpha_2}|\IL^{-\alpha_1}h_y^1||\IL^{-\alpha_2}h_x^2|\Bigr),
\end{align*}
for all $t\in[0,T]$. Using that inequality and~\eqref{eq:eta-x_bound}, one then obtains
\begin{align*}
|D^2u^\epsilon&(t,x,y).({\bf h}^1,{\bf h}^2)|\le \vvvert\varphi\vvvert_1|\zeta_x^{\epsilon,{\bf h}^1,{\bf h}^2}(t)|+\vvvert\varphi\vvvert_2|\eta_x^{\epsilon,{\bf h}^1}(t)||\eta_x^{\epsilon,{\bf h}^2}(t)|\\
&\le C_{\alpha_1,\alpha_2}(T)\vvvert\varphi\vvvert_1\Bigl(|\IL^{-\alpha_1}h_x^1||\IL^{-\alpha_2}h_x^2|+\epsilon|\IL^{-\alpha_1}h_y^1||\IL^{-\alpha_2}h_y^2|\Bigr)\\
&+C_{\alpha_1,\alpha_2}(T)\vvvert\varphi\vvvert_1\Bigl(\epsilon^{1-\alpha_1}|\IL^{-\alpha_1}h_x^1||\IL^{-\alpha_2}h_y^2|+\epsilon^{1-\alpha_2}|\IL^{-\alpha_1}h_y^1||\IL^{-\alpha_2}h_x^2|\Bigr)\\
&+C_\kappa(T)\vvvert\varphi\vvvert_2 \bigl(\frac{1}{t^{\alpha_1}}|\IL^{-\alpha_1}h_x^1|+\epsilon|\IL^{-\alpha_1}h_y^1|\bigr)\bigl(\frac{1}{t^{\alpha_2}}|\IL^{-\alpha_2}h_x^2|+\epsilon|\IL^{-\alpha_2}h_y^2|\bigr).
\end{align*}
Let $h,k\in H$. Considering the case with ${\bf h}^1=(h^1,0)$ and ${\bf h}^2=(h^2,0)$, one obtains
\begin{align*}
|D_x^2u^\epsilon(t,x,y).(h^1,h^2)|&=|D^2u^\epsilon(t,x,y).((h^1,0),(h^2,0))|\\
&\le \frac{C_{\alpha_1,\alpha_2}(T)}{t^{\alpha_1+\alpha_2}}\vvvert\varphi\vvvert_2 |\IL^{-\alpha_1}h^1||\IL^{-\alpha_2}h^2|.
\end{align*}
Similarly, considering the case with ${\bf h}^1=(h^1,0)$ and ${\bf h}^2=(0,h^2)$, one obtains
\begin{align*}
|D_xD_yu^\epsilon(t,x,y).(h^1,h^2)|&=|D^2u^\epsilon(t,x,y).((h^1,0),(0,h^2))|\\
&\le \epsilon^{1-\alpha_1}\frac{C_{\alpha_1,\alpha_2}(T)}{t^{\alpha_1}}\vvvert\varphi\vvvert_2 |\IL^{-\alpha_1}h^1||\IL^{-\alpha_2}h^2|,
\end{align*}
and considering the case with ${\bf h}^1=(0,h^1)$ and ${\bf h}^2=(0,h^2)$, one obtains
\begin{align*}
|D_y^2u^\epsilon(t,x,y).(h^1,h^2)|&=|D^2u^\epsilon(t,x,y).((0,h^1),(0,h^2))|\\
&\le \epsilon C_{\alpha_1,\alpha_2}(T)\vvvert\varphi\vvvert_2 |\IL^{-\alpha_1}h^1||\IL^{-\alpha_2}h^2|.
\end{align*}
The proof of the inequality~\eqref{eq:lemuepsilon-2} is thus completed.
\end{proof}

\begin{proof}[Proof of the inequality~\eqref{eq:lemuepsilon-3}]
Using the fact that $u^\epsilon$ solves the Kolmogorov equation~\eqref{eq:Kolmogorov}, one has
\begin{align*}
\partial_t\langle D_xu^\epsilon(t,x,y),h\rangle&=\langle D_x\partial_tu^\epsilon(t,x,y),h\rangle\\
&=\langle D_xu^\epsilon(t,x,y),-\IL h+D_xF(x,y).h\rangle\\
&+D_x^2u^\epsilon(t,x,y).(-\IL x+F(x,y),h)\\
&-\frac{1}{\epsilon}D_xD_yu^\epsilon(t,x,y).(h,\IL y)+\frac{1}{\epsilon}\sum_{j\in\N}D_xD_y^2u^\epsilon(t,x,y).(h,e_j,e_j).
\end{align*}
Using the inequality~\eqref{eq:lemuepsilon-1}, one obtains the upper bound
\begin{equation}\label{eq:lemuepsilon-3-1}
|\langle D_xu^\epsilon(t,x,y),-\IL h+D_xF(x,y).h\rangle|\le \frac{C_\kappa(T)}{t^{1-\kappa}}\vvvert\varphi\vvvert_1 |\IL^\kappa h|.
\end{equation}
Using the inequality~\eqref{eq:lemuepsilon-2} and the linear growth property of $F$, one obtains the upper bounds
\begin{equation}\label{eq:lemuepsilon-3-2}
|D_x^2u^\epsilon(t,x,y).(-\IL x+F(x,y),h)|\le \frac{C_\kappa(T)}{t^{1-\kappa}}\vvvert\varphi\vvvert_2(1+|\IL^\kappa x|+|y|)|h|
\end{equation}
and
\begin{equation}\label{eq:lemuepsilon-3-3}
|D_xD_yu^\epsilon(t,x,y).(h,\IL y)|\le \frac{C_\kappa(T)}{t^{1-\kappa}}\vvvert\varphi\vvvert_2|\IL^\kappa y||h|.
\end{equation}
In order to deal with the last term in the expression above, one needs to prove the following upper bound: for all $t\in(0,T]$, $x,y\in H$ and ${\bf h}^1,{\bf h}^2,{\bf h}^3\in\HH$, one has
\begin{equation}\label{eq:lemuepsilon-aux}
|D^3u^\epsilon(t,x,y).({\bf h}^1,{\bf h}^2,{\bf h}^3)|\le \frac{C_\kappa(T)}{t^{1-\kappa}}|{\bf h}^1|(|\IL^{-1+\kappa}h_x^2|+|\IL^{-1+\kappa}h_y^2|)|{\bf h}^3|,
\end{equation}
with ${\bf h}^2=(h_x^2,h_y^2)$. The proof of the auxiliary inequality~\eqref{eq:lemuepsilon-aux} is similar to the proofs of the inequalities~\eqref{eq:lemuepsilon-1} and~\eqref{eq:lemuepsilon-2}, but there is a crucial difference which makes the arguments simpler: the inequality~\eqref{eq:lemuepsilon-aux} states bounds which are uniform with respect to $\epsilon$, whereas for the two other inequalities the dependence with respect to $\epsilon$ is made more explicit. A version of~\eqref{eq:lemuepsilon-aux} with a similar analysis of the dependence with respect to $\epsilon$ may be obtained but is useless for the proof of the inequality~\eqref{eq:lemuepsilon-3} and is therefore omitted.

Let us give the proof of the auxiliary inequality~\eqref{eq:lemuepsilon-aux}. One has the expression
\begin{align*}
D^3u^\epsilon(t,x,y).({\bf h}^1,{\bf h^2},{\bf h^3})&=\E_{x,y}[D\varphi(\XX^{\epsilon}(t)).\xi_x^{\epsilon,{\bf h}^1,{\bf h}^2,{\bf h}^3}(t)]\\
&+\E_{x,y}[D^2\varphi(\XX^{\epsilon}(t)).(\eta_x^{\epsilon,{\bf h}^1}(t),\zeta_x^{\epsilon,{\bf h}^2,{\bf h}^3}(t))]\\
&+\E_{x,y}[D^2\varphi(\XX^{\epsilon}(t)).(\eta_x^{\epsilon,{\bf h}^2}(t),\zeta_x^{\epsilon,{\bf h}^3,{\bf h}^1}(t))]\\
&+\E_{x,y}[D^2\varphi(\XX^{\epsilon}(t)).(\eta_x^{\epsilon,{\bf h}^3}(t),\zeta_x^{\epsilon,{\bf h}^1,{\bf h}^2}(t))]\\
&+\E_{x,y}[D^3\varphi(\XX^{\epsilon}(t)).\bigl(\eta_x^{\epsilon,{\bf h}^1}(t),\eta_x^{\epsilon,{\bf h}^2}(t),\eta_x^{\epsilon,{\bf h}^3}(t))\bigr)],
\end{align*}
where $t\in[0,T]\mapsto \xi^{\epsilon,{\bf h}^1,{\bf h}^2,{\bf h}^3}(t)=(\xi_x^{\epsilon,{\bf h}^1,{\bf h}^2,{\bf h}^3}(t),\xi_y^{\epsilon,{\bf h}^1,{\bf h}^2,{\bf h}^3}(t)) \HH$ is solution of
\[
\left\lbrace
\begin{aligned}
\frac{d\xi_x^{\epsilon,{\bf h}^1,{\bf h}^2,{\bf h}^3}(t)}{dt}&=-\IL\xi_x^{\epsilon,{\bf h}^1,{\bf h}^2,{\bf h}^3}(t)+DF(\XX^\epsilon(t),\YY^\epsilon(t)).\xi^{\epsilon,{\bf h}^1,{\bf h}^2,{\bf h}^3}(t)\\
&+D^2F(\XX^\epsilon(t),\YY^\epsilon(t)).(\eta^{\epsilon,{\bf h}^1}(t),\zeta^{\epsilon,{\bf h}^2,{\bf h}^3}(t))\\
&+D^2F(\XX^\epsilon(t),\YY^\epsilon(t)).(\eta^{\epsilon,{\bf h}^2}(t),\zeta^{\epsilon,{\bf h}^3,{\bf h}^1}(t))\\
&+D^2F(\XX^\epsilon(t),\YY^\epsilon(t)).(\eta^{\epsilon,{\bf h}^3}(t),\zeta^{\epsilon,{\bf h}^1,{\bf h}^2}(t))\\
&+D^3F(\XX^\epsilon(t),\YY^\epsilon(t)).(\eta^{\epsilon,{\bf h}^1}(t),\eta^{\epsilon,{\bf h}^2}(t),\eta^{{\bf h}^3}(t)),\\
\frac{d\xi_y^{\epsilon,{\bf h}^1,{\bf h}^2,{\bf h}^3}(t)}{dt}&=0,
\end{aligned}
\right.
\]
with initial values $\xi_x^{\epsilon,{\bf h}^1,{\bf h}^2,{\bf h}^3}(0)=0$ and $\xi_y^{\epsilon,{\bf h}^1,{\bf h}^2,{\bf h}^3}(0)=0$.
 
In the proofs of the inequalities~\eqref{eq:lemuepsilon-1} and~\eqref{eq:lemuepsilon-2}, the following auxiliary results have been obtained (where the dependence with respect to $\epsilon$ is not indicated): for all $t\in(0,T]$, $x,y\in H$ and ${\bf h}=(h_x,h_y),{\bf k}\in\HH$, one has
\begin{align*}
|\eta^{\epsilon,{\bf h}}(t)|&\le C_\kappa(T)t^{-1+\kappa}|\IL^{-1+\kappa}{\bf h}|\\
|\zeta^{\epsilon,{\bf h},{\bf k}}(t)|&\le C_\kappa(T)(|\IL^{-1+\kappa}h_x|+|\IL^{-1+\kappa}h_y|)|{\bf k}|.
\end{align*}
Using a mild formulation for $\xi_x^{\epsilon,{\bf h}^1,{\bf h}^2,{\bf h}^3}(t)$, the boundedness of the derivatives of $F$ of order $1,2,3$ (Assumption~\ref{ass:F}), the two upper bounds above (and versions using symmetries with respect to permutations of ${\bf h}^1$, ${\bf h}^2$ and ${\bf h}^3$), and Gronwall's lemma, one obtains the upper bound
\[
\underset{0\le t\le T}\sup~|\xi_x^{\epsilon,{\bf h}^1,{\bf h}^2,{\bf h}^3}(t)|\le C_\kappa(T)|{\bf h}^1|(|\IL^{-\frac12-\kappa}h_x^2|+|\IL^{-\frac12-\kappa}h_y^2|)|{\bf h}^3|.
\]
Using the expression for $D^3u^\epsilon(t,x,y).({\bf h}^1,{\bf h^2},{\bf h^3})$ above and the upper bounds, the proof of the auxiliary inequality~\eqref{eq:lemuepsilon-aux} is completed.

We are now in position to conclude the proof of the inequality~\eqref{eq:lemuepsilon-3}: using the auxiliary inequality~\eqref{eq:lemuepsilon-aux}, the last term satisfies the following upper bound:
\begin{equation}\label{eq:lemuepsilon-3-4}
\big|\sum_{j\in\N}D_xD_y^2u^\epsilon(t,x,y).(h,e_j,e_j)\big|\le \frac{C_\kappa(T)}{t^{1-\kappa}}\vvvert\varphi\vvvert_3|h|\sum_{j\in \N}\lambda_j^{-1+\kappa}\le \frac{C_\kappa(T)}{t^{1-\kappa}}\vvvert\varphi\vvvert_3|h|.
\end{equation}

Gathering the four upper bounds~\eqref{eq:lemuepsilon-3-1}, \eqref{eq:lemuepsilon-3-2}, \eqref{eq:lemuepsilon-3-3} and~\eqref{eq:lemuepsilon-3-4} and using the expression of $\partial_t\langle D_xu^\epsilon(t,x,y),h\rangle$ then gives the inequality~\eqref{eq:lemuepsilon-3}.
\end{proof}

\subsection{Proof of Lemma~\ref{lem:ubarDeltat}}\label{sec:proof-ubarDeltat}

Before proceeding with the proofs of the regularity estimates stated in Lemma~\ref{lem:ubarDeltat}, note that the mapping $\overline{u}_{n}^{\Delta t}:H\to\R$ is of class $\mathcal{C}^2$: this is proved by recursion using the expression
\begin{equation}\label{eq:u_n+1-n}
\overline{u}_{n+1}^{\Delta t}(x)=\overline{u}_n^{\Delta t}(\IA_{\Delta t}x+\Delta t\IA_{\Delta t}\overline{F}(x)),
\end{equation}
with the initial value $\overline{u}_0^{\Delta t}=\varphi$ being of class $\mathcal{C}^2$.

In addition, using the identity
\[
\overline{u}_{n+1}^{\Delta t}(x)=\varphi(\IA_{\Delta t}\overline{\XX}_n^{\Delta t;x}+\Delta t\IA_{\Delta t}\overline{F}(\overline{\XX}_n^{\Delta t;x})),
\]
a recursion argument proves the following expressions: for all $x,h,k\in H$, and $n\in\{0,\ldots,N\}$, one has
\begin{align}
\langle D\overline{u}_n^{\Delta t}(x),h\rangle&=\langle D\varphi(\overline{\XX}_n^{\Delta;x}),\eta_n^{\Delta t;x,h}\rangle\\
D^2\overline{u}_n^{\Delta t}(x).(h,k)&=D^2\varphi(\overline{\XX}_n^{\Delta t;x}).(\eta_n^{\Delta t;x,h},\eta_n^{\Delta t;x,k})+\langle D\varphi(\overline{\XX}_n^{\Delta t;x}),\zeta_n^{\Delta t;x,h,k}\rangle,
\end{align}
where the auxiliary sequences $\bigl(\eta_n^{\Delta t;x,h}\bigr)_{n\ge 0}$ and $\bigl(\zeta_n^{\Delta t;x,h}\bigr)_{n\ge 0}$ are defined by
\begin{align}
\eta_{n+1}^{\Delta t;x,h}&=\IA_{\Delta t}\eta_n^{\Delta t;x,h}+\Delta t\IA_{\Delta t}D\overline{F}(\overline{\XX}_n^{\Delta t;x}).\eta_n^{\Delta t;x,h}\\
\zeta_{n+1}^{\Delta t;x,h}&=\IA_{\Delta t}\zeta_n^{\Delta t;x,h,k}+\Delta t\IA_{\Delta t}D\overline{F}(\overline{\XX}_n^{\Delta t;x}).\zeta_n^{\Delta t;x,h,k}+\Delta t\IA_{\Delta t}D^2\overline{F}(\overline{\XX}_n^{\Delta t;x}).(\eta_n^{\Delta t;x,h},\eta_n^{\Delta t;x,k}),
\end{align}
with initial values $\eta_0^{\Delta t;x,h}=h$ and $\zeta_0^{\Delta t;x,h,k}=0$.

\begin{proof}[Proof of the inequality~\eqref{eq:lemubarDeltat_1}]
Introduce the auxiliary variable $\tilde{\eta}_n^{\Delta t;x,h}=\eta_n^{\Delta t;x,h}-\IA_{\Delta t}^nh$ for all $n\in\{0,\ldots,N\}$. Using the  inequality~\eqref{eq:smoothingnum}, one obtains
\begin{align*}
|\langle D\overline{u}_n^{\Delta t}(x),h\rangle|&\le \big|\langle D\varphi(\overline{\XX}_n^{\Delta t;x}),\IA_{\Delta t}^n h\rangle\big|+\big|\langle D\varphi(\overline{\XX}_n^{\Delta t;x}),\tilde{\eta}_n^{\Delta t;x,h}\rangle\big|\\
&\le \frac{C_\kappa\vvvert\varphi\vvvert_1}{(n\Delta t)^{1-\kappa}}|\IL^{-1+\kappa}h|+\vvvert\varphi\vvvert_1|\tilde{\eta}_n^{\Delta t;x,h}|.
\end{align*}
Observe that the auxiliary sequence $\bigl(\tilde{\eta}_n^{\Delta t;x,h}\bigr)_{n\ge 0}$ satisfies for all $n\ge 0$
\[
\tilde{\eta}_{n+1}^h=\IA_{\Delta t}\tilde{\eta}_n^h+\Delta t\IA_{\Delta t}D\overline{F}(\overline{\XX}_n^{\Delta t;x}).\tilde{\eta}_n^h+\Delta t\IA_{\Delta t}D\overline{F}(\overline{\XX}_n^{\Delta t;x}).(\IA_{\Delta t}^n h),
\]
with $\tilde{\eta}_0^{\Delta t;x,h}=0$. As a consequence, one obtains the equality
\[
\tilde{\eta}_{n}^{\Delta t;x,h}=\Delta t\sum_{\ell=0}^{n-1}\IA_{\Delta t}^{n-\ell}D\overline{F}(\overline{\XX}_\ell^{\Delta t;x}).\tilde{\eta}_\ell^{\Delta t;x,h}+\Delta t\sum_{\ell=0}^{n-1}\IA_{\Delta t}^{n-\ell}D\overline{F}(\overline{\XX}_\ell^{\Delta t;x}).(\IA_{\Delta t}^\ell h),
\]
which gives, using the inequality~\eqref{eq:smoothingnum}, for all $n\in\{1,\ldots,N\}$
\begin{align*}
|\tilde{\eta}_n^{\Delta t;x,h}|&\le C\Delta t\sum_{\ell=0}^{n-1}|\tilde{\eta}_\ell^{\Delta t;x,h}|+C\Delta t\sum_{\ell=0}^{n-1}|\IA_{\Delta t}^\ell h|\\
&\le C\Delta t\sum_{\ell=0}^{n-1}|\tilde{\eta}_\ell^{\Delta t;x,h}|+C\Delta t|h|+\Delta t\sum_{\ell=1}^{n-1}\frac{C_\kappa}{(\ell\Delta t)^h}|\IL^{-1+\kappa}h|\\
&\le C\Delta t\sum_{\ell=0}^{n-1}|\tilde{\eta}_k^{\Delta t;x,h}|+C\Delta t|h|+C_{\kappa}(T)|\IL^{-1+\kappa}h|.
\end{align*}
Applying the discrete Gronwall inequality yields
\[
|\tilde{\eta}_n^{\Delta t;x,h}|\le C_\kappa(T)(\Delta t|h|+|\IL^{-1+\kappa}h|)
\]
for all $n\in\{0,\ldots,N\}$, and finally one obtains the inequality
\begin{align*}
|\langle D\overline{u}_n^{\Delta t}(x),h\rangle|&\le \frac{C_\kappa\vvvert\varphi\vvvert_1}{(n\Delta t)^{1-\kappa}}|\IL^{-1+\kappa}h|+\vvvert\varphi\vvvert_1|\tilde{\eta}_n^{\Delta t;x,h}|\\
&\le \frac{C_\kappa\vvvert\varphi\vvvert_1}{(n\Delta t)^{1-\kappa}}|\IL^{-1+\kappa}h|+C_\kappa(T)\vvvert\varphi\vvvert_1(\Delta t|h|+|\IL^{-1+\kappa}h|)
\end{align*}
which concludes the proof of the inequality~\eqref{eq:lemubarDeltat_1}.
\end{proof}

\begin{proof}[Proof of the inequality~\eqref{eq:lemubarDeltat_2}]

Using the identity $\eta_n^{\Delta t;x,h}=\IA_{\Delta t}^{n}h+\tilde{\eta}_n^{\Delta t;x,h}$ for all $n\in\{0,\ldots,N\}$ and the inequality above, one obtains
\begin{align*}
\big|D^2\overline{u}_n^{\Delta t}(x).(h,k)|&\le \vvvert\varphi\vvvert_1|\zeta_n^{\Delta t;x,h,k}|+\vvvert\varphi\vvvert_2|\eta_n^{\Delta t;x,h}||\eta_n^{\Delta t;x,k}|\\
&\le \vvvert\varphi\vvvert_1|\zeta_n^{\Delta t;x,h,k}|+C_\kappa(T)\vvvert\varphi\vvvert_2\frac{C_\kappa(T)}{(n\Delta t)^{1-\kappa}}\vvvert\varphi\vvvert_1 \bigl(\Delta t|h|+|\IL^{-1+\kappa}h|\bigr)|k|.
\end{align*}
It remains to give an upper bound for $|\zeta_n^{\Delta t;x,h,k}|$: for all $n\in\{0,\ldots,N\}$, one has
\[
\zeta_n^{\Delta t;x,h,k}=\Delta t\sum_{\ell=0}^{n-1}\IA_{\Delta t}^{n-\ell}D\overline{F}(\overline{\XX}_\ell^{\Delta t;x}).\zeta_\ell^{\Delta t;x,h,k}+\Delta t\sum_{\ell=0}^{n-1}\IA_{\Delta t}^{n-\ell}D^2\overline{F}(\XX_\ell^{\Delta t;x}).(\eta_\ell^{\Delta t;x,h},\eta_\ell^{\Delta t;x,k}).
\]
Since $\overline{F}:H\to\R$ is of class $\mathcal{C}^2$ with bounded first and second order derivatives, one obtains
\begin{align*}
|\zeta_n^{\Delta t;x,h,k}|&\le C\Delta t\sum_{\ell=0}^{n-1}|\zeta_\ell^{\Delta t;x,h,k}|+C\Delta t\sum_{\ell=0}^{n-1}|\eta_\ell^{\Delta t;x,h}||\eta_\ell^{\Delta t;x,k}|\\
&\le C\Delta t\sum_{\ell=0}^{n-1}|\zeta_\ell^{h,k}|+C\Delta t|h||k|+\Delta t\sum_{\ell=1}^{n-1}\frac{C_\kappa(T)}{(\ell\Delta t)^{1-\kappa}}|\IL^{-1+\kappa}h||k|\\
&\le C\Delta t\sum_{\ell=0}^{n-1}|\zeta_\ell^{h,k}|+C_\kappa(T)(\Delta t|h|+|\IL^{-1+\kappa}h|)|k|
\end{align*}
for all $n\in\{0,\ldots,N\}$. The discrete Gronwall inequality then yields
\[
\underset{0\le n\le N}\sup~|\zeta_n^{\Delta t;x,h,k}|\le C_\kappa(T)(\Delta t|h|+|\IL^{-1+\kappa}h|)|k|.
\]
Gathering the estimates then concludes the proof of the inequality~\eqref{eq:lemubarDeltat_2}.
\end{proof}

\begin{proof}[Proof of the inequality~\eqref{eq:lemubarDeltat_3}]
Using the identity~\eqref{eq:u_n+1-n}, for all $x,h\in H$ and for all $n\in\{1,\ldots,N\}$, one obtains the equality
\[
\langle D\overline{u}_{n+1}^{\Delta t}(x),h\rangle=\langle D\overline{u}_n^{\Delta t}\bigl(\IA_{\Delta t}x+\Delta t\IA_{\Delta t}\overline{F}(x)\bigr),\IA_{\Delta t}h+\Delta t\IA_{\Delta t}D\overline{F}(x).h\rangle.
\]
As a consequence, one has the inequality
\begin{align*}
\big|\langle D\overline{u}_{n+1}^{\Delta t}(x)-D\overline{u}_{n}^{\Delta t}(x),h\rangle\big|&\le \big|\langle D\overline{u}_n^{\Delta t}(\IA_{\Delta t}x+\Delta t\IA_{\Delta t}\overline{F}(x))-D\overline{u}_n^{\Delta t}(x+\Delta t\IA_{\Delta t}\overline{F}(x)),h\rangle\big|\\
&+\big|\langle D\overline{u}_n^{\Delta t}(x+\Delta t\IA_{\Delta t}\overline{F}(x))-D\overline{u}_n^{\Delta t}(x),h\rangle\big|\\
&+\big|\langle D\overline{u}_n^{\Delta t}(\IA_{\Delta t}x+\Delta t\IA_{\Delta t}\overline{F}(x)),(\IA_{\Delta t}-I)h\rangle\big|\\
&+\big|\langle D\overline{u}_n^{\Delta t}(\IA_{\Delta t}x+\Delta t\IA_{\Delta t}\overline{F}(x)),\Delta t\IA_{\Delta t}D\overline{F}(x).h\rangle\big|
\end{align*}
and it remains to prove upper bounds for the four terms appearing in the right-hand side above. Let $\kappa\in(0,1]$.

$\bullet$ Using the inequality~\eqref{eq:lemubarDeltat_2}, for the first term, one obtains
\begin{align*}
\big|\langle D\overline{u}_n^{\Delta t}(\IA_{\Delta t}x+\Delta t\IA_{\Delta t}\overline{F}(x))&-D\overline{u}_n^{\Delta t}(x+\Delta t\IA_{\Delta t}\overline{F}(x)),h\rangle\big|\\
&\le \frac{C_\kappa(T)}{(n\Delta t)^{1-\kappa}}\vvvert\varphi\vvvert_2\bigl(\Delta t|(\IA_{\Delta t}-I)x|+|\IL^{-1+\kappa}(\IA_{\Delta t}-I)x|\bigr)|h|\\
&\le \frac{C_\kappa(T)\Delta t^{1-\kappa}}{(n\Delta t)^{1-\kappa}}\vvvert\varphi\vvvert_2|x||h|.
\end{align*}

$\bullet$ Using the inequality~\eqref{eq:lemubarDeltat_4} -- or the inequality~\eqref{eq:lemubarDeltat_2} with $\kappa=1$ -- and the global Lipschitz continuity of $\overline{F}$, for the second term, one obtains
\begin{align*}
\big|\langle D\overline{u}_n^{\Delta t}(x+\Delta t\IA_{\Delta t}\overline{F}(x))-D\overline{u}_n^{\Delta t}(x),h\rangle\big|&\le C(T)\vvvert\varphi\vvvert_2 \Delta t|\overline{F}(x)||h|\\
&\le C(T)\Delta t\vvvert\varphi\vvvert_2 (1+|x|)|h|.
\end{align*}

$\bullet$ Using the inequality~\eqref{eq:lemubarDeltat_1}, for the third term, one obtains
\begin{align*}
\big|\langle D\overline{u}_n^{\Delta t}(\IA_{\Delta t}x+\Delta t\IA_{\Delta t}\overline{F}(x))&,(\IA_{\Delta t}-I)h\rangle\big|\\
&\le \frac{C_\kappa(T)}{(n\Delta t)^{1-\kappa}}\vvvert\varphi\vvvert_1\bigl(\Delta t|(\IA_{\Delta t}-I)h|+|\IL^{-1+\kappa}(\IA_{\Delta t}-I)h|\bigr)\\
&\le \frac{C_\kappa(T)\Delta t^{1-\kappa}}{(n\Delta t)^{1-\kappa}}\vvvert\varphi\vvvert_1|h|.
\end{align*}

$\bullet$ Using the inequality~\eqref{eq:lemubarDeltat_4} -- or the inequality~\eqref{eq:lemubarDeltat_1} with $\kappa=1$ -- and the global Lipschitz continuity of $\overline{F}$, for the fourth term, one obtains
\begin{align*}
\big|\langle D\overline{u}_n^{\Delta t}(\IA_{\Delta t}x+\Delta t\IA_{\Delta t}\overline{F}(x)),\Delta t\IA_{\Delta t}D\overline{F}(x).h\rangle\big|&\le C(T)\vvvert\varphi\vvvert_1 \Delta t|\IA_{\Delta t}D\overline{F}(x).h|\\
&\le C(T)\Delta t\vvvert\varphi\vvvert_1(1+|x|)|h|.
\end{align*}
Gathering the estimates for the fourth terms considered above, one obtains the upper bound
\[
\big|\langle D\overline{u}_{n+1}^{\Delta t}(x)-D\overline{u}_{n}^{\Delta t}(x),h\rangle\big|\le \frac{C_\kappa(T)\Delta t^{1-\kappa}}{(n\Delta t)^{1-\kappa}}\vvvert\varphi\vvvert_2(1+|x|)|h|,
\]
which concludes the proof of the inequality~\eqref{eq:lemubarDeltat_3}.
\end{proof}

\section{Proof of the asymptotic preserving property}\label{sec:proof-AP}

This section is devoted to the proof of Propositions~\ref{propo:cv_scheme-limitingscheme} and~\ref{propo:error_limitingscheme-averagedequation} stated in Section~\ref{sec:AP}. The arguments are the same as in the proof of~\cite[Theorem~9.1]{B}, however they are also given here to make the presentation self-contained.

\subsection{Proof of Proposition~\ref{propo:cv_scheme-limitingscheme}}\label{sec:proof-AP1}

\begin{proof}[Proof of Proposition~\ref{propo:cv_scheme-limitingscheme}]
It is convenient to employ the following interpretation of the modified Euler scheme (see~\cite[Section~3.2]{B}): if $\bigl(\Gamma_n\bigr)_{n\ge 0}$ is a sequence of independent cylindrical Gaussian random variables, one has the equality in distribution 
\[
\IB_{\frac{\Delta t}{\epsilon},1}\Gamma_{n,1}+\IB_{\frac{\Delta t}{\epsilon},2}\Gamma_{n,2}=\IB_{\frac{\Delta t}{\epsilon}}\Gamma_{n}
\]
for all $n\ge 0$, where the self-adjoint linear operator $\IB_{\frac{\Delta t}{\epsilon}}=\IB_\tau$ is defined by
\[
\IB_\tau x=\sum_{j\in\N}\frac{\sqrt{2+\lambda_j\tau}}{\sqrt{2}~(1+\lambda_j\tau)}\langle x,e_j\rangle e_j
\]
for all $x\in H$, and satisfies the identity
\[
\IB_\tau^2=\IB_{\tau,1}^2+\IB_{\tau,2} \IB_{\tau,2}^\star=\frac12\bigl(\IA_\tau^2+\IA_\tau\bigr)=\frac12(2I+\tau\IL)(I+\tau\IL)^{-2}.
\]
As a consequence, one has the equality in distribution
\[
\bigl(\XX_n^{\epsilon,\Delta t},\YY_n^{\epsilon,\Delta t}\bigr)_{n\ge 0}=\bigl(\hat{\XX}_n^{\epsilon,\Delta t},\hat{\YY}_n^{\epsilon,\Delta t}\bigr)_{n\ge 0}
\]
where the scheme $\bigl(\hat{\XX}_n^{\epsilon,\Delta t},\hat{\YY}_n^{\epsilon,\Delta t}\bigr)_{n\ge 0}$ is defined by
\begin{equation}\label{eq:APscheme2nd}
\left\lbrace
\begin{aligned}
\hat{\XX}_{n+1}^{\epsilon,\Delta t}&=\IA_{\Delta t}\bigl(\hat{\XX}_n^{\epsilon,\Delta t}+\tau F(\hat{\XX}_n^{\epsilon,\Delta t},\hat{\YY}_{n+1}^{\epsilon,\Delta t})\bigr)\\
\hat{\YY}_{n+1}^{\epsilon,\Delta t}&=\IA_{\frac{\Delta t}{\epsilon}}\hat{\YY}_n^{\epsilon,\Delta t}+\sqrt{\frac{2\Delta t}{\epsilon}}\IB_{\frac{\Delta t}{\epsilon}}\Gamma_{n},
\end{aligned}
\right.
\end{equation}
with initial values $\hat{\XX}_0^{\epsilon,\Delta t}=x_0^\epsilon=\XX_0^{\epsilon,\Delta t}$ and $\hat{\YY}_0^{\epsilon,\Delta t}=y_0^\epsilon=\YY_0^{\epsilon,\Delta t}$. In particular, one has
\[
\E[\varphi(\XX_{n}^{\epsilon,\Delta t})]=\E[\varphi(\hat{\XX}_{n}^{\epsilon,\Delta t})].
\]
Since the function $\varphi$ is assumed to be globally Lipschitz continuous, it suffices to prove that, for all $n\in\{0,\ldots,N\}$, one has
\[
\E[|\hat{\XX}_n^{\epsilon,\Delta t}-\XX_n^{\Delta t}|]\underset{\epsilon\to 0}\to 0.
\]
Note that for all $n\in\{0,\ldots,N\}$, one has the identities
\begin{align*}
\hat{\XX}_n^{\epsilon,\Delta t}&=\IA_{\Delta t}^n x_0^\epsilon+\Delta t\sum_{\ell=0}^{n-1}\IA_{\Delta t}^{n-\ell}F(\hat{\XX}_\ell^{\epsilon,\Delta t},\hat{\YY}_{\ell+1}^{\epsilon,\Delta t}),\\
{\XX}_n^{\Delta t}&=\IA_{\Delta t}^n x_0+\Delta t\sum_{\ell=0}^{n-1}\IA_{\Delta t}^{n-\ell}F({\XX}_\ell^{\Delta t},\IL^{-\frac12}\Gamma_{\ell}).
\end{align*}
Therefore, for all $n\in\{0,\ldots,N\}$, one has
\begin{align*}
\E[|\hat{\XX}_n^{\epsilon,\Delta t}-\XX_n^{\Delta t}|]&\le |x_0^\epsilon-x_0|+\Delta t\sum_{\ell=0}^{n-1}\E[|F(\hat{\XX}_\ell^{\epsilon,\Delta t},\hat{\YY}_{\ell+1}^{\epsilon,\Delta t})-F(\hat{\XX}_\ell^{\epsilon,\Delta t},\IL^{-\frac12}\Gamma_{\ell})|]\\
&\le |x_0^\epsilon-x_0|+\Delta t\sum_{\ell=0}^{n-1}\E[|F(\hat{\XX}_\ell^{\epsilon,\Delta t},\hat{\YY}_{\ell+1}^{\epsilon,\Delta t})-\hat{F}(\XX_\ell^{\Delta t},\hat{\YY}_{\ell+1}^{\epsilon,\Delta t})|]\\
&+\Delta t\sum_{\ell=0}^{n-1}\E[|F(\XX_\ell^{\Delta t},\hat{\YY}_{\ell+1}^{\epsilon,\Delta t})-F(\XX_\ell^{\Delta t},\IL^{-\frac12}\Gamma_{\ell,2})|]\\
&\le |x_0^\epsilon-x_0|+C\Delta t\sum_{\ell=0}^{n-1}\E[|\hat{\XX}_\ell^{\epsilon,\Delta t}-\XX_\ell^{\Delta t}|]+C\Delta t\sum_{\ell=0}^{n-1}\E[|\hat{\YY}_{\ell+1}^{\epsilon,\Delta t}-\IL^{-\frac12}\Gamma_{\ell}|],
\end{align*}
using the global Lipschitz continuity property of $F$ (Assumption~\ref{ass:F}).

Using a straightforward recursion argument (see details below), note that it suffices to check the following claim: for all $\ell\ge 0$, one has
\[
\E[|\YY_{\ell+1}^{\epsilon,\Delta t}-\IL^{-\frac12}\Gamma_{\ell}|^2]\underset{\epsilon\to 0}\to 0.
\]
By the definition of the scheme~\eqref{eq:APscheme2nd}, one has
\[
\hat{\YY}_{\ell+1}^{\epsilon,\Delta t}-\IL^{-\frac12}\Gamma_{\ell}=\IA_{\frac{\Delta t}{\epsilon}}\hat{\YY}_\ell^{\epsilon,\Delta t}+\Bigl(\sqrt{\frac{2\Delta t}{\epsilon}}\IB_{\frac{\Delta t}{\epsilon}}-\IL^{-\frac12}\Bigr)\Gamma_{\ell}.
\]
On the one hand, using the moment bound
\[
\underset{\epsilon\in(0,\epsilon_0),\Delta t\in(0,\Delta t_0)}\sup~\underset{\ell\ge 0}\sup~\E[|\hat{\YY}_\ell^{\epsilon,\Delta t}|^2]=\underset{\epsilon\in(0,\epsilon_0),\Delta t\in(0,\Delta t_0)}\sup~\underset{\ell\ge 0}\sup~\E[|\YY_\ell^{\epsilon,\Delta t}|^2]<\infty,
\]
see the inequality~\eqref{eq:boundYnepsilonDeltat} from Lemma~\ref{lem:momentboundsscheme-epsilonDeltat}, one obtains
\[
\E[|\IA_{\frac{\Delta t}{\epsilon}}\hat{\YY}_\ell^{\epsilon,\Delta t}|]\le \frac{C}{1+\lambda_1\frac{\Delta t}{\epsilon}}\underset{\epsilon\to 0}\to 0.
\]
On the other hand, one has
\begin{align*}
\E[|\bigl(\sqrt{\frac{2\Delta t}{\epsilon}}\IB_{\frac{\Delta t}{\epsilon}}-\IL^{-\frac12}\bigr)\Gamma_{\ell}|^2]&=\sum_{j\in\N}\bigl(\sqrt{\frac{2\Delta t}{\epsilon}}\frac{\sqrt{2+\lambda_j\frac{\Delta t}{\epsilon}}}{\sqrt{2}(1+\lambda_j\frac{\Delta t}{\epsilon})}-\frac{1}{\sqrt{\lambda_j}}\bigr)^2\\
&=\sum_{j\in\N}\frac{1}{\bigl(\sqrt{2\frac{\Delta t}{\epsilon}}\frac{\sqrt{2+\lambda_j\frac{\Delta t}{\epsilon}}}{\sqrt{2}(1+\lambda_j\frac{\Delta t}{\epsilon})}+\frac{1}{\sqrt{\lambda_j}}\bigr)^2}\bigl(\frac{\frac{\Delta t}{\epsilon}(2+\lambda_j\frac{\Delta t}{\epsilon})}{(1+\lambda_j\frac{\Delta t}{\epsilon})^2}-\frac{1}{\lambda_j}\bigr)^2\\
&\le \sum_{j\in\N}\lambda_j\frac{1}{\bigl(\lambda_j(1+\lambda_j\frac{\Delta t}{\epsilon})^2\bigr)^2}\\
&\le \sum_{j\in\N}\frac{1}{\lambda_j(1+\lambda_j\frac{\Delta t}{\epsilon})^4}\\
&\underset{\epsilon\to 0}\to 0.
\end{align*}

As a consequence, one obtains
\[
\underset{\epsilon\to 0}\limsup~\E[|\hat{\XX}_n^{\epsilon,\Delta t}-\XX_n^{\Delta t}|]\le C\Delta t\sum_{\ell=0}^{n-1}\underset{\epsilon\to 0}\limsup~\E[|\hat{\XX}_\ell^{\epsilon,\Delta t}-\XX_\ell^{\Delta t}|],
\]
for all $n\in\{0,\ldots,N\}$. Since $\hat{\XX}_0^{\epsilon,\Delta t}-\XX_0^{\Delta t}=x_0^\epsilon-x_0\underset{\epsilon\to 0}\to 0$ owing to Assumption~\ref{ass:init}, it is then straightforward to conclude that
\[
\underset{\epsilon\to 0}\limsup~\E[|\hat{\XX}_n^{\epsilon,\Delta t}-\XX_n^{\Delta t}|]=0
\]
for all $n\in\{0,\ldots,N\}$. As explained above, this yields
\[
\underset{\epsilon\to 0}\lim~\E[\varphi(\XX_{n}^{\epsilon,\Delta t})]=\underset{\epsilon\to 0}\lim~\E[\varphi(\hat{\XX}_{n}^{\epsilon,\Delta t})]=\E[\varphi(\XX_{n}^{\Delta t})]
\]
and concludes the proof of~\eqref{eq:cvlimitingscheme} and of Proposition~\ref{propo:cv_scheme-limitingscheme}.
\end{proof}

\subsection{Proof of Proposition~\ref{propo:error_limitingscheme-averagedequation}}\label{sec:proof-AP2}

\begin{proof}
Let $\varphi:H\to \R$ be a mapping of class $\mathcal{C}^2$, with bounded first and second order derivatives, and let $\Delta t=T/N\in(0,\Delta t_0)$. Then the weak error in the left-hand side of~\eqref{eq:error_limitingscheme-averagedequation} can be decomposed as follows:
\[
\big|\E[\varphi(\XX_N^{\Delta t})]-\E[\varphi(\overline{\XX}(T))]\big|\le \big|\varphi(\overline{\XX}_N^{\Delta t})-\varphi(\overline{\XX}(T))\big|+\big|\E[\varphi(\XX_N^{\Delta t})]-\varphi(\overline{\XX}_N^{\Delta t})\big|,
\]
where $\overline{\XX}_N^{\Delta t}=\overline{\XX}_N^{\Delta t;x_0}$ is obtained by using the auxiliary scheme~\eqref{eq:auxiliaryscheme}, with initial value given by $\overline{\XX}_0^{\Delta t}=\XX_0^{\Delta t}=x_0$.

On the one hand, using the error estimate~\eqref{eq:error_auxiliaryscheme} from Proposition~\ref{propo:error_auxiliaryscheme} for the auxiliary scheme, one has
\begin{align*}
\big|\varphi(\overline{\XX}_N^{\Delta t})-\varphi(\overline{\XX}(T))\big|&\le \vvvert\varphi\vvvert_1|\overline{\XX}_N^{\Delta t}-\overline{\XX}(T)|\\
&\le C_\kappa(T)\Delta t^{1-\kappa}(1+|x_0|).
\end{align*}

On the other hand, the second error term can be written as follows, in terms of the auxiliary mappings $\overline{u}_n^{\Delta t}$ given by~\eqref{eq:ubar}, using a telescoping sum argument: one has
\begin{align*}
\E[\varphi(\XX_N^{\Delta t})]-\varphi(\overline{\XX}_N^{\Delta t})&=\E[\overline{u}_0^{\Delta t}(\XX_N^{\Delta t})]-\E[\overline{u}_N^{\Delta t}(\XX_0^{\Delta t})]\\
&=\sum_{n=0}^{N-1}\bigl(\E[\overline{u}_{N-n-1}^{\Delta t}(\XX_{n+1}^{\Delta t})]-\E[\overline{u}_{N-n}^{\Delta t}(\XX_n^{\Delta t})]\bigr)\\
&=\sum_{n=0}^{N-1}\bigl(\E[\overline{u}_{N-n-1}^{\Delta t}(\IA_{\Delta t}\XX_n^{\Delta t}+\Delta t \IA_{\Delta t} F(\XX_n^{\Delta t},\IL^{-\frac12}\Gamma_n))]\\
&\hspace{1cm}-\E[\overline{u}_{N-n-1}^{\Delta t}(\IA_{\Delta t}\XX_n^{\Delta t}+\Delta t\IA_{\Delta t} \overline{F}(\XX_n^{\Delta t}))]\bigr).
\end{align*}
Owing to Lemma~\ref{lem:ubarDeltat}, for all $n\in\{0,\ldots,N-1\}$, the mapping $\overline{u}_n^{\Delta t}$ is of class $\mathcal{C}^2$ and has a bounded first and second order derivatives. By a Taylor expansion argument, one obtains
\begin{align*}
\E[\overline{u}_{N-n-1}^{\Delta t}(\IA_{\Delta t}\XX_n^{\Delta t}&+\Delta t \IA_{\Delta t} F(\XX_n^{\Delta t},\IL^{-\frac12}\Gamma_n))]=\E[\overline{u}_{N-n-1}^{\Delta t}(\IA_{\Delta t}\XX_n^{\Delta t}+\Delta t\IA_{\Delta t} \overline{F}(\XX_n^{\Delta t}))]\\
&+\Delta t\E[\langle D\overline{u}_{N-n-1}^{\Delta t}(\IA_{\Delta t}\XX_n^{\Delta t}),\IA_{\Delta t} F(\XX_n^{\Delta t},\IL^{-\frac12}\Gamma_n)-\IA_{\Delta t}\overline{F}(\XX_n^{\Delta t})\rangle]+r_n^{\Delta t}
\end{align*}
where
\begin{align*}
|r_n^{\Delta t}|&\le C\vvvert \overline{u}_{N-n-1}^{\Delta t}\vvvert_2\Delta t^2 \E[|F(\XX_n^{\Delta t},\IL^{-\frac12}\Gamma_n)-\overline{F}(\XX_n^{\Delta t})|^2]\\
&\le C(T)\Delta t^2\vvvert\varphi\vvvert_2(1+\E[|\XX_n^{\Delta t}|^2])\\
&\le C(T)\Delta t^2\vvvert\varphi\vvvert_2 (1+|x_0|^2).
\end{align*}
using the inequality~\eqref{eq:lemubarDeltat_4}, the Lipschitz continuity of $F$, the moment bound~\eqref{eq:momentboundscheme} and the bound $\E[|\IL^{-\frac12}\Gamma_n|]=\int|y|d\nu(y)<\infty$. Moreover, by the definition~\eqref{eq:Fbar} of the nonlinearity $\overline{F}$, and since the random variables $\XX_n^{\Delta t}$ and $\Gamma_n$ are independent, a conditional expectation argument yields the identity
\begin{equation}\label{eq:crucialidentity}
\E[\langle D\overline{u}_{N-n-1}^{\Delta t}(\IA_{\Delta t}\XX_n^{\Delta t}),\IA_{\Delta t} F(\XX_n^{\Delta t},\IL^{-\frac12}\Gamma_n)-\IA_{\Delta t}\overline{F}(\XX_n^{\Delta t})\rangle]=0
\end{equation}
for all $n\in\{0,\ldots,N-1\}$. As a consequence, one obtains
\begin{equation}\label{eq:errorXbarX}
\big|\E[\varphi(\XX_N^{\Delta t})]-\varphi(\overline{\XX}_N^{\Delta t})\big|\le C(T)\Delta t\vvvert\varphi\vvvert_2 (1+|x_0|^2).
\end{equation}
Gathering the error estimates then concludes the proof of the inequality~\eqref{eq:error_limitingscheme-averagedequation} and of Proposition~\ref{propo:error_limitingscheme-averagedequation}.
\end{proof}

Note that the most fundamental argument in the proof of Proposition~\ref{propo:error_limitingscheme-averagedequation} is the identity~\eqref{eq:crucialidentity}. It explains both why the limiting scheme~\eqref{eq:limitingscheme} is consistent with the averaged equation~\eqref{eq:averaged}, and why it convergence in distribution is considered.

\section{Proofs of the error estimates}\label{sec:proof-UA}

\subsection{Proof of Proposition~\ref{propo:error_auxiliaryscheme}}\label{sec:proof-UA1}

As already explained in Section~\ref{sec:auxiliary-error}, the proof of Proposition~\ref{propo:error_auxiliaryscheme} is given below even if it is a standard result in numerical analysis of parabolic semilinear evolution equations. Providing a detailed proof allows us to exhibit the absence of regularity requirement for the initial value $x_0$. In the proofs, to simplify notation, let $\overline{\XX}_n=\overline{\XX_n}^{\Delta t;x_0}$.

Before proceeding with the proof, let us state auxiliary bounds for the solutions of the averaged equation~\eqref{eq:averaged} and of the auxiliary scheme~\eqref{eq:auxiliaryscheme}.
\begin{lemma}\label{lem:bounds_averaged-auxiliaryscheme}
For all $T\in(0,\infty)$ and $\kappa\in(0,1)$, there exists $C_\kappa(T)\in(0,\infty)$ such that for all $0<t_1<t_2\le T$, one has
\begin{equation}\label{eq:bound_averaged-increment}
|\overline{\XX}(t_2)-\overline{\XX}(t_1)|\le C_\kappa(T)(t_2-t_1)^{1-\kappa}(1+t_1^{-1+\kappa}|x_0|).
\end{equation}
Moreover, there exists $C\in(0,\infty)$ such that for all $n\in\N$ and $\Delta t\in(0,\Delta t_0)$, one has
\begin{equation}\label{eq:bound_auxiliaryscheme}
|\overline{\XX}_n^{\Delta t}|\le e^{Cn\Delta t}(1+|x_0|).
\end{equation}
\end{lemma}

\begin{proof}[Proof of Lemma~\ref{lem:bounds_averaged-auxiliaryscheme}]
Let us first prove the inequality~\eqref{eq:bound_averaged-increment}. Since $\overline{F}$ is globally Lipschitz continuous, for all $t\ge 0$, one has
\begin{align*}
|\overline{\XX}(t)|&\le |e^{-t\IL}x_0|+\int_{0}^{t}|e^{-(t-s)\IL}\overline{F}(\overline{\XX}(s))|ds\\
&\le |x_0|+C\int_{0}^{t}(1+|\overline{\XX}(s)|)ds.
\end{align*}
Applying Gronwall's lemma, one then obtains for all $t\ge 0$
\[
|\overline{\XX}(t)|\le e^{Ct}(1+|x_0|).
\]
Let $\kappa\in(0,1)$, using the inequality~\eqref{eq:smoothing}, one then has for all $t>0$
\begin{align*}
|\IL^{1-\kappa}\overline{\XX}(t)|&\le |\IL^{1-\kappa}e^{-t\IL}x_0|+\int_0^t |\IL^{1-\kappa}e^{-(t-s)\IL}\overline{F}(\overline{\XX}(s))|ds\\
&\le C_\kappa t^{-1+\kappa}|x_0|+C_\kappa \int_0^t (t-s)^{-1+\kappa}dse^{Ct}(1+|x_0|)\\
&\le C_\kappa(T)(1+t^{-1+\kappa}|x_0|).
\end{align*}

For all $0<t_1<t_2\le T$, using the inequality~\eqref{eq:smoothing}, one then has
\begin{align*}
|\overline{\XX}(t_2)-\overline{\XX}(t_1)|&\le |(e^{-(t_2-t_1)\IL}-I)\overline{\XX}(t_1)|+\int_{t_1}^{t_2}|e^{-(t_2-t)\IL}\overline{\XX}(t))|dt\\
&\le C_\kappa(T)(t_2-t_1)^{1-\kappa}|\IL^{1-\kappa}\overline{\XX}(t_1)|+(t_2-t_1)\bigl(1+\underset{0\le t\le T}\sup~|\overline{\XX}(t))|\bigr)\\
&\le C_\kappa(T)(t_2-t_1)^{1-\kappa}(1+t_1^{-1+\kappa}|x_0|).
\end{align*}
This concludes the proof of the inequality~\eqref{eq:bound_averaged-increment}. Let us now prove the inequality~\eqref{eq:bound_auxiliaryscheme}. Since $\overline{F}$ is globally Lipschitz continuous, for all $n\ge 0$, one has
\[
|\overline{\XX}_{n+1}|\le |\IA_{\Delta t}\overline{\XX}_n|+\Delta t|\IA_{\Delta t}\overline{F}(\overline{\XX}_n)|\le (1+C\Delta t)|\overline{\XX}_n|+C\Delta t.
\]
The inequality~\eqref{eq:bound_auxiliaryscheme} then follows from a straightforward argument. The proof of Lemma~\ref{lem:bounds_averaged-auxiliaryscheme} is thus completed.
\end{proof}

We are now in position to prove Proposition~\ref{propo:error_auxiliaryscheme}.
\begin{proof}[Proof of Proposition~\ref{propo:error_auxiliaryscheme}]
For all $n\ge 0$, with the notation $t_n=n\Delta t$, one has
\begin{align*}
\overline{\XX}_n&=\IA_{\Delta t}^n x_0+\Delta t\sum_{\ell=0}^{n-1}\IA_{\Delta t}^{n-\ell}\overline{F}(\overline{\XX}_\ell)\\
\overline{\XX}(n\Delta t)&=e^{-t_n\IL}x_0+\int_{0}^{t_n}e^{-(t_n-t)\IL}\overline{F}(\overline{\XX}(t))dt.
\end{align*}
For all $n\in\{0,\ldots,N\}$, set $e_n=|\overline{\XX}_n-\overline{\XX}(n\Delta t)|$. Using the expressions above, the error $e_n$ can be decomposed as follows: for all $n\in\{0,\ldots,N\}$
\[
e_n\le e_n^{(1)}+e_n^{(2)}+e_n^{(3)}+e_n^{(4)}+e_n^{(5)},
\]
where
\begin{align*}
e_n^{(1)}&=\big|(\IA_{\Delta t}^n-e^{-n\Delta t\IL})x_0|\\
e_n^{(2)}&=\Delta t\sum_{\ell=0}^{n-1}\big|\IA_{\Delta t}^{n-\ell}\bigl(\overline{F}(\overline{\XX}_\ell)-\overline{F}(\overline{\XX}(t_\ell))\bigr)\big|\\
e_n^{(3)}&=\Delta t\sum_{\ell=0}^{n-1}\big|\bigl(\IA_{\Delta t}^{n-\ell}-e^{-(t_n-t_{\ell})\IL}\bigr)\overline{F}(\overline{\XX}(t_\ell))\big|\\
e_n^{(4)}&=\int_{0}^{t_n}\big|\bigl(e^{-(t_n-t_\ell)\IL}-e^{-(t_n-t)\IL}\bigr)\overline{F}(\overline{\XX}(t_\ell))\big|dt\\
e_n^{(5)}&=\sum_{\ell=0}^{n-1}\int_{t_{\ell}}^{t_{\ell+1}}\big|e^{-(t_n-t)\IL}\bigl(\overline{F}(\overline{\XX}(t_\ell))-\overline{F}(\overline{\XX}(t))\bigr)\big|dt.
\end{align*}

$\bullet$ Recall the inequality
\begin{equation}\label{eq:parabolicscheme}
\underset{n\in\N}\sup~\underset{z\in(0,\infty)}\sup~n\big|\frac{1}{(1+z)^n}-e^{-nz}\big|<\infty.
\end{equation}
As a consequence, for all $\kappa\in(0,1)$, there exists $C_\kappa$ such that one obtains
\[
e_n^{(1)}=\big|(\IA_{\Delta t}^n-e^{-n\Delta t\IL})x_0|\le \frac{C_\kappa}{n}|x_0|\le C_\kappa\frac{\Delta t^{1-\kappa}}{(n\Delta t)^{1-\kappa}}|x_0|
\]
for all $n\in\{1,\ldots,N\}$.

$\bullet$ Using the global Lipschitz continuity property of $\overline{F}$, one obtains
\begin{align*}
e_n^{(2)}&=\Delta t\sum_{\ell=0}^{n-1}\big|\IA_{\Delta t}^{n-\ell}\bigl(\overline{\F}(\overline{\XX}_\ell)-\overline{F}(\overline{\XX}(t_\ell))\bigr)\big|\\
&\le C\Delta t\sum_{\ell=0}^{n-1}|\overline{\XX}_\ell-\overline{\XX}(t_{\ell})|\\
&\le C\Delta t\sum_{\ell=0}^{n-1}e_\ell.
\end{align*}

$\bullet$ To deal with the third term, using the inequality~\eqref{eq:parabolicscheme}: one has
\begin{align*}
e_n^{(3)}&=\Delta t\sum_{\ell=0}^{n-1}\big|\bigl(\IA_{\Delta t}^{n-\ell}-e^{-(t_n-t_{\ell})\IL}\bigr)\overline{F}(\overline{\XX}(t_\ell))\big|\\
&\le C\Delta t\sum_{\ell=0}^{n-1}\frac{1}{(n-k)}|\overline{F}(\overline{\XX}_\ell)|\\
&\le C_\kappa\Delta t^{1-\kappa} \Delta t\sum_{\ell=0}^{n-1}\frac{1}{\bigl((n-k)\Delta t\bigr)^{1-\kappa}} \bigl(1+\underset{\ell=0,\ldots,N}\sup~|\overline{\XX}_\ell|\bigr)\\
&\le C_\kappa(T)\Delta t^{1-\kappa}\bigl(1+|x_0|\bigr),
\end{align*}
using the global Lipschitz continuous property of $\overline{F}$, and the bound~\eqref{eq:bound_auxiliaryscheme} from Lemma~\ref{lem:bounds_averaged-auxiliaryscheme}.

$\bullet$ To deal with the fourth term, the identity $e^{-(t_n-t_\ell)\IL}-e^{-(t_n-t)\IL}=e^{-(t_n-t)\IL}(e^{-(t-t_\ell)\IL}-I)$ is combined with the inequalities~\eqref{eq:smoothing} and~\eqref{eq:regularity}, one has
\begin{align*}
e_n^{(4)}&=\int_{0}^{t_n}\big|\bigl(e^{-(t_n-t_\ell)\IL}-e^{-(t_n-t)\IL}\bigr)\overline{F}(\overline{\XX}(t_\ell))\big|dt\\
&\le C_\kappa\int_{0}^{t_n}\frac{(t-t_\ell)^{1-\kappa}}{(t_n-t)^{1-\kappa}}|\overline{F}(\overline{\XX}(t_\ell))|dt\\
&\le C_\kappa\Delta t^{1-\kappa}\int_{0}^{T}\frac{1}{t^{1-\kappa}}dt\bigl(1+\underset{\ell=0,\ldots,N}\sup~|\overline{\XX}_\ell|\bigr)\\
&\le C_\kappa(T)\Delta t^{1-\kappa}\bigl(1+|x_0|\bigr)
\end{align*}
using the global Lipschitz continuous property of $\overline{F}$, and the bound~\eqref{eq:bound_auxiliaryscheme} from Lemma~\ref{lem:bounds_averaged-auxiliaryscheme}.

$\bullet$ To deal with the fifth term, using the global Lipschitz continuity property of $\overline{F}$ and the inequality~\eqref{eq:bound_averaged-increment}, one has
\begin{align*}
e_n^{(5)}&=\sum_{\ell=0}^{n-1}\int_{t_{\ell}}^{t_{\ell+1}}\big|e^{-(t_n-t)\IL}\bigl(\overline{F}(\overline{\XX}(t_\ell))-\overline{F}(\overline{\XX}(t))\bigr)\big|dt\\
&\le C\sum_{\ell=0}^{n-1}\int_{t_{\ell}}^{t_{\ell+1}}|\overline{\XX}(t_\ell)-\overline{\XX}(t)|dt\\
&\le C_\kappa(T)\sum_{\ell=1}^{n-1}\int_{t_{\ell}}^{t_{\ell+1}}(t-t_\ell)^{1-\kappa}(1+t_\ell^{-1+\kappa}|x_0|)dt\\
&+C\Delta t(1+\underset{0\le t\le T}\sup~|\overline{\XX}(t)|)\\
&\le C_\kappa(T)\Delta t^{1-\kappa}(1+\Delta t\sum_{\ell=0}^{n-1}\frac{1}{(\ell\Delta t)^{1-\kappa}}|x_0|+C\Delta t(1+|x_0|)\\
&\le C_\kappa(T)\Delta t^{1-\kappa}(1+|x_0|\bigr).
\end{align*}

$\bullet$ Gathering the estimates then gives
\[
e_n\le \Delta t\sum_{\ell=0}^{n-1}e_\ell+C_\kappa(T)\Delta t^{1-\kappa}(1+\frac{1}{(n\Delta t)^{1-\kappa}}|x_0|),
\]
and applying the discrete Gronwall lemma then concludes the proof of the inequality~\eqref{eq:error_auxiliaryscheme}.
\end{proof}

\subsection{Proof of Proposition~\ref{propo:error_fixedepsilon}}\label{sec:proof-UA2}

\begin{proof}[Proof of Proposition~\ref{propo:error_fixedepsilon}]
Recall that the mapping $u^\epsilon$ is defined by~\eqref{eq:uepsilon}, and is the solution of the Kolmogorov equation~\eqref{eq:Kolmogorov} with initial value $u^\epsilon(0,x,y)=\varphi(x)$. Without loss of generality, it is assumed that $\vvvert\varphi\vvvert_3\le 1$ to simplify notation. Recall also that $T=N\Delta t$. The weak error is written and then decomposed as follows, using a standard telescoping sum argument:
\begin{align*}
\E[\varphi(\XX_N^{\epsilon,\Delta t})]&-\E[\varphi(\XX^\epsilon(T))]=\E[u^\epsilon(0,\XX_N^{\epsilon,\Delta t},\YY_N^{\epsilon,\Delta t})]-\E[u^\epsilon(T,\XX_0^{\epsilon,\Delta t},\YY_0^{\epsilon,\Delta t})]\\
&=\sum_{n=0}^{N-1}\bigl(\E[u^\epsilon(T-t_{n+1},\XX_{n+1}^{\epsilon,\Delta t},\YY_{n+1}^{\epsilon,\Delta t})]-\E[u^\epsilon(T-t_n,\XX_n^{\epsilon,\Delta t},\YY_n^{\epsilon,\Delta t})]\bigr)\\
&=\sum_{n=0}^{N-1}\bigl(\E[u^\epsilon(T-t_{n+1},\XX_{n+1}^{\epsilon,\Delta t},\YY_{n+1}^{\epsilon,\Delta t})]-\E[u^\epsilon(T-t_{n+1},\XX_{n}^{\epsilon,\Delta t},\YY_{n+1}^{\epsilon,\Delta t})]\bigr)\\
&+\sum_{n=0}^{N-1}\bigl(\E[u^\epsilon(T-t_{n+1},\XX_{n}^\epsilon,\YY_{n+1}^{\epsilon,\Delta t})]-\E[u^\epsilon(T-t_n,\XX_n^\epsilon,\YY_n^{\epsilon,\Delta t})]\bigr).
\end{align*}

On the one hand, using a Taylor expansion argument and Lemma~\ref{lem:uepsilon}, one has
\begin{align*}
\E[u^\epsilon(T-t_{n+1},\XX_{n+1}^{\epsilon,\Delta t},\YY_{n+1}^{\epsilon,\Delta t})]&=\E[u^\epsilon(T-t_{n+1},\XX_{n}^{\epsilon,\Delta t},\YY_{n+1}^{\epsilon,\Delta t})]\\
&+\E[\langle D_xu^\epsilon(T-t_{n+1},\XX_n^{\epsilon,\Delta t},\YY_{n+1}^{\epsilon,\Delta t}),\XX_{n+1}^{\epsilon,\Delta t}-\XX_n^{\epsilon,\Delta t}\rangle]+r_n^{\epsilon,\Delta t},
\end{align*}
where, owing to the regularity estimate~\eqref{eq:lemuepsilon-2} from Lemma~\ref{lem:uepsilon}, for all $n\in\{0,\ldots,N-2\}$, one has
\[
|r_n^{\epsilon,\Delta t}|\le \frac{C_\kappa(T)}{(T-t_{n+1})^{1-\kappa}}\vvvert\varphi\vvvert_2\E\bigl[|\IL^{-1+\kappa}(\XX_{n+1}^{\epsilon,\Delta t}-\XX_n^{\epsilon,\Delta t})||\XX_{n+1}^{\epsilon,\Delta t}-\XX_n^{\epsilon,\Delta t}|\bigr].
\]
When $n=0$, using the inequality~\eqref{eq:incrementsscheme-bis} from Lemma~\ref{lem:incrementsscheme} and the moment bound~\eqref{eq:boundXnepsilonDeltat} from Lemma~\ref{lem:momentboundsscheme-epsilonDeltat}, one obtains
\[
|r_0^{\epsilon,\Delta t}|\le C_\kappa(T)\vvvert\varphi\vvvert_2\Delta t^{1-\kappa}(1+|x_0|^2).
\]
When $n\in\{1,\ldots,N-2\}$, using the inequalities~\eqref{eq:incrementsscheme} and~\eqref{eq:incrementsscheme-bis} from Lemma~\ref{lem:incrementsscheme}, one obtains
\[
|r_n^{\epsilon,\Delta t}|\le \frac{C_\kappa(T)}{(n\Delta t)^{1-\kappa}(T-t_{n+1})^{1-\kappa}}\vvvert\varphi\vvvert_2\Delta t^{2(1-\kappa)}(1+|x_0|^2).
\]
The case $n=N-1$ is treated differently: using the regularity estimate~\eqref{eq:lemuepsilon-2} with $\kappa=1$ and the inequality~\eqref{eq:incrementsscheme} from Lemma~\ref{lem:incrementsscheme} one has
\begin{align*}
|r_{N-1}^{\epsilon,\Delta t}|&\le C(T)\vvvert\varphi\vvvert_2\E[|\XX_{N}^{\epsilon,\Delta t}-\XX_{N-1}^{\epsilon,\Delta t}|^2]\\
&\le \frac{C_\kappa(T)}{\bigl((N-1)\Delta t\bigr)^{1-\kappa}}\vvvert\varphi\vvvert_2\Delta t^{1-\kappa}(1+|x_0|^2)\\
&\le C_\kappa(T)\vvvert\varphi\vvvert_2\Delta t^{1-\kappa}(1+|x_0|^2),
\end{align*}
using the lower bound $(N-1)\Delta t=T-\Delta t\ge T-\Delta t_0$.

Gathering the estimates, one obtains
\begin{equation}\label{eq:rn}
\sum_{n=0}^{N-1}|r_n^{\epsilon,\Delta t}|\le C_\kappa(T)\vvvert\varphi\vvvert_2\Delta t^{1-2\kappa}(1+|x_0|^2).
\end{equation}

On the other hand, introduce the auxiliary process $\bigl(\tilde{\YY}^{\epsilon,\Delta t}(t)\bigr)_{t\ge 0}$ defined as the solution of the stochastic evolution equation
\begin{equation}\label{eq:tildeY}
d\tilde{\YY}^{\epsilon,\Delta t}(t)=-\frac{1}{\epsilon}\IL_{\frac{\Delta t}{\epsilon}}\tilde{\YY}^{\epsilon,\Delta t}(t)dt+\sqrt{\frac{{2}}{{\epsilon}}}Q_{\frac{\Delta t}{\epsilon}}^{\frac12}dW(t),
\end{equation}
with initial value $\tilde{\YY}^{\epsilon,\Delta t}(0)=y_0^\epsilon$, where the linear operators $\IL_\tau$ and $Q_\tau$ with $\tau=\Delta t/\epsilon$ are given by~\eqref{eq:linearoperators}. By construction, one checks that for all $n\in\N$ one has the equality in distribution
\begin{equation}\label{eq:equalitytildeY}
\tilde{\YY}^{\epsilon,\Delta t}(t_n)=\YY_n^{\epsilon,\Delta t}.
\end{equation}
The equality above is based on the interpretation of the modified Euler scheme as the accelerated exponential Euler scheme applied to a modified stochastic evolution equation, see Section~\ref{sec:scheme} and~\cite[Section~3.3]{B} for details. More precisely, one has $\tilde{\YY}^{\epsilon,\Delta t}(t)=\IY^\tau(\frac{t}{\epsilon})$ for all $t\ge 0$ and $\tilde{\YY}_n^{\epsilon,\Delta t}=\YY_n^{\epsilon,\Delta t}=\IY_n^\tau=\IY^\tau(t_n^\tau)$, with $t_n^\tau=\frac{n\Delta t}{\epsilon}=\frac{t_n}{\epsilon}$, where the processes $\bigl(\IY^\tau(t)\bigr)_{t\ge 0}$ and $\bigl(\IY_n^\tau\bigr)_{n\ge 0}$ are defined by~\eqref{eq:auxOU} and~\eqref{eq:scheme-interp} respectively.

Owing to Assumption~\ref{ass:init}, for all $\kappa\in(0,\kappa_0)$, one has
\begin{equation}\label{eq:boundtildeYYkappa}
\underset{\epsilon\in(0,\epsilon_0)}\sup~\underset{\Delta t\in(0,\Delta t_0)}\sup~\underset{t\ge 0}\sup~\E[|\IL^{\frac{\kappa}{2}}\tilde{\YY}^{\epsilon,\Delta t}(t)|^2]<\infty.
\end{equation}
The proof is a consequence of It\^o's isometry formula and straightforward computations, see~\cite[Lemma~5.3]{B} for details.

The mild solution of the auxiliary stochastic evolution equation~\eqref{eq:tildeY} has the expression
\begin{equation}\label{eq:mildtildeYY}
\tilde{\YY}^{\epsilon,\Delta t}(t)=e^{-\frac{t}{\epsilon}\IL_{\frac{\Delta t}{\epsilon}}}y_0^\epsilon+\sqrt{\frac{{2}}{{\epsilon}}}\int_{0}^{t}e^{-\frac{t-s}{\epsilon}\IL_{\frac{\Delta t}{\epsilon}}}Q_{\frac{\Delta t}{\epsilon}}^{\frac12}dW(s),
\end{equation}
for all $t\ge 0$.

As a consequence of the equality~\eqref{eq:equalitytildeY} and using It\^o's formula, one obtains
\begin{align*}
\E[u^\epsilon(T-t_{n+1},&\XX_{n}^{\epsilon,\Delta t},\YY_{n+1}^{\epsilon,\Delta t})]-\E[u^\epsilon(T-t_n,\XX_n^{\epsilon,\Delta t},\YY_n^{\epsilon,\Delta t})]\\
&=\E[u^\epsilon(T-t_{n+1},\XX_{n}^{\epsilon,\Delta t},\tilde{\YY}^{\epsilon,\Delta t}(t_{n+1}))]-\E[u^\epsilon(T-t_n,\XX_n^{\epsilon,\Delta t},\tilde{\YY}^{\epsilon,\Delta t}(t_n))]\\
&=-\int_{t_n}^{t_{n+1}}\E[\partial_tu^\epsilon(T-t,\XX_n^{\epsilon,\Delta t},\tilde{\YY}^{\epsilon,\Delta t}(t))]dt\\
&-\frac{1}{\epsilon}\int_{t_n}^{t_{n+1}}\E[\langle D_yu^\epsilon(T-t,\XX_n^{\epsilon,\Delta t},\tilde{\YY}^{\epsilon,\Delta t}(t)),\IL_{\frac{\Delta t}{\epsilon}}\tilde{\YY}^{\epsilon,\Delta t}(t)\rangle]dt\\
&+\frac{1}{\epsilon}\int_{t_n}^{t_{n+1}}\sum_{j\in\N}\E[D_y^2u^\epsilon(T-t,\XX_n^{\epsilon,\Delta t},\tilde{\YY}^{\epsilon,\Delta t}(t)).\bigl(Q_{\frac{\Delta t}{\epsilon}}e_j,e_j\bigr)]dt.
\end{align*}

Since $u^\epsilon$ solves the Kolmogorov equation~\eqref{eq:Kolmogorov}, one obtains the following decomposition of the error terms
\begin{equation}
\E[u^\epsilon(T-t_{n+1},\XX_{n+1}^{\epsilon,\Delta t},\YY_{n+1}^{\epsilon,\Delta t})]-\E[u^\epsilon(T-t_n,\XX_n^{\epsilon,\Delta t},\YY_n^{\epsilon,\Delta t})]=e_n^{0,\epsilon,\Delta t}+e_n^{1,\epsilon,\Delta t}+e_n^{2,\epsilon,\Delta t}+r_n^{\epsilon,\Delta t},
\end{equation}
where the error terms $e_n^{0,\epsilon,\Delta t}$, $=e_n^{1,\epsilon,\Delta t}$ and $=e_n^{2,\epsilon,\Delta t}$ for $n\in\{0,\ldots,N-1\}$ are defined by
\begin{equation}\label{eq:en0}
\begin{aligned}
e_n^{0,\epsilon,\Delta t}&=\E[\langle D_xu^\epsilon(T-t_{n+1},\XX_n^{\epsilon,\Delta t},\YY_{n+1}^{\epsilon,\Delta t}),\XX_{n+1}^{\epsilon,\Delta t}-\XX_n^{\epsilon,\Delta t}\rangle]\\
&-\int_{t_n}^{t_{n+1}}\E[\langle D_xu^\epsilon(T-t,\XX_n^{\epsilon,\Delta t},\tilde{\YY}^{\epsilon,\Delta t}(t)),-\IL\XX_n^{\epsilon,\Delta t}+F(\XX_n^{\epsilon,\Delta t},\tilde{\YY}^{\epsilon,\Delta t}(t))\rangle]dt
\end{aligned}
\end{equation}
and by
\begin{align}
e_n^{1,\epsilon,\Delta t}&=-\frac{1}{\epsilon}\int_{t_n}^{t_{n+1}}\E[\langle D_yu^\epsilon(T-t,\XX_n^{\epsilon,\Delta t},\tilde{\YY}^{\epsilon,\Delta t}(t)),(\IL_{\frac{\Delta t}{\epsilon}}-\IL)\tilde{\YY}^{\epsilon,\Delta t}(t)\rangle]dt\label{eq:en1}\\
e_n^{2,\epsilon,\Delta t}&=\frac{1}{\epsilon}\int_{t_n}^{t_{n+1}}\sum_{j\in\N}\E[D_y^2u^\epsilon(T-t,\XX_n^{\epsilon,\Delta t},\tilde{\YY}^{\epsilon,\Delta t}(t)).\bigl((Q_{\frac{\Delta t}{\epsilon}}-I)e_j,e_j\bigr)]dt.\label{eq:en2}
\end{align}
Before proceeding with the proof of upper bounds for the error terms, it is necessary to introduce a further decomposition for $e_n^{0,\epsilon,\Delta t}$. Using the equalities
\begin{align*}
\XX_{n+1}^{\epsilon,\Delta t}-\XX_n^{\epsilon,\Delta t}&=
\bigl(\IA_{\Delta t}-I)\XX_n^{\epsilon,\Delta t}+\Delta t\IA_{\Delta t}F(\XX_n^{\epsilon,\Delta t},\YY_{n+1}^{\epsilon,\Delta t})\\
&=\bigl(\IA_{\Delta t}-I+\Delta t\IL)\XX_n^{\epsilon,\Delta t}\\
&+\Delta t(\IA_{\Delta t}-I)F(\XX_n^{\epsilon,\Delta t},\YY_{n+1}^{\epsilon,\Delta t})\\
&+\Delta t\bigl(-\IL\XX_n^{\epsilon,\Delta t}+F(\XX_n^{\epsilon,\Delta t},\YY_{n+1}^{\epsilon,\Delta t})\bigr),
\end{align*}
the error term $e_n^0$ is decomposed as
\[
e_n^{0,\epsilon,\Delta t}=e_n^{0,1,\epsilon,\Delta t}+e_n^{0,2,\epsilon,\Delta t}+e_n^{0,3,\epsilon,\Delta t}+e_n^{0,4,\epsilon,\Delta t},
\]
where the error terms in the right-hand side of the expression above are defined for $n\in\{0,\ldots,N-1\}$ by
\begin{align}
e_n^{0,1,\epsilon,\Delta t}&=\E[\langle D_xu^\epsilon(T-t_{n+1},\XX_n^{\epsilon,\Delta t},\YY_{n+1}^{\epsilon,\Delta t}),(\IA_{\Delta t}-I+\Delta t\IL)\XX_n^{\epsilon,\Delta t}\rangle]\label{eq:en01}\\
e_n^{0,2,\epsilon,\Delta t}&=\E[\langle D_xu^\epsilon(T-t_{n+1},\XX_n^{\epsilon,\Delta t},\YY_{n+1}^{\epsilon,\Delta t}),\Delta t(\IA_{\Delta t}-I)F(\XX_n^{\epsilon,\Delta t},\YY_{n+1}^{\epsilon,\Delta t})\rangle]\label{eq:en02}\\
e_n^{0,3,\epsilon,\Delta t}&=\int_{t_n}^{t_{n+1}}\E[\langle D_xu^\epsilon(T-t,\XX_n^{\epsilon,\Delta t},\tilde{\YY}^{\epsilon,\Delta t}(t)),\IL\XX_n^{\epsilon,\Delta t}\rangle]dt\label{eq:en03}\\
&-\Delta t\E[\langle D_xu^\epsilon(T-t_{n+1},\XX_n^{\epsilon,\Delta t},\YY^{\epsilon,\Delta t}(t_{n+1})),\IL\XX_n^{\epsilon,\Delta t}\rangle]\\
e_n^{0,4,\epsilon,\Delta t}&=\int_{t_n}^{t_{n+1}}\E[\langle D_xu^\epsilon(T-t_{n+1},\XX_n^{\epsilon,\Delta t},\tilde{\YY}^{\epsilon,\Delta t}(t_{n+1})),F(\XX_n^{\epsilon,\Delta t},\tilde{\YY}^{\epsilon,\Delta t}(t_{n+1}))\rangle]dt\label{eq:en04}\\
&-\int_{t_n}^{t_{n+1}}\E[\langle D_xu^\epsilon(T-t,\XX_n^{\epsilon,\Delta t},\tilde{\YY}^{\epsilon,\Delta t}(t)),F(\XX_n^{\epsilon,\Delta t},\tilde{\YY}^{\epsilon,\Delta t}(t))\rangle]dt.\nonumber
\end{align}

The weak error estimate is then a straightforward consequence of the following inequalities, which are proved below (using the conditions from Assumption~\ref{ass:init} on the initial values $x_0^\epsilon$ and $y_0^\epsilon$) there exists $C_\kappa(T)\in(0,\infty)$, such that for all $\epsilon\in(0,\epsilon_0)$ and $\Delta t\in(0,\Delta t_0)$, one has
\begin{align}
\sum_{n=0}^{N-1}|e_n^{1,\epsilon,\Delta t}|&\le C_\kappa(T)\Delta t(1+|\IL^{\frac{\kappa}{2}}x_0|)+C_\kappa(T)\bigl(\frac{\Delta t}{\epsilon}\bigr)^{\frac12-\kappa}(1+|\IL^{\frac{\kappa}{2}}x_0|)\label{eq:ineqen1}\\
\sum_{n=0}^{N-1}|e_n^{2,\epsilon,\Delta t}|&\le C_\kappa(T)\bigl(\frac{\Delta t}{\epsilon}\bigr)^{\frac12-\kappa}\label{eq:ineqen2}\\
\sum_{n=0}^{N-1}|e_n^{0,1,\epsilon,\Delta t}|&\le C_\kappa(T)\Delta t(1+|\IL^{\frac{\kappa}{2}}x_0|)+C_\kappa(T)\Delta t^{1-\kappa}(1+|x_0|)\label{eq:ineqen01}\\
\sum_{n=0}^{N-1}|e_n^{0,2,\epsilon,\Delta t}|&\le C_\kappa(T)\Delta t^{1-\kappa}(1+|x_0|)\label{eq:ineqen02}\\
\sum_{n=0}^{N-1}|e_n^{0,3,\epsilon,\Delta t}|&\le C_\kappa(T)\Delta t(1+|\IL^{\frac{\kappa}{2}} x_0|)+C_\kappa(T)\Bigl(\frac{\Delta t}{\epsilon}\Bigr)^{1-\kappa}\bigl(1+|\IL^{\frac{\kappa}{2}}x_0|^2\bigr)\label{eq:ineqen03}\\
\sum_{n=0}^{N-1}|e_n^{0,4,\epsilon,\Delta t}|&\le C_\kappa(T)\Bigl(\bigl(\frac{\Delta t}{\epsilon})^{\frac12-\kappa}+\bigl(\frac{\Delta t}{\epsilon})^{1-\kappa}+\frac{\Delta t}{\epsilon}\Bigr)\bigl(1+|\IL^{\frac{\kappa}{2}}x_0|^2\bigr).\label{eq:ineqen04}
\end{align}
Gathering the inequalities above then concludes the proof of Proposition~\ref{propo:error_fixedepsilon}.
\end{proof}

It remains to prove the inequalities~\eqref{eq:ineqen1}-\eqref{eq:ineqen04}. To simplify the notation, in the proofs below, the parameters $\epsilon,\Delta t$ are omitted when refering to the error terms defined above, or to other error terms introduced below: for instance one has $e_n^{1}=e_n^{1,\epsilon,\Delta t}$ in the proof of the inequality~\eqref{eq:ineqen1}.

\begin{proof}[Proof of the inequality~\eqref{eq:ineqen1}]
Recall that the error term $e_n^{1}=e_n^{1,\epsilon,\Delta t}$ is defined by~\eqref{eq:en1} for all $n\in\{0,\ldots,N-1\}$. Using~\eqref{eq:mildtildeYY}, the error term $e_n^1$ is decomposed as
\[
e_n^1=e_n^{1,1}+e_n^{1,2}
\]
with
\begin{align*}
e_n^{1,1}&=\frac{1}{\epsilon}\int_{t_n}^{t_{n+1}}\E[\langle D_yu^\epsilon(T-t,\XX_n^{\epsilon,\Delta t},\tilde{\YY}^{\epsilon,\Delta t}(t)),(\IL-\IL_{\frac{\Delta t}{\epsilon}})e^{-\frac{t}{\epsilon}\IL_{\frac{\Delta t}{\epsilon}}}y_0^\epsilon\rangle]dt\\
e_n^{1,2}&=\frac{1}{\epsilon}\int_{t_n}^{t_{n+1}}\E[\langle D_yu^\epsilon(T-t,\XX_n^{\epsilon,\Delta t},\tilde{\YY}^{\epsilon,\Delta t}(t)),(\IL-\IL_{\frac{\Delta t}{\epsilon}})\sqrt{\frac{{2}}{{\epsilon}}}\int_{0}^{t}e^{-\frac{t-s}{\epsilon}\IL_{\frac{\Delta t}{\epsilon}}}Q_{\frac{\Delta t}{\epsilon}}^{\frac12}dW(s)\rangle]dt.
\end{align*}

$\bullet$ {\it Error term $e_n^{1,1}$.}Owing to the regularity estimate~\eqref{eq:lemuepsilon-1} from Lemma~\ref{lem:uepsilon}, for all $n\in\{0,\ldots,N-1\}$, one has
\begin{align*}
|e_n^{1,1}|&\le \int_{t_n}^{t_{n+1}}\frac{C_\kappa(T)}{(T-t)^{1-\frac{\kappa}{2}}}\big|\IL^{-1+\frac{\kappa}{2}}(\IL-\IL_{\frac{\Delta t}{\epsilon}})e^{-\frac{t}{\epsilon}\IL_{\frac{\Delta t}{\epsilon}}}y_0^\epsilon\big|dt.
\end{align*}
The cases $n=0$ and $n\in\{1,\ldots,N-1\}$ are treated differently. On the one hand, owing to Assumption~\ref{ass:init}, one obtains
\[
|e_0^{1,1}|\le \int_{t_0}^{t_1}\frac{C_\kappa(T)}{(T-t)^{1-\frac{\kappa}{2}}}dt|\IL^{\frac{\kappa}{2}}y_0^\epsilon|\le C_\kappa(T)\Delta t(1+|\IL^{\frac{\kappa}{2}}x_0|).
\]
On the other hand, owing to the inequality (see~\cite[Lemma~5.1]{B})
\[
\underset{\tau\in(0,\infty)}\sup~\underset{t\in(\tau,\infty)}\sup~(t-\tau)^{\frac12-\frac{\kappa}{2}} \|\IL^{\frac12-\frac{\kappa}{2}} e^{-t\IL_\tau}\|_{\mathcal{L}(H)}<\infty,
\]
and to the inequalities~\eqref{eq:error-eigenvalues}, for all $n\in\{1,\ldots,N-1\}$, one obtains
\begin{align*}
|e_n^{1,1}|&\le \int_{t_n}^{t_{n+1}}\frac{C_\kappa(T)}{(T-t)^{1-\frac{\kappa}{2}}}\frac{\epsilon^{\frac12-\frac{\kappa}{2}}}{(t-\Delta t)^{\frac12-\frac{\kappa}{2}}}dt |\IL^{-\frac32+\kappa}(\IL-\IL_{\frac{\Delta t}{\epsilon}})y_0^\epsilon|\\
&\le \int_{t_n}^{t_{n+1}}\frac{C_\kappa(T)}{(T-t)^{1-\frac{\kappa}{2}}}\frac{\epsilon^{1-\frac{\kappa}{2}}}{(t-\Delta t)^{1-\frac{\kappa}{2}}}dt\frac{\Delta t^{\frac12-\kappa}}{\epsilon^{\frac12-\kappa}}|\IL^{\frac{\kappa}{2}}y_0^\epsilon|\\
&\le \int_{t_n}^{t_{n+1}}\frac{C_\kappa(T)}{(T-t)^{1-\frac{\kappa}{2}}}\frac{1}{(t-\Delta t)^{1-\frac{\kappa}{2}}}dt\Delta t^{\frac12-\kappa}(1+|\IL^{\frac{\kappa}{2}}x_0|).
\end{align*}
using Assumption~\ref{ass:init}.

$\bullet$ {\it Error term $e_n^{1,2}$.} Using the Malliavin integration by parts formula~\eqref{eq:MalliavinIBP}, one obtains
\begin{align*}
e_n^{1,2}&=\int_{t_n}^{t_{n+1}}\sum_{j\in\N}\E[\langle D_yu^\epsilon(T-t,\XX_n^{\epsilon,\Delta t},\tilde{\YY}^{\epsilon,\Delta t}(t)),e_j\rangle \int_{0}^{t}e^{-\frac{t-s}{\epsilon}\lambda_{\frac{\Delta t}{\epsilon},j}}d\beta_j(s)]dt \frac{(\lambda_j-\lambda_{\frac{\Delta t}{\epsilon},j})}{\epsilon}\sqrt{\frac{{2q_{\frac{\Delta t}{\epsilon},j}}}{{\epsilon}}}\\
&=\frac{1}{\epsilon}\int_{t_n}^{t_{n+1}}\int_0^t\sum_{j\in\N}\E[\mathcal{D}_s^{e_j}\bigl(\langle D_yu^\epsilon(T-t,\XX_n^{\epsilon,\Delta t},\tilde{\YY}^{\epsilon,\Delta t}(t)),e_j\rangle\bigr)]\mathcal{I}_j^{\epsilon,\Delta t}(t,s)dsdt,
\end{align*}
where for all $t\ge 0$ and $j\in\N$ one has
\[
\mathcal{I}_j^{\epsilon,\Delta t}(t,s)=e^{-\frac{t-s}{\epsilon}\lambda_{\frac{\Delta t}{\epsilon},j}} (\lambda_j-\lambda_{\frac{\Delta t}{\epsilon},j})\sqrt{\frac{{2q_{\frac{\Delta t}{\epsilon},j}}}{{\epsilon}}}\ge 0.
\]
Note that for all $t\ge 0$ and $j\in\N$ one has
\[
\int_0^t\mathcal{I}_j^{\epsilon,\Delta t}(t,s)ds\le \frac{(\lambda_j-\lambda_{\frac{\Delta t}{\epsilon},j})}{\lambda_{\frac{\Delta t}{\epsilon},j}}\sqrt{2q_{\frac{\Delta t}{\epsilon},j}\epsilon}.
\]
Using the chain rule, one obtains
\[
e_n^{1,2}=e_n^{1,2,1}+e_n^{1,2,2},
\]
for all $n\in\{0,\ldots,N-1\}$, where
\begin{align*}
e_n^{1,2,1}&=\frac{1}{\epsilon}\int_{t_n}^{t_{n+1}}\int_0^t\sum_{j\in\N}\E[D_xD_yu^\epsilon(T-t,\XX_n^{\epsilon,\Delta t},\tilde{\YY}^{\epsilon,\Delta t}(t)).\bigl(\mathcal{D}_s^{e_j}\XX_n^{\epsilon,\Delta t},e_j\bigr)]\mathcal{I}_j^{\epsilon,\Delta t}(t,s)dsdt\\
e_n^{1,2,2}&=\frac{1}{\epsilon}\int_{t_n}^{t_{n+1}}\int_0^t\sum_{j\in\N}\E[D_y^2u^\epsilon(T-t,\XX_n^{\epsilon,\Delta t},\tilde{\YY}^{\epsilon,\Delta t}(t)).\bigl(\mathcal{D}_s^{e_j}\tilde{\YY}^{\epsilon,\Delta t}(t),e_j\bigr)]\mathcal{I}_j^{\epsilon,\Delta t}(t,s)dsdt.
\end{align*}

For all $h\in H$ and $t\ge s\ge 0$, the random variable $\mathcal{D}_s^{h}\tilde{\YY}^{\epsilon,\Delta t}(t)$ satisfies
\[
\mathcal{D}_s^{h}\tilde{\YY}^{\epsilon,\Delta t}(t)=\sqrt{\frac{{2}}{{\epsilon}}}e^{-\frac{t-s}{\epsilon}\IL_{\frac{\Delta t}{\epsilon}}}Q_{\frac{\Delta t}{\epsilon}}^{\frac12}h.
\]
In particular, one obtains the inequality
\begin{equation}\label{eq:DsY}
|\mathcal{D}_s^{e_j}\tilde{\YY}^{\epsilon,\Delta t}(t)|\le \sqrt{\frac{{2q_{\frac{\Delta t}{\epsilon},j}}}{{\epsilon}}}
\end{equation}
for all $t\ge s\ge 0$ and $j\in\N$. Using the chain rule and the definition~\eqref{eq:scheme} of the scheme, for all $h\in H$, if $s\ge t_n$, one has
\[
\mathcal{D}_s^h\XX_{n+1}^{\epsilon,\Delta t}=\IA_{\Delta t}\mathcal{D}_s^h\XX_{n}^{\epsilon,\Delta t}+\Delta t\IA_{\Delta t}D_xF(\XX_n^{\epsilon,\Delta t},\YY_{n+1}^{\epsilon,\Delta t}).\mathcal{D}_s^h\XX_n^{\epsilon,\Delta t}+\Delta t\IA_{\Delta t}D_yF(\XX_n^{\epsilon,\Delta t},\YY_{n+1}^{\Delta t}).\mathcal{D}_s^h\YY_{n+1}^{\epsilon,\Delta t}
\]
and $\mathcal{D}_s^h\XX_n^{\epsilon,\Delta t}=0$ if $t_n<s$. Using the identity $\YY_{n+1}^{\epsilon,\Delta t}=\tilde{\YY}^{\epsilon,\Delta t}(t_n)$, the inequality~\eqref{eq:DsY} above, the boundedness of $D_xF$ and $D_yF$ (Assumption~\ref{ass:F}), one obtains the upper bound
\begin{equation}\label{eq:DsX}
|\mathcal{D}_s^{e_j}\XX_n^{\epsilon,\Delta t}|\le C(T)\sqrt{\frac{{2q_{\frac{\Delta t}{\epsilon},j}}}{{\epsilon}}}
\end{equation}
for all $s\in[0,T]$, $n\in\{0,\ldots,N-1\}$ and $j\in\N$.

Using the regularity estimates~\eqref{eq:lemuepsilon-2} from Lemma~\ref{lem:uepsilon} with $\alpha_1=0$ and $\alpha_2=1-\frac{\kappa}{2}$, and the inequalities~\eqref{eq:DsY} and~\eqref{eq:DsX}, one thus obtains the upper bound
\begin{align*}
|e_n^{1,2,1}|+|e_n^{1,2,2}|&\le \int_{t_n}^{t_{n+1}}\frac{C_\kappa(T)}{(T-t)^{1-\frac{\kappa}{2}}}\sum_{j\in\N}\sqrt{\frac{{2q_{\frac{\Delta t}{\epsilon},j}}}{{\epsilon}}}\lambda_j^{-1+\frac{\kappa}{2}}\int_0^t\mathcal{I}_j^{\epsilon,\Delta t}(t,s)dsdt\\
&\le \int_{t_n}^{t_{n+1}}\frac{C_\kappa(T)}{(T-t)^{1-\frac{\kappa}{2}}}dt\sum_{j\in\N}\frac{q_{\frac{\Delta t}{\epsilon},j}}{\lambda_{\frac{\Delta t}{\epsilon},j}}\lambda_j^{-1+\frac{\kappa}{2}}(\lambda_j-\lambda_{\frac{\Delta t}{\epsilon},j}).
\end{align*}
Using the identity $q_{\frac{\Delta t}{\epsilon},j}=\frac{\lambda_{\frac{\Delta t}{\epsilon},j}}{\lambda_j}$ and the inequality~\eqref{eq:error-eigenvalues} (with $\alpha=\frac12-\kappa$), one then obtains, for all $n\in\{0,\ldots,N-1\}$, the upper bound
\begin{align*}
|e_n^{1,2}|\le |e_n^{1,2,1}|+|e_n^{1,2,2}|\le \int_{t_n}^{t_{n+1}}\frac{C_\kappa(T)}{(T-t)^{1-\frac{\kappa}{2}}}dt \bigl(\frac{\Delta t}{\epsilon}\bigr)^{\frac12-\kappa}\sum_{j\in\N}\lambda_j^{-\frac12-\frac{\kappa}{2}},
\end{align*}
with $\sum_{j\in\N}\lambda_j^{-\frac12-\frac{\kappa}{2}}<\infty$.

Gathering the estimates for the error terms $e_n^{1,1}$ and $e_n^{1,2}$ and summing for $n\in\{0,\ldots,N-1\}$ then concludes the proof of the inequality~\eqref{eq:ineqen1} for the error term $e_n^1=e_n^{1,1}+e_n^{1,2}$.
\end{proof}

\begin{proof}[Proof of the inequality~\eqref{eq:ineqen2}]
Recall that the error term $e_n^{2}=e_n^{2,\epsilon,\Delta t}$ is defined by~\eqref{eq:en2} for all $n\in\{0,\ldots,N-1\}$. Owing to the regularity estimate~\eqref{eq:lemuepsilon-2} from Lemma~\ref{lem:uepsilon} and to the inequalities~\eqref{eq:error-eigenvalues} (with $\alpha=\frac12-\kappa$), for all $n\in\{0,\ldots,N-1\}$, one has
\begin{align*}
|e_n^2|&\le\frac{1}{\epsilon}\int_{t_n}^{t_{n+1}}\sum_{j\in\N}\E[\Big|D_y^2u^\epsilon(T-t,\XX_n^{\epsilon,\Delta t},\tilde{\YY}^{\epsilon,\Delta t}(t)).\bigl((Q_{\frac{\Delta t}{\epsilon}}-I)e_j,e_j\bigr)\Big|]dt\\
&\le \int_{t_n}^{t_{n+1}}\frac{C_\kappa(T)}{(T-t)^{1-\frac{\kappa}{2}}}dt \sum_{j\in\N}\frac{1-q_{\frac{\Delta t}{\epsilon},j}}{\lambda_j^{1-\frac{\kappa}{2}}}\\
&\le \int_{t_n}^{t_{n+1}}\frac{C_\kappa(T)}{(T-t)^{1-\frac{\kappa}{2}}}dt \bigl(\frac{\Delta t}{\epsilon}\bigr)^{\frac12-\kappa}\sum_{j\in\N}\frac{1}{\lambda_j^{1+\frac{\kappa}{2}}}\\
&\le \int_{t_n}^{t_{n+1}}\frac{C_\kappa(T)}{(T-t)^{1-\frac{\kappa}{2}}}dt\bigl(\frac{\Delta t}{\epsilon}\bigr)^{\frac12-\kappa}.
\end{align*}
Summing for $n\in\{0,\ldots,N-1\}$ then concludes the proof of the inequality~\eqref{eq:ineqen2}.
\end{proof}

\begin{proof}[Proof of the inequality~\eqref{eq:ineqen01}]
Recall that the error term $e_n^{0,1}e=n^{0,1,\epsilon,\Delta t}$ is defined by~\eqref{eq:en01} for all $n\in\{0,\ldots,N-1\}$. The cases $n\in\{1,\ldots,N-2\}$ and $n\in\{0,N-1\}$ are treated differently.

On the one hand, if $n\in\{0,\ldots,N-2\}$, owing to the regularity estimate~\eqref{eq:lemuepsilon-1} from Lemma~\ref{lem:uepsilon} (with $\alpha=1-\frac{\kappa}{2}$), one obtains
\begin{align*}
|e_n^{0,1}|&\le \frac{C_\kappa(T)}{(T-t_{n+1})^{1-\frac{\kappa}{2}}}\E[|\IL^{-1+\frac{\kappa}{2}}(\IA_{\Delta t}-I+\Delta t\IL)\XX_n^{\epsilon,\Delta t}|]\\
&\le \frac{C_\kappa(T)}{(T-t_{n+1})^{1-\frac{\kappa}{2}}}\E[|\IL^{-2+\kappa}(\IA_{\Delta t}-I+\Delta t\IL)\IL^{1-\frac{\kappa}{2}}\XX_n^{\epsilon,\Delta t}|]\\
&\le \frac{C_\kappa(T)\Delta t}{(T-t_{n+1})^{1-\frac{\kappa}{2}}}\Delta t^{1-\kappa}\E[|\IL^{1-\frac{\kappa}{2}}\XX_n^{\epsilon,\Delta t}|].
\end{align*}
Using the moment bound~\eqref{eq:momentschemepower}, if $n\in\{1,\ldots,N-2\}$, one obtains
\[
|e_n^{0,1}|\le \frac{C_\kappa(T)\Delta t}{(T-t_{n+1})^{1-\frac{\kappa}{2}}}\Delta t^{1-\kappa}\frac{1}{t_n^{1-\frac{\kappa}{2}}}(1+|x_0|).
\]
If $n=0$, using Assumption~\ref{ass:init} one obtains
\begin{align*}
|e_0^{0,1}|&\le \frac{C_\kappa(T)\Delta t}{(T-t_{1})^{1-\frac{\kappa}{2}}}|\IL^{-1+\frac{\kappa}{2}}(\IA_{\Delta t}-I+\Delta t\IL)x_0^{\epsilon}|\\
&\le C_\kappa(T)\Delta t|\IL^{\frac{\kappa}{2}}x_0^\epsilon|\\
&\le C_\kappa(T)\Delta t(1+|\IL^{\frac{\kappa}{2}}x_0|).
\end{align*}

On the other hand, if $n=N-1$, owing to the regularity estimate~\eqref{eq:lemuepsilon-1} from Lemma~\ref{lem:uepsilon} (with $\alpha=0$), one obtains
\begin{align*}
|e_{N-1}^{0,1}|&\le C(T)\E[|(\IA_{\Delta t}-I+\Delta t\IL)\XX_{N-1}^{\epsilon,\Delta t}|]\\
&\le C(T)\Delta t\E[|\IL\XX_{N-1}^{\epsilon,\Delta t}|].
\end{align*}
Using the identity
\[
\XX_{N-1}^{\epsilon,\Delta t}=\IA_{\Delta t}^{N-1} x_0^\epsilon+\Delta t\sum_{\ell=0}^{N-2}\IA_{\Delta t}^{N-1-\ell}F(\XX_\ell^{\epsilon,\Delta t},\YY_\ell^{\epsilon,\Delta t}),
\]
see the proof of Lemma~\ref{lem:incrementsscheme}, and the moment bounds~\eqref{eq:boundXnepsilonDeltat} and~\eqref{eq:boundYnepsilonDeltat} from Lemma~\ref{lem:momentboundsscheme-epsilonDeltat}, one obtains (with $(N-1)\Delta t=T-\Delta t\ge T-\Delta t_0$)
\[
\E[|\IL\XX_{N-1}^{\epsilon,\Delta t}|]\le \frac{C}{(N-1)\Delta t}|x_0^\epsilon|+\Delta t\sum_{\ell=0}^{N-2}\frac{C}{(N-1-\ell)\Delta t}(1+|x_0|)\le C(T)\Delta t^{-\kappa}(1+|x_0|),
\]
which gives
\[
|e_{N-1}^{0,1}|\le C(T)\Delta t^{1-\kappa}(1+|x_0|).
\]

Gathering the estimates and summing for $n\in\{0,\ldots,N-1\}$ then concludes the proof of the inequality~\eqref{eq:ineqen01}.
\end{proof}

\begin{proof}[Proof of the inequality~\eqref{eq:ineqen02}]
Recall that the error term $e_n^{0,2}=e_n^{0,2,\epsilon,\Delta t}$ is defined by~\eqref{eq:en02} for all $n\in\{0,\ldots,N-1\}$. The cases $n\in\{0,\ldots,N-2\}$ and $n=N-1$ are treated differently.

On the one hand, if $n\in\{0,\ldots,N-2\}$, owing to the regularity estimate~\eqref{eq:lemuepsilon-1} from Lemma~\ref{lem:uepsilon} (with $\alpha=1-\kappa$), one obtains
\begin{align*}
|e_n^{0,2}|&\le \frac{C_\kappa(T)\Delta t}{(T-t_{n+1})^{1-\kappa}}\E[|\IL^{-1+\kappa}(\IA_{\Delta t}-I)F(\XX_n^{\epsilon,\Delta t},\YY_{n+1}^{\epsilon,\Delta t})|]\\
&\le \frac{C_\kappa(T)\Delta t}{(T-t_{n+1})^{1-\kappa}}\Delta t^{1-\kappa}\E[|F(\XX_n^{\epsilon,\Delta t},\YY_{n+1}^{\epsilon,\Delta t})|]\\
&\le \frac{C_\kappa(T)\Delta t}{(T-t_{n+1})^{1-\kappa}}\Delta t^{1-\kappa}(1+|x_0|),
\end{align*}
using the moment bounds~\eqref{eq:boundXnepsilonDeltat} and~\eqref{eq:boundYnepsilonDeltat} from Lemma~\ref{lem:momentboundsscheme-epsilonDeltat}.

On the other hand, if $n=N-1$, owing to the regularity estimate~\eqref{eq:lemuepsilon-1} from Lemma~\ref{lem:uepsilon} (with $\alpha=0$), one obtains
\[
|e_{N-1}^{0,2}|\le C(T)\Delta t\E[|(\IA_{\Delta t}-I)F(\XX_{N-1}^{\epsilon,\Delta t},\YY_N^{\epsilon,\Delta t})|]\le C(T)\Delta t(1+|x_0|),
\]
using the moment bounds~\eqref{eq:boundXnepsilonDeltat} and~\eqref{eq:boundYnepsilonDeltat} from Lemma~\ref{lem:momentboundsscheme-epsilonDeltat}.

Gathering the estimates and summing for $n\in\{0,\ldots,N-1\}$ then concludes the proof of the inequality~\eqref{eq:ineqen02}.
\end{proof}

\begin{proof}[Proof of the inequality~\eqref{eq:ineqen03}]
Recall that the error term $e_n^{0,3}=e_n^{0,3,\epsilon,\Delta t}$ is defined by~\eqref{eq:en03} for all $n\in\{0,\ldots,N-1\}$. The proof of the inequality~\eqref{eq:ineqen03} requires more delicate arguments than the proofs of the inequalities obtained above. The cases $n\in\{1,\ldots,N-1\}$ and $n=0$ are treated differently.

Introduce the auxiliary mapping $v^{\epsilon,\Delta t}:[0,T]\times H^2\to\R$ defined as follows: for all $t\in[0,T]$, $x,y\in H$, set
\begin{equation}
v^{\epsilon,\Delta t}(t,x,y)=\langle D_xu^\epsilon(T-t,x,y),e^{-\frac{\Delta t}{\epsilon}\IL}\IL x\rangle.
\end{equation}
For all $n\in\{1,\ldots,N-1\}$, the error term $e_n^{0,3}$ can then be decomposed as
\[
e_n^{0,3}=e_n^{0,3,1}+e_n^{0,3,2}+e_n^{0,3,3},
\]
where
\begin{align*}
e_n^{0,3,1}&=\int_{t_n}^{t_{n+1}}\E[\langle D_xu^\epsilon(T-t,\XX_n^{\epsilon,\Delta t},\tilde{\YY}^{\epsilon,\Delta t}(t)),(I-e^{-\frac{\Delta t}{\epsilon}\IL})\IL\XX_n^{\epsilon,\Delta t}\rangle] dt\\
e_n^{0,3,2}&=\Delta t\E[\langle D_xu^\epsilon(T-t_{n+1},\XX_n^{\epsilon,\Delta t},\tilde{\YY}^{\epsilon,\Delta t}(t_{n+1})),(I-e^{-\frac{\Delta t}{\epsilon}\IL})\IL\XX_n^{\epsilon,\Delta t}\rangle]\\
e_n^{0,3,3}&=\int_{t_n}^{t_{n+1}}\bigl(\E[v^{\epsilon,\Delta t}(t,\XX_n^{\epsilon,\Delta t},\tilde{\YY}^{\epsilon,\Delta t}(t))]-\E[v^{\epsilon,\Delta t}(t_{n+1},\XX_n^{\epsilon,\Delta t},\tilde{\YY}^{\epsilon,\Delta t}(t_{n+1}))]\bigr)dt.
\end{align*}

$\bullet$ {\it Error term $e_n^{0,3,1}$.} Owing to the regularity estimate~\eqref{eq:lemuepsilon-1} from Lemma~\ref{lem:uepsilon} (with $\alpha=1-\frac{\kappa}{2}$), to the inequality~\eqref{eq:regularity} and to the moment bound~\eqref{eq:momentschemepower} from Lemma~\ref{lem:incrementsscheme}, one obtains for all $n\in\{1,\ldots,N-1\}$
\begin{align*}
|e_n^{0,3,1}|&\le \int_{t_n}^{t_{n+1}}\frac{C_\kappa(T)}{(T-t)^{1-\frac{\kappa}{2}}}dt\E[|(I-e^{-\frac{\Delta t}{\epsilon}\IL})\IL^{\frac{\kappa}{2}}\XX_n^{\epsilon,\Delta t}|]\\
&\le \int_{t_n}^{t_{n+1}}\frac{C_\kappa(T)}{(T-t)^{1-\kappa}}dt\bigl(\frac{\Delta t}{\epsilon}\bigr)^{1-\kappa}\E[|\IL^{1-\frac{\kappa}{2}}\XX_n^{\epsilon,\Delta t}|]\\
&\le \int_{t_n}^{t_{n+1}}\frac{C_\kappa(T)}{(T-t)^{1-\kappa}}dt\bigl(\frac{\Delta t}{\epsilon}\bigr)^{1-\kappa} \frac{1}{(n\Delta t)^{1-\frac{\kappa}{2}}}(1+|x_0|).
\end{align*}

$\bullet$ {\it Error term $e_n^{0,3,2}$.} The cases $n\in\{1,\ldots,N-2\}$ and $n=N-1$ are treated differently.

On the one hand, owing to the regularity estimate~\eqref{eq:lemuepsilon-1} from Lemma~\ref{lem:uepsilon} (with $\alpha=1-\frac{\kappa}{2}$), to the inequality~\eqref{eq:regularity} and to the moment bound~\eqref{eq:momentschemepower} from Lemma~\ref{lem:incrementsscheme} one obtains for all $n\in\{1,\ldots,N-2\}$
\begin{align*}
|e_n^{0,3,2}|&\le \frac{C_\kappa(T)\Delta t}{(T-t_{n+1})^{1-\frac{\kappa}{2}}}\E[|(I-e^{-\frac{\Delta t}{\epsilon}\IL})\IL^{\frac{\kappa}{2}}\XX_n^{\epsilon,\Delta t}|]\\
&\le \int_{t_n}^{t_{n+1}}\frac{C_\kappa(T)}{(T-t)^{1-\kappa}}dt\bigl(\frac{\Delta t}{\epsilon}\bigr)^{1-\kappa} \frac{1}{(n\Delta t)^{1-\frac{\kappa}{2}}}(1+|x_0|).
\end{align*}
On the other hand, owing to the regularity estimate~\eqref{eq:lemuepsilon-1} from Lemma~\ref{lem:uepsilon} (with $\alpha=0$), one obtains
\begin{align*}
|e_{N-1}^{0,3,2}|&\le C(T)\Delta t\E[|(I-e^{-\frac{\Delta t}{\epsilon}\IL})\IL\XX_{N-1}^{\epsilon,\Delta t}|]\\
&\le C(T)\Delta t^{1-\kappa}(1+|x_0|).
\end{align*}

$\bullet$ {\it Error term $e_n^{0,3,3}$.} Recall that the process $\bigl(\tilde{\YY}^{\epsilon,\Delta t}(t)\bigr)_{t\ge 0}$ is the solution of the stochastic evolution equation~\eqref{eq:tildeY}. Applying It\^o's formula, for all $n\in\{1,\ldots,N-1\}$ and $t\in[t_n,t_{n+1}]$, one has
\begin{align*}
\E[v^{\epsilon,\Delta t}(t,\XX_n^{\epsilon,\Delta t},\tilde{\YY}^{\epsilon,\Delta t}(t))]&-\E[v^{\epsilon,\Delta t}(t_{n+1},\XX_n^{\epsilon,\Delta t},\tilde{\YY}^{\epsilon,\Delta t}(t_{n+1}))]\\
&=-\int_{t}^{t_{n+1}}\E[\partial_tv^{\epsilon,\Delta t}(s,\XX_n^{\epsilon,\Delta t},\tilde{\YY}^{\epsilon,\Delta t}(s))]ds\\
&+\frac{1}{\epsilon}\int_{t}^{t_{n+1}}\E[\langle D_yv^{\epsilon,\Delta t}(s,\XX_n^{\epsilon,\Delta t},\tilde{\YY}^{\epsilon,\Delta t}(s)),\IL_{\frac{\Delta t}{\epsilon}}\tilde{\YY}^{\epsilon,\Delta t}(s)\rangle]ds\\
&-\frac{1}{\epsilon}\int_{t}^{t_{n+1}}\sum_{j\in\N}q_{\frac{\Delta t}{\epsilon},j}\E[D_y^2v^{\epsilon,\Delta t}(s,\XX_n^{\epsilon,\Delta t},\tilde{\YY}^{\epsilon,\Delta t}(s)).(e_j,e_j)]ds.
\end{align*}
Therefore, the error term $e_n^{0,3,3}$ is decomposed as
\[
e_n^{0,3,3}=e_n^{0,3,3,1}+e_n^{0,3,3,2}+e_n^{0,3,3,3}
\]
with for all $n\in\{1,\ldots,N-1\}$ one has
\begin{align*}
e_n^{0,3,3,1}&=\int_{t_n}^{t_{n+1}}\int_{t}^{t_{n+1}}\E[\partial_tv^{\epsilon,\Delta t}(s,\XX_n^{\epsilon,\Delta t},\tilde{\YY}^{\epsilon,\Delta t}(s))]dsdt\\
e_n^{0,3,3,2}&=\int_{t_n}^{t_{n+1}}\frac{1}{\epsilon}\int_{t}^{t_{n+1}}\E[\langle D_yv^{\epsilon,\Delta t}(s,\XX_n^{\epsilon,\Delta t},\tilde{\YY}^{\epsilon,\Delta t}(s)),\IL_{\frac{\Delta t}{\epsilon}}\tilde{\YY}^{\epsilon,\Delta t}(s)\rangle]dsdt\\
e_n^{0,3,3,3}&=-\int_{t_n}^{t_{n+1}}\frac{1}{\epsilon}\int_{t}^{t_{n+1}}\sum_{j\in\N}q_{\frac{\Delta t}{\epsilon},j}D_y^2v^{\epsilon,\Delta}(s,\XX_n^{\epsilon,\Delta t},\tilde{\YY}^{\epsilon,\Delta t}(s)).(e_j,e_j)dsdt.
\end{align*}

For the error term $e_n^{0,3,3,1}$, note that, owing to the regularity estimate~\eqref{eq:lemuepsilon-3} from Lemma~\ref{lem:uepsilon}, for all $t\in[0,T)$ and $x,y\in H^{\kappa/2}$, one has
\begin{align*}
|\partial_tv^{\epsilon,\Delta t}(t,x,y)|&=|\langle \partial_tD_xu^\epsilon(T-t,x,y),\IL e^{-\frac{\Delta t}{\epsilon}\IL}x\rangle|\\
&\le \frac{C_\kappa(T)}{\epsilon (T-t)^{1-\frac{\kappa}{2}}}\bigl(1+|\IL^{\frac{\kappa}{2}} x|+|\IL^{\frac{\kappa}{2}} y|\bigr)|\IL^{1+\frac{\kappa}{2}}e^{-\frac{\Delta t}{\epsilon}\IL}x|\\
&\le \frac{C_\kappa(T)}{\Delta t^{\kappa}\epsilon^{1-\kappa} (T-t)^{1-\frac{\kappa}{2}}}\bigl(1+|\IL^{\frac{\kappa}{2}}x|+|\IL^{\frac{\kappa}{2}}y|\bigr)|\IL^{1-\frac{\kappa}{2}}x|,
\end{align*}
using the smoothing inequality~\eqref{eq:smoothing} in the last step. As a consequence, using the moment bounds~\eqref{eq:momentschemepowerbis} and~\eqref{eq:boundtildeYYkappa}, one obtains, for all $n\in\{1,\ldots,N-1\}$,
\begin{align*}
|e_n^{0,3,3,1}|&\le \int_{t_n}^{t_{n+1}}\int_{t}^{t_{n+1}}\E[|\partial_tv^{\epsilon,\Delta t}(s,\XX_n^{\epsilon,\Delta },\tilde{\YY}^{\epsilon,\Delta t}(s))|]dsdt\\
&\le \int_{t_n}^{t_{n+1}}\int_{t}^{t_{n+1}}\frac{C_\kappa(T)}{\Delta t^{\kappa}\epsilon^{1-\kappa} (T-s)^{1-\frac{\kappa}{2}}}\E[\bigl(1+|\IL^{\frac{\kappa}{2}}\XX_n^{\epsilon,\Delta t}|+|\IL^{\frac{\kappa}{2}}\tilde{\YY}^{\epsilon,\Delta t}(s)|\bigr)|\IL^{1-\frac{\kappa}{2}}\XX_n^{\Delta t,\epsilon}|]dsdt\\
&\le \int_{t_n}^{t_{n+1}}\int_{t}^{t_{n+1}}\frac{C_\kappa(T)}{\Delta t^{\kappa}\epsilon^{1-\kappa} (T-s)^{1-\frac{\kappa}{2}}}dsdt\bigl(1+|\IL^{\frac{\kappa}{2}}x_0|+|\IL^{\frac{\kappa}{2}}y_0^\epsilon|\bigr)\frac{1}{(n\Delta t)^{1-\frac{\kappa}{2}}}|x_0|\\
&\le \bigl(\frac{\Delta t}{\epsilon}\bigr)^{1-\kappa}\int_{t_n}^{t_{n+1}}\frac{C_\kappa(T)}{(T-s)^{1-\frac{\kappa}{2}}}ds\bigl(1+|\IL^{\frac{\kappa}{2}}x_0^\epsilon|\bigr)^2\frac{1}{(n\Delta t)^{1-\frac{\kappa}{2}}}.
\end{align*}

For the error term $e_n^{0,3,3,2}$, note that, owing to the regularity estimate~\eqref{eq:lemuepsilon-2} from Lemma~\ref{lem:uepsilon} (with $\alpha_1=0$ and $\alpha_2=1-\kappa/2$) and to the inequality~\eqref{eq:smoothing}, for all $t\in[0,T)$ and $x,y\in H$, one has
\begin{align*}
|\langle D_yv^{\epsilon,\Delta t}(t,x,y),h\rangle|=|D_xD_yu^\epsilon(t,x,y).(e^{-\frac{\Delta t}{\epsilon}\IL}\IL x,h)|\le \frac{C_\kappa(T)\epsilon}{(T-t)^{1-\frac{\kappa}{2}}}\big(\frac{\epsilon}{\Delta t}\bigr)^{\frac{\kappa}{2}}|\IL^{1-\frac{\kappa}{2}}x||\IL^{-1+\frac{\kappa}{2}}h|.
\end{align*}
As a consequence, using the moment bounds~\eqref{eq:momentschemepower} and~\eqref{eq:boundtildeYYkappa}, and the inequality $\lambda_{\tau,j}\le \lambda_j$ for all $j\in\N$ and $\tau\in(0,\infty)$ one obtains, for all $n\in\{1,\ldots,N-1\}$,
\begin{align*}
|e_n^{0,3,3,2}|&\le \int_{t_n}^{t_{n+1}}\int_{t}^{t_{n+1}} \frac{1}{\epsilon}\E[|\langle D_yv^{\epsilon,\Delta t}(s,\XX_n^{\epsilon,\Delta t},\tilde{\YY}^{\epsilon,\Delta t}(s)),\IL_{\frac{\Delta t}{\epsilon}}\tilde{\YY}^{\epsilon,\Delta t}(s)\rangle|]dsdt\\
&\le \int_{t_n}^{t_{n+1}}\int_{t}^{t_{n+1}}\frac{C_\kappa(T)\epsilon^{\frac{\kappa}{2}}}{(T-s)^{1-\frac{\kappa}{2}}\Delta t^\frac{\kappa}{2}}\E[|\IL^{1-\frac{\kappa}{2}}\XX_n^{\epsilon,\Delta t}||\IL^{-1+\frac{\kappa}{2}}\IL_{\frac{\Delta t}{\epsilon}}\tilde{\YY}^{\epsilon,\Delta t}(s)|]dsdt\\
&\le \int_{t_n}^{t_{n+1}}\int_{t}^{t_{n+1}}\frac{C_\kappa(T)\epsilon^{\frac{\kappa}{2}}}{(T-s)^{1-\frac{\kappa}{2}}}\Delta t^{\frac{\kappa}{2}}\E[|\IL^{1-\frac{\kappa}{2}}\XX_n^{\epsilon,\Delta t}||\IL^{\frac{\kappa}{2}}\tilde{\YY}^{\epsilon,\Delta t}(s)|]dsdt\\
&\le \Delta t^{1-\frac{\kappa}{2}}\int_{t_n}^{t_{n+1}}\frac{C_\kappa(T)}{(T-s)^{1-{\frac{\kappa}{2}}}}ds\frac{1}{(n\Delta t)^{1-\frac{\kappa}{2}}}(1+|\IL^{\frac{\kappa}{2}}x_0|)^2.
\end{align*}

For the error term $e_n^{0,3,3,3}$, note that, owing to the regularity estimate~\eqref{eq:lemuepsilon-aux} from the proof of Lemma~\ref{lem:uepsilon} and to the inequality~\eqref{eq:smoothing}, for all $j\in\N$, $t\in[0,T)$ and $x,y\in H$, one has
\begin{align*}
|D^yv^{\epsilon,\Delta t}(t,x,y).(e_j,e_j)|&=|D_xD_y^2v^{\epsilon,\Delta t}(T-t,x,y).(e^{-\frac{\Delta t}{\epsilon}\IL}\IL x,e_j,e_j)|\\
&\le \frac{C_\kappa(T)}{(T-t)^{1-\kappa}}\bigl(\frac{\epsilon}{\Delta t}\bigr)^{\kappa}|\IL^{1-\kappa}x|\lambda_j^{-1+\kappa}.
\end{align*}
As a consequence, using the moment bound~\eqref{eq:momentschemepower}, one obtains, for all $n\in\{1,\ldots,N-1\}$,
\begin{align*}
|e_n^{0,3,3,3}|&\le \frac{1}{\epsilon}\int_{t_n}^{t_{n+1}}\int_{t}^{t_{n+1}}\sum_{j\in\N}q_{\frac{\Delta t}{\epsilon},j}\E[D_y^2v^{\epsilon,\Delta t}(s,\XX_n^{\epsilon,\Delta t},\tilde{\YY}^{\epsilon,\Delta t}(s)).(e_j,e_j)|]dsdt\\
&\le \frac{1}{\epsilon}\int_{t_n}^{t_{n+1}}\int_{t}^{t_{n+1}}\frac{C_\kappa(T)}{(T-s)^{1-\kappa}}dsdt\bigl(\frac{\epsilon}{\Delta t}\bigr)^{\kappa}\E[|\IL^{1-\kappa}\XX_n^{\epsilon,\Delta t}|]\sum_{j\in\N}q_{\frac{\Delta t}{\epsilon},j}\lambda_j^{-1+\kappa}\\
&\le \bigl(\frac{\Delta t}{\epsilon}\bigr)^{1-\kappa}\int_{t_n}^{t_{n+1}}\frac{C_\kappa(T)}{(T-s)^{1-\kappa}}ds\frac{1}{(n\Delta t)^{1-\kappa}}(1+|x_0|)
\end{align*}

Gathering the estimates, for all $n\in\{1,\ldots,N-1\}$, one then obtains the inequality
\begin{align*}
|e_n^{0,3,3}|&\le |e_n^{0,3,3,1}|+|e_n^{0,3,3,2}|+|e_n^{0,3,3,3}|\\
&\le \Bigl(\frac{\Delta t}{\epsilon}\Bigr)^{1-\kappa}\int_{t_n}^{t_{n+1}}\frac{C_\kappa(T)}{(T-s)^{1-\frac{\kappa}{2}}}ds\bigl(1+|\IL^{\frac{\kappa}{2}}x_0|\bigr)^2\frac{1}{(n\Delta t)^{1-\frac{\kappa}{2}}}.
\end{align*}

$\bullet$ {\it Error term $e_0^{0,3}$.} Note that, owing to the regularity estimate~\eqref{eq:lemuepsilon-1} from Lemma~\ref{lem:uepsilon}, one has
\begin{align*}
|e_0^{0,3}|&\le \int_{t_0}^{t_{1}}\E[|\langle D_xu^\epsilon(T-t,x_0^{\epsilon,\Delta t},\tilde{\YY}^{\epsilon,\Delta t}(t)),\IL x_0^\epsilon\rangle|]dt\\
&+\Delta t\E[|\langle D_xu^\epsilon(T-t_{1},x_0^\epsilon,\YY^{\epsilon,\Delta t}(t_{1})),\IL x_0^{\epsilon,\Delta t}\rangle|]\\
&\le C(T)\Delta t(1+|\IL^{\frac{\kappa}{2}}x_0|).
\end{align*}
Gathering the estimates for the error terms $e_n^{0,3,1}$, $e_n^{0,3,2}$ and $e_n^{0,3,3}$ and summing for $n\in\{0,\ldots,N-1\}$, the proof of the inequality~\eqref{eq:ineqen03} is thus completed.
\end{proof}

\begin{proof}[Proof of the inequality~\eqref{eq:ineqen04}]
Recall that the error term $e_n^{0,4}=e_n^{0,4,\epsilon,\Delta t}$ is defined by~\eqref{eq:en04} for all $n\in\{0,\ldots,N-1\}$. The cases $n\in\{1,\ldots,N-1\}$ and $n=0$ are treated differently.

Like in the proof of the inequality~\eqref{eq:ineqen03}, it is necessary to introduce an auxiliary mapping $w^{\epsilon,\Delta t}:[0,T]\times H^2\to\R$ defined as follows: for all $t\in[0,T]$, $x,y\in H$, set
\begin{equation}
w^{\epsilon,\Delta t}(t,x,y)=\langle D_xu^\epsilon(T-t,x,y),e^{-\frac{\Delta t}{\epsilon}\IL} F(x,y)\rangle.
\end{equation}
For all $n\in\{1,\ldots,N-1\}$, the error term $e_n^{0,3}$ is then decomposed as
\[
e_n^{0,4}=e_n^{0,4,1}+e_n^{0,4,2}+e_n^{0,4,3},
\]
where, for all $n\in\{1,\ldots,N-1\}$, one has
\begin{align*}
e_n^{0,4,1}&=\int_{t_n}^{t_{n+1}}\E[\langle D_xu^\epsilon(T-t,\XX_n^{\epsilon,\Delta t},\tilde{\YY}^{\epsilon,\Delta t}(t)),(I-e^{-\frac{\Delta t}{\epsilon}\IL})F(\XX_n^{\epsilon,\Delta t},\tilde{\YY}^{\epsilon,\Delta t}(t))\rangle] dt\\
e_n^{0,4,2}&=\Delta t\E[\langle D_xu^\epsilon(T-t_{n+1},\XX_n^{\epsilon,\Delta t},\tilde{\YY}^{\epsilon,\Delta t}(t_{n+1})),(I-e^{-\frac{\Delta t}{\epsilon}\IL})F(\XX_n^{\epsilon,\Delta t},\tilde{\YY}^{\epsilon,\Delta t}(t_{n+1}))\rangle]\\
e_n^{0,4,3}&=\int_{t_n}^{t_{n+1}}\bigl(\E[w^{\epsilon,\Delta t}(t,\XX_n^{\epsilon,\Delta t},\tilde{\YY}^{\epsilon,\Delta t}(t))]-\E[w^{\epsilon,\Delta t}(t_{n+1},\XX_n^{\epsilon,\Delta t},\tilde{\YY}^{\epsilon,\Delta t}(t_{n+1}))]\bigr)dt
\end{align*}

$\bullet$ {\it Error term $e_n^{0,4,1}$.} Owing to the regularity estimate~\eqref{eq:lemuepsilon-1} from Lemma~\ref{lem:uepsilon} (with $\alpha=1-\kappa$) and to the inequality~\eqref{eq:regularity}, for all $n\in\{1,\ldots,N-1\}$, one has
\begin{align*}
|e_n^{0,4,1}|&\le \int_{t_n}^{t_{n+1}}\frac{C_\kappa(T)}{(T-t)^{1-\kappa}}\E[|\IL^{-1+\kappa}(I-e^{-\frac{\Delta t}{\epsilon}\IL})F(\XX_n^{\epsilon,\Delta t},\tilde{\YY}^{\epsilon,\Delta t}(t))|]dt\\
&\le \bigl(\frac{\Delta t}{\epsilon}\bigr)^{1-\kappa}\int_{t_n}^{t_{n+1}}\frac{C_\kappa(T)}{(T-t)^{1-\kappa}}\E[|F(\XX_n^{\epsilon,\Delta t},\tilde{\YY}^{\epsilon,\Delta t}(t))|]dt\\
&\le \bigl(\frac{\Delta t}{\epsilon}\bigr)^{1-\kappa}\int_{t_n}^{t_{n+1}}\frac{C_\kappa(T)}{(T-t)^{1-\kappa}}dt(1+|x_0|),
\end{align*}
using the Lipschitz continuity of $F$ and the moment bounds~\eqref{eq:boundYnepsilonDeltat} and~\eqref{eq:boundXnepsilonDeltat} from Lemma~\ref{lem:momentboundsscheme-epsilonDeltat}.

$\bullet$ {\it Error term $e_n^{0,4,2}$.} The cases $n\in\{1,\ldots,N-2\}$ and $n=N-1$ are treated differently.

On the one hand, owing to the regularity estimate~\eqref{eq:lemuepsilon-1} from Lemma~\ref{lem:uepsilon} (with $\alpha=1-\kappa$) and to the inequality~\eqref{eq:regularity}, for all $n\in\{0,\ldots,N-2\}$, one has
\begin{align*}
|e_n^{0,4,2}|&\le \frac{C_\kappa(T)\Delta t}{(T-t_{n+1})^{1-\kappa}}\E[|\IL^{-1+\kappa}(I-e^{-\frac{\Delta t}{\epsilon}\IL})F(\XX_n^{\epsilon,\Delta t},\tilde{\YY}^{\epsilon,\Delta t}(t_{n+1}))|]\\
&\le \bigl(\frac{\Delta t}{\epsilon}\bigr)^{1-\kappa}\frac{C_\kappa(T)\Delta t}{(T-t_{n+1})^{1-\kappa}}\E[|F(\XX_n^{\epsilon,\Delta t},\tilde{\YY}^{\epsilon,\Delta t}(t_{n+1}))|]dt\\
&\le \bigl(\frac{\Delta t}{\epsilon}\bigr)^{1-\kappa}\frac{C_\kappa(T)\Delta t}{(T-t_{n+1})^{1-\kappa}}dt(1+|x_0|),
\end{align*}
using the Lipschitz continuity of $F$ and the moment bounds~\eqref{eq:boundYnepsilonDeltat} and~\eqref{eq:boundXnepsilonDeltat} from Lemma~\ref{lem:momentboundsscheme-epsilonDeltat}.

On the other hand, owing to the regularity estimate~\eqref{eq:lemuepsilon-1} from Lemma~\ref{lem:uepsilon} (with $\alpha=0$), one has
\begin{align*}
|e_{N-1}^{0,4,2}|&\le C(T)\Delta t\E[|(I-e^{-\frac{\Delta t}{\epsilon}\IL})F(\XX_{N-1}^{\epsilon,\Delta t},\tilde{\YY}^{\epsilon,\Delta t}(t_{N}))|]\\
&\le C(T)\Delta t(1+|x_0|)
\end{align*}
using the Lipschitz continuity of $F$ and the moment bounds~\eqref{eq:boundYnepsilonDeltat} and~\eqref{eq:boundXnepsilonDeltat} from Lemma~\ref{lem:momentboundsscheme-epsilonDeltat}.

$\bullet$ {\it Error term $e_n^{0,4,3}$.} Applying It\^o's formula, for all $n\in\{1,\ldots,N-1\}$ and $t\in[t_n,t_{n+1}]$, one has
\begin{align*}
\E[w^{\epsilon,\Delta t}(t,\XX_n^{\epsilon,\Delta t},\tilde{\YY}^{\epsilon,\Delta t}(t))]&-\E[w^{\epsilon,\Delta t}(t_{n+1},\XX_n^{\epsilon,\Delta t},\tilde{\YY}^{\epsilon,\Delta t}(t_{n+1}))]\\
&=-\int_{t}^{t_{n+1}}\E[\partial_tw^{\epsilon,\Delta t}(s,\XX_n^{\epsilon,\Delta },\tilde{\YY}^{\epsilon,\Delta t}(s))]ds\\
&+\frac{1}{\epsilon}\int_{t}^{t_{n+1}}\E[\langle D_yw^{\epsilon,\Delta t}(s,\XX_n^{\epsilon,\Delta t},\tilde{\YY}^{\epsilon,\Delta t}(s)),\IL_{\frac{\Delta t}{\epsilon}}\tilde{\YY}^{\epsilon,\Delta t}(s)\rangle]ds\\
&-\frac{1}{\epsilon}\int_{t}^{t_{n+1}}\sum_{j\in\N}q_{\frac{\Delta t}{\epsilon},j}\E[D_y^2w^{\epsilon,\Delta t}(s,\XX_n^{\epsilon,\Delta t},\tilde{\YY}^{\epsilon,\Delta t}(s)).(e_j,e_j)]ds.
\end{align*}
Therefore, one has the decomposition $e_n^{0,4,3}=e_n^{0,4,3,1}+e_n^{0,4,3,2}+e_n^{0,4,3,3}$, with
\begin{align*}
e_n^{0,4,3,1}&=-\int_{t_n}^{t_{n+1}}\int_{t}^{t_{n+1}}\E[\partial_tw^{\epsilon,\Delta t}(s,\XX_n^{\epsilon,\Delta t},\tilde{\YY}^{\epsilon,\Delta t}(s))]dsdt\\
e_n^{0,4,3,2}&=\int_{t_n}^{t_{n+1}}\frac{1}{\epsilon}\int_{t}^{t_{n+1}}\E[\langle D_yw^{\epsilon,\Delta t}(s,\XX_n^{\epsilon,\Delta t},\tilde{\YY}^{\epsilon,\Delta t}(s)),\IL_{\frac{\Delta t}{\epsilon}}\tilde{\YY}^{\epsilon,\Delta t}(s)\rangle]dsdt\\
e_n^{0,4,3,3}&=-\int_{t_n}^{t_{n+1}}\frac{1}{\epsilon}\int_{t}^{t_{n+1}}\sum_{j\in\N}q_{\frac{\Delta t}{\epsilon},j}D_y^2w^{\epsilon,\Delta t}(s,\XX_n^{\epsilon,\Delta t},\tilde{\YY}^{\epsilon,\Delta t}(s)).(e_j,e_j)dsdt.
\end{align*}

For the error term $e_n^{0,4,3,1}$, note that, owing to the regularity estimate~\eqref{eq:lemuepsilon-3} from Lemma~\ref{lem:uepsilon}, for all $t\in[0,T]$ and $x,y\in H^{\frac{\kappa}{2}}$, one has
\begin{align*}
|\partial_tw^{\epsilon,\Delta t}(t,x,y)|&=|\langle \partial_tD_xu^\epsilon(T-t,x,y),e^{-\frac{\Delta t}{\epsilon}\IL}F(x,y)\rangle|\\
&\le \frac{C_\kappa(T)}{\epsilon(T-t)^{1-\frac{\kappa}{2}}}\bigl(1+|\IL^{\frac{\kappa}{2}} x|+|\IL^{\frac{\kappa}{2}} y|)|\IL^{\frac{\kappa}{2}}e^{-\frac{\Delta t}{\epsilon}\IL}F(x,y)|\\
&\le \frac{C_\kappa(T)}{\Delta t^{\frac{\kappa}{2}}\epsilon^{1-\frac{\kappa}{2}}(T-t)^{1-\frac{\kappa}{2}}}\bigl(1+|\IL^{\frac{\kappa}{2}}x|+|\IL^{\frac{\kappa}{2}}y|)^2,
\end{align*}
using the smoothing inequality~\eqref{eq:smoothing} and the linear growth of $F$ in the last step. As a consequence, using the moment bounds~\eqref{eq:momentschemepower} and~\eqref{eq:boundtildeYYkappa}, for all $n\in\{1,\ldots,N-1\}$, one obtains
\begin{align*}
|e_n^{0,4,3,1}|&\le \int_{t_n}^{t_{n+1}}\int_{t}^{t_{n+1}}\E[|\partial_tw^{\epsilon,\Delta t}(s,\XX_n^{\epsilon,\Delta t},\tilde{\YY}^{\epsilon,\Delta t}(s))|]dsdt\\
&\le \int_{t_n}^{t_{n+1}}\int_{t}^{t_{n+1}} \frac{C_\kappa(T)}{\Delta^{\frac{\kappa}{2}}\epsilon^{1-\frac{\kappa}{2}}(T-s)^{1-\frac{\kappa}{2}}}\bigl(1+\E[|\IL^{\frac{\kappa}{2}}\XX_n^{\epsilon,\Delta t}|^2]+\E[|\IL^{\frac{\kappa}{2}}\tilde{\YY}^{\epsilon,\Delta t}(s)|^2]\bigr) dsdt\\
&\le \bigl(\frac{\Delta t}{\epsilon}\bigr)^{1-\kappa}\int_{t_n}^{t_{n+1}}\frac{C_\kappa(T)}{(T-s)^{1-\frac{\kappa}{2}}}ds\frac{1}{(n\Delta t)^{1-\kappa}}\bigl(1+|\IL^{\frac{\kappa}{2}}x_0|\bigr)^2.
\end{align*}

For the error term $e_n^{0,4,3,2}$, like for the treatment of the error term $e_n^{1,2}$ above (proof of the inequality~\eqref{eq:ineqen1}), the Malliavin integration by parts formula~\eqref{eq:MalliavinIBP} is employed. Recall that the mild solution $\bigl(\tilde{\YY}^{\epsilon,\Delta t}(t)\bigr)_{t\ge 0}$ of the stochastic evolution equation~\eqref{eq:tildeY} is given by~\eqref{eq:mildtildeYY}. The error term $e_n^{0,4,3,2}$ is then decomposed as
\[
e_n^{0,4,3,2}=e_n^{0,4,3,2,1}+e_n^{0,4,3,2,2},
\]
where for all $n\in\{1,\ldots,N-1\}$ one has
\begin{align*}
e_n^{0,4,3,2,1}&=\frac{1}{\epsilon}\int_{t_n}^{t_{n+1}}\int_{t}^{t_{n+1}}\E[\langle D_yw^{\epsilon,\Delta t}(s,\XX_n^{\epsilon,\Delta t},\tilde{\YY}^{\epsilon,\Delta t}(s)),\IL_{\frac{\Delta t}{\epsilon}}e^{-\frac{s}{\epsilon}\IL_{\frac{\Delta t}{\epsilon}}}y_0^\epsilon\rangle]dsdt\\
e_n^{0,4,3,2,2}&=\frac{1}{\epsilon}\int_{t_n}^{t_{n+1}}\int_{t}^{t_{n+1}}\E[\langle D_yw^{\epsilon,\Delta t}(s,\XX_n^{\epsilon,\Delta t},\tilde{\YY}^{\epsilon,\Delta t}(s)),\IL_{\frac{\Delta t}{\epsilon}}\sqrt{\frac{{2}}{{\epsilon}}}\int_{0}^{s}e^{-\frac{s-r}{\epsilon}\IL_{\frac{\Delta t}{\epsilon}}}Q_{\frac{\Delta t}{\epsilon}}^{\frac12}dW(r)\rangle]dsdt.
\end{align*}

To deal with the error term $e_n^{0,4,3,2,1}$, note that for all $t\in[0,T]$, $x,y\in H$ and $h\in H$, one has
\begin{align*}
|\langle D_yw^{\epsilon,\Delta t}(t,x,y),h\rangle|&\le\big|\langle D_xu^\epsilon(T-t,x,y),e^{-\frac{\Delta t}{\epsilon}\IL} D_yF(x,y).h\rangle|\\
&+|D_xD_yu^\epsilon(T-t,x,y).(e^{-\frac{\Delta t}{\epsilon}\IL}F(x,y),h)|\\
&\le C(T)\vvvert\varphi\vvvert_2(1+|x|+|y|)|h|.
\end{align*}
Owing to the regularity estimates~\eqref{eq:lemuepsilon-1} and~\ref{eq:lemuepsilon-2} from Lemma~\ref{lem:uepsilon}. Therefore, using the moment bounds~\eqref{eq:boundXnepsilonDeltat} and~\eqref{eq:boundtildeYYkappa}, one obtains
\begin{align*}
|e_n^{0,4,3,2,1}|&\le\frac{1}{\epsilon}\int_{t_n}^{t_{n+1}}\int_{t}^{t_{n+1}}\E[|\langle D_yw^{\epsilon,\Delta t}(s,\XX_n^{\epsilon,\Delta t},\tilde{\YY}^{\epsilon,\Delta t}(s)),\IL_{\frac{\Delta t}{\epsilon}}e^{-\frac{s}{\epsilon}\IL_{\frac{\Delta t}{\epsilon}}}y_0^\epsilon\rangle|]dsdt\\
&\le \frac{C(T)}{\epsilon}\int_{t_n}^{t_{n+1}}\int_{t}^{t_{n+1}}\E[\bigl(1+|\XX_n^{\epsilon,\Delta t}|+|\tilde{\YY}^{\epsilon,\Delta t}(t)|\bigr)|\IL_{\frac{\Delta t}{\epsilon}}e^{-\frac{s}{\epsilon}\IL_{\frac{\Delta t}{\epsilon}}}y_0^\epsilon|]dsdt\\
&\le \frac{C_\kappa(T)\Delta t}{\epsilon}\int_{t_n}^{t_{n+1}}\frac{\epsilon^{1-\frac{\kappa}{2}}}{s^{1-\frac{\kappa}{2}}}ds(1+|\IL^{\frac{\kappa}{2}}x_0|)^2,
\end{align*}
using the inequality
\[
|\IL_{\frac{\Delta t}{\epsilon}}e^{-\frac{s}{\epsilon}\IL_{\frac{\Delta t}{\epsilon}}}y_0^\epsilon|\le C_\kappa(T)\frac{\epsilon^{1-\kappa}}{s^{1-\kappa}}|\IL_{\frac{\Delta t}{\epsilon}}^{\kappa}y_0^\epsilon|\le C_\kappa(T)\frac{\epsilon^{1-\frac{\kappa}{2}}}{s^{1-\frac{\kappa}{2}}}|\IL^{\frac{\kappa}{2}}y_0^\epsilon|
\]
which follows from a version of the smoothing inequality~\eqref{eq:smoothing} applied to the linear operator $\IL_{\frac{\Delta t}{\epsilon}}$ instead of $\IL$ and its associated semi-group, and from the inequality $\lambda_{\tau,j}\le \lambda_j$ for all $j\in\N$ and $\tau\in(0,\infty)$.

To deal with the error term $e_n^{0,4,3,2,2}$, applying the Malliavin integration by parts formula~\eqref{eq:MalliavinIBP}, one obtains
\begin{align*}
&e_n^{0,4,3,2,2}\\
&=\frac{\sqrt{2}}{\epsilon^{\frac{3}{2}}}\int_{t_n}^{t_{n+1}}\int_{t}^{t_{n+1}}\sum_{j\in\N}\E[\langle D_yw^{\epsilon,\Delta t}(s,\XX_n^{\epsilon,\Delta t},\tilde{\YY}^{\epsilon,\Delta t}(s)),e_j\rangle \lambda_{\frac{\Delta t}{\epsilon},j}\int_{0}^{s}e^{-\frac{s-r}{\epsilon}\lambda_{\frac{\Delta t}{\epsilon},j}}\sqrt{q_{\frac{\Delta t}{\epsilon},j}}d\beta_j(r)]dsdt\\
&=\frac{\sqrt{2}}{\epsilon^{\frac{3}{2}}}\int_{t_n}^{t_{n+1}}\int_{t}^{t_{n+1}}\int_0^s\sum_{j\in\N}\E[\mathcal{D}_r^{e_j}\bigl(\langle D_yw^{\epsilon,\Delta t}(s,\XX_n^{\epsilon,\Delta t},\tilde{\YY}^{\epsilon,\Delta t}(s)),e_j\rangle\bigr)] \lambda_{\frac{\Delta t}{\epsilon},j}e^{-\frac{s-r}{\epsilon}\lambda_{\frac{\Delta t}{\epsilon},j}}\sqrt{q_{\frac{\Delta t}{\epsilon},j}}drdsdt\\
&=e_n^{0,4,3,2,2,1}+e_n^{0,4,3,2,2,2}
\end{align*}
where, using the chain rule, one has
\begin{align*}
e_n^{0,4,3,2,2,1}&=\frac{\sqrt{2}}{\epsilon\sqrt{\epsilon}}\int_{t_n}^{t_{n+1}}\int_{t}^{t_{n+1}}\int_0^s\sum_{j\in\N} d_{n,j}^{x,\epsilon,\Delta t}(r,s)\lambda_{\frac{\Delta t}{\epsilon},j}e^{-\frac{s-r}{\epsilon}\lambda_{\frac{\Delta t}{\epsilon},j}}\sqrt{q_{\frac{\Delta t}{\epsilon},j}}drdsdt\\
e_n^{0,4,3,2,2,2}&=\frac{\sqrt{2}}{\epsilon\sqrt{\epsilon}}\int_{t_n}^{t_{n+1}}\int_{t}^{t_{n+1}}\int_0^s\sum_{j\in\N}d_{n,j}^{y,\epsilon,\Delta t}(r,s)\lambda_{\frac{\Delta t}{\epsilon},j}e^{-\frac{s-r}{\epsilon}\lambda_{\frac{\Delta t}{\epsilon},j}}\sqrt{q_{\frac{\Delta t}{\epsilon},j}}drdsdt
\end{align*}
with
\begin{align*}
d_{n,j}^{x,\epsilon,\Delta t}(r,s)&=\E[D_xD_yw^{\epsilon,\Delta t}(s,\XX_n^{\epsilon,\Delta t},\tilde{\YY}^{\epsilon,\Delta t}(s)).\bigl(\mathcal{D}_r^{e_j}\XX_n^{\epsilon,\Delta t},e_j\bigr)]\\
d_{n,j}^{y,\epsilon,\Delta t}(r,s)&=\E[D_y^2w^{\epsilon,\Delta t}(s,\XX_n^{\epsilon,\Delta t},\tilde{\YY}^{\epsilon,\Delta t}(s)).\bigl(\mathcal{D}_r^{e_j}\tilde{\YY}^{\epsilon,\Delta t}(s),e_j\bigr)].
\end{align*}

On the one hand, note that for all $t\in[0,T]$, $x,y\in H$ and $h^1,h^2\in H$, one has
\begin{align*}
D_xD_yw^{\epsilon,\Delta}(t,x,y).(h^1,h^2)&=\langle D_xu^\epsilon(T-t,x,y),e^{-\frac{\Delta t}{\epsilon}\IL}D_xD_yF(x,y).(h^1,h^2)\rangle\\
&+D_x^2u^\epsilon(T-t,x,y).(h^1,e^{-\frac{\Delta t}{\epsilon}\IL}D_yF(x,y).h^2)\\
&+D_xD_yu^\epsilon(T-t,x,y).(e^{-\frac{\Delta t}{\epsilon}\IL}D_xF(x,y).h^1,h^2)\\
&+D_x^2D_yu^\epsilon(T-t,x,y).(h^1,e^{-\frac{\Delta t}{\epsilon}\IL}F(x,y),h^2).
\end{align*}
Owing to the regularity estimates~\eqref{eq:lemuepsilon-1},~\eqref{eq:lemuepsilon-2} and~\eqref{eq:lemuepsilon-aux} from Lemma~\ref{lem:uepsilon} and its proof, and to the properties of $F$ stated in Assumption~\ref{ass:F}, one has the upper bound
\[
|D_xD_y^2w^{\epsilon,\Delta t}(t,x,y).(h^1,h^2)|\le C(T)(1+|x|+|y|)|h^1||h^2|.
\]
Using the bound~\eqref{eq:DsX} for the Malliavin derivative $\mathcal{D}_r^{e_j}\XX_n^{\epsilon,\Delta t}$ (see the proof of the inequality~\eqref{eq:ineqen2}) and the moment bounds~\eqref{eq:boundXnepsilonDeltat} from Lemma~\ref{lem:momentboundsscheme-epsilonDeltat} and~\eqref{eq:boundtildeYYkappa}, one has
\begin{align*}
|e_n^{0,4,3,2,2,1}|&\le \frac{C(T)}{\epsilon^2}(1+|x_0|)\int_{t_n}^{t_{n+1}}\int_{t}^{t_n}\int_0^s\sum_{j\in\N}\lambda_{\frac{\Delta t}{\epsilon},j}e^{-\frac{s-r}{\epsilon}\lambda_{\frac{\Delta t}{\epsilon},j}}q_{\frac{\Delta t}{\epsilon},j}drdsdt\\
&\le \frac{C(T)}{\epsilon}(1+|x_0|)\int_{t_n}^{t_{n+1}}\int_{t}^{t_n}\sum_{j\in\N}q_{\frac{\Delta t}{\epsilon},j}dsdt\\
&\le \frac{C(T)\Delta t^2}{\epsilon}(1+|x_0|)\vvvert\varphi\vvvert_3\sum_{j\in\N}q_{\frac{\Delta t}{\epsilon},j}.
\end{align*}

On the other hand, note that for all $t\in[0,T]$, $x,y\in H$ and $h^1,h^2\in H$, one has
\begin{align*}
D_y^2w^{\epsilon,\Delta t}(t,x,y).(h^1,h^2)&=D_xD_y^2 u^\epsilon(T-t,x,y).\bigl(e^{-\frac{\Delta t}{\epsilon}\IL}F(x,y),h^1,h^2\bigr)\\
&+D_xD_yu^\epsilon(T-t,x,y).\bigl(e^{-\frac{\Delta t}{\epsilon}\IL}D_yF(x,y).h^1,h^2\bigr)\\
&+D_xD_yu^\epsilon(T-t,x,y).\bigl(e^{-\frac{\Delta t}{\epsilon}\IL}D_yF(x,y).h^2,h^1\bigr)\\
&+\langle D_xu^\epsilon(T-t,x,y),e^{-\frac{\Delta t}{\epsilon}\IL}D_y^2F(x,y).(h^1,h^2)\rangle.
\end{align*}
Owing to the regularity estimates~\eqref{eq:lemuepsilon-1},~\eqref{eq:lemuepsilon-2} and~\eqref{eq:lemuepsilon-aux} from Lemma~\ref{lem:uepsilon} and its proof, and to the properties of $F$ stated in Assumption~\ref{ass:F}, one has the upper bound
\[
|D_y^2w^{\epsilon,\Delta t}(t,x,y).(h_1,h_2)|\le C(T)(1+|x|+|y|)|h^1||h^2|.
\]
Using the bound~\eqref{eq:DsY} for the Malliavin derivative $\mathcal{D}_r^{e_j}\tilde{\YY}^{\epsilon,\Delta t}(s)$ (see the proof of the inequality~\eqref{eq:ineqen2}) and the moment bounds~\eqref{eq:boundXnepsilonDeltat} from Lemma~\ref{lem:momentboundsscheme-epsilonDeltat} and~\eqref{eq:boundtildeYYkappa}, one obtains
\begin{align*}
|e_n^{0,4,3,2,2,2}|&\le \frac{C(T)}{\epsilon^2}(1+|x_0|)\int_{t_n}^{t_{n+1}}\int_{t}^{t_n}\int_0^s\sum_{j\in\N}\lambda_{\frac{\Delta t}{\epsilon},j}e^{-\frac{s-r}{\epsilon}\lambda_{\frac{\Delta t}{\epsilon},j}}q_{\frac{\Delta t}{\epsilon},j}drdsdt\\
&\le \frac{C(T)}{\epsilon}(1+|x_0|)\int_{t_n}^{t_{n+1}}\int_{t}^{t_n}\sum_{j\in\N}q_{\frac{\Delta t}{\epsilon},j}dsdt\\
&\le \frac{C(T)\Delta t^2}{\epsilon}(1+|x_0|)\sum_{j\in\N}q_{\frac{\Delta t}{\epsilon},j}.
\end{align*}
Note that, for all $\kappa\in(0,\frac12)$ and all $\tau\in(0,\infty)$, one has
\[
\sum_{j\in\N}q_{\tau,j}=\sum_{j\in\N}\frac{\log(1+\tau\lambda_j)}{\tau\lambda_j}\le C_\kappa\sum_{j\in\N}\frac{(\tau\lambda_j)^{\frac12-\kappa}}{\tau\lambda_j}\le C_\kappa \tau^{-\frac12-\kappa},
\]
using the auxiliary inequality
\[
\underset{z\in(0,\infty)}\sup~\frac{\log(1+z)}{z^{\frac12-\kappa}}<\infty.
\]

Gathering the estimates for the error terms $e_n^{0,4,3,2,1}$, $e_n^{0,4,3,2,2,1}$ and $e_n^{0,4,3,2,2,2}$, one obtains
\begin{align*}
|e_n^{0,4,3,2}|&\le |e_n^{0,4,3,2,1}|+|e_n^{0,4,3,2,2}|\\
&\le |e_n^{0,4,3,2,1}|+|e_n^{0,4,3,2,2,1}|+|e_n^{0,4,3,2,2,2`}|\\
&\le \frac{C_\kappa(T)\Delta t}{\epsilon}(1+|x_0|)\int_{t_n}^{t_{n+1}}\frac{1}{s^{1-\frac{\kappa}{2}}}ds|\IL^{\frac{\kappa}{2}}y_0^\epsilon|+C_\kappa(T)\bigl(\frac{\Delta t}{\epsilon}\bigr)^{\frac12-\kappa}(1+|x_0|).
\end{align*}

For the error term $e_n^{0,4,3,3}$, note that for all $t\in[0,T]$, $x,y\in H$ and $h^1,h^2\in H$, one has
\begin{align*}
D_y^2w^{\epsilon,\Delta t}(t,x,y).(h^1,h^2)&=D_xD_y^2 u^\epsilon(T-t,x,y).\bigl(e^{-\frac{\Delta t}{\epsilon}\IL}F(x,y),h^1,h^2\bigr)\\
&+D_xD_yu^\epsilon(T-t,x,y).\bigl(e^{-\frac{\Delta t}{\epsilon}\IL}D_yF(x,y).h^1,h^2\bigr)\\
&+D_xD_yu^\epsilon(T-t,x,y).\bigl(e^{-\frac{\Delta t}{\epsilon}\IL}D_yF(x,y).h^2,h^1\bigr)\\
&+\langle D_xu^\epsilon(T-t,x,y),e^{-\frac{\Delta t}{\epsilon}\IL}D_y^2F(x,y).(h^1,h^2)\rangle.
\end{align*}
Owing to the regularity estimates~\eqref{eq:lemuepsilon-1},~\eqref{eq:lemuepsilon-2} and~\eqref{eq:lemuepsilon-aux} from Lemma~\ref{lem:uepsilon} and its proof, and to the properties of $F$ stated in Assumption~\ref{ass:F}, one has the upper bound
\[
|D_y^2w^{\epsilon,\Delta t}(t,x,y).(h^1,h^2)|\le C(T)(1+|x|+|y|)|h^1||h^2|.
\]
As a consequence, using the moment bounds~\eqref{eq:boundXnepsilonDeltat} and~\eqref{eq:boundtildeYYkappa}, one obtains
\[
|e_n^{0,4,3,3}|\le \frac{C(T)\Delta t^2}{\epsilon}(1+|x_0|)\sum_{j\in\N}q_{\frac{\Delta t}{\epsilon},j}\le C_\kappa(T)\bigl(\frac{\Delta t}{\epsilon}\bigr)^{\frac12-\kappa}(1+|x_0|).
\]

Finally, gathering the estimates for the error terms $e_n^{0,4,3,1}$, $e_n^{0,4,3,2}$ and $e_n^{0,4,3,3}$, one obtains, for all $n\in\{1,\ldots,N-1\}$
\begin{align*}
|e_n^{0,4,3}|&\le |e_n^{0,4,3,1}|+|e_n^{0,4,3,2}|+|e_n^{0,4,3,3}|\\
&\le \bigl(\frac{\Delta t}{\epsilon}\bigr)^{1-\kappa}\int_{t_n}^{t_{n+1}}\frac{C_\kappa(T)}{(T-s)^{1-\frac{\kappa}{2}}}ds\frac{1}{(n\Delta t)^{1-\kappa}}\bigl(1+|\IL^{\frac{\kappa}{2}}x_0|\bigr)^2\\
&+\frac{C_\kappa(T)\Delta t}{\epsilon}\int_{t_n}^{t_{n+1}}\frac{1}{s^{1-\frac{\kappa}{2}}}ds(1+|\IL^{\frac{\kappa}{2}}x_0|)^2+C_\kappa(T)\bigl(\frac{\Delta t}{\epsilon}\bigr)^{\frac12-\kappa}(1+|x_0|).
\end{align*}

$\bullet$ Error term $e_0^{0,4}$.

Note that, owing to the regularity estimate~\eqref{eq:lemuepsilon-1} from Lemma~\ref{lem:uepsilon} (with $\alpha=0$), for all $t\in[0,T)$ and $x,y\in H$, one has
\begin{align*}
|e_0^{0,4}|&\le \int_{t_0}^{t_{1}}\E[|\langle D_xu^\epsilon(T-t,x_0^{\epsilon,\Delta t},\tilde{\YY}^{\epsilon,\Delta t}(t)),F(\XX_0^{\epsilon,\Delta t},\tilde{\YY}^{\epsilon,\Delta t}(t))\rangle|]\\
&+\Delta t\E[|\langle D_xu^\epsilon(T-t_{1},x_0^\epsilon,\YY^{\epsilon,\Delta t}(t_{1})),F(\XX_0^{\epsilon,\Delta t},\tilde{\YY}^{\epsilon,\Delta t}(t_1))\rangle|]\\
&\le C(T)\Delta t(1+|x_0|),
\end{align*}
using the moment bound~\eqref{eq:boundtildeYYkappa}, the linear growth of $F$ and Assumption~\ref{ass:init}.

$\bullet$ Gathering the estimates for the error terms $e_n^{0,4,1}$, $e_n^{0,4,2}$ and $e_n^{0,4,3}$ and summing for $n\in\{0,\ldots,N-1\}$, the proof of the inequality~\eqref{eq:ineqen03} is thus completed.
\end{proof}

\subsection{Proof of Proposition~\ref{propo:error_scheme-limitingscheme}}\label{sec:proof-UA3}

Before proceeding with the proof, auxiliary tools are required. The statements and the arguments are similar to those in~\cite{BR}. Let us first state and prove an auxiliary lemma about discrete-time Poisson equations. For all $\tau\in(0,\infty)$, let $\bigl(\YY_k^\tau\bigr)_{k\ge 0}$ be defined using the modified Euler scheme from~\cite{B} applied to the stochastic evolution equation $d\YY(s)=-\IL\YY(s)ds+dW(s)$ with time-step size $\tau$: for all $m\ge 0$,
\begin{equation}\label{eq:schemePoisson}
\YY_{m+1}^\tau=\IA_\tau\YY_m^{\tau}+\IB_{\tau,1}\sqrt{\tau}\Gamma_{m,1}+\IB_{\tau,2}\sqrt{\tau}\Gamma_{m,2},
\end{equation}
where the linear operators $\IA_\tau$, $\IB_{\tau,1}$ and $\IB_{\tau,2}$ are given by~\eqref{eq:operators}. Let $\IP_\tau$ denote the associated Markov transition operator: for any bounded and measurable mapping $\phi:H\to\R$ and all $y\in H$,
\[
\IP_\tau\phi(y)=\E_y[\phi(\YY_1^{\tau})]=\E[\phi(\IA_\tau y+\IB_{\tau,1}\sqrt{\tau}\Gamma_{1,1}+\IB_{\tau,2}\sqrt{\tau}\Gamma_{1,2})].
\]
\begin{lemma}\label{lem:Poisson}
Let $\phi:H\to \R$ be a Lipschitz continuous function, which satisfies the centering condition $\int\phi(y)d\nu(y)=0$. For all $\tau\in(0,\infty)$ and all $y\in H$, define
\[
\psi^\tau(y)=\tau\sum_{m=0}^{\infty}\IP_\tau^m\phi(y).
\]
Then $\psi^\tau$ is a solution of the Poisson equation
\begin{equation}\label{eq:Poisson}
\IP_\tau \psi-\psi=-\tau\phi.
\end{equation}
Moreover, there exists $C\in(0,\infty)$, such that for all $\tau\in(0,\infty)$, one has
\begin{equation}\label{eq:boundPoisson}
\underset{y\in H}\sup~\frac{|\psi^{\tau}(y)|}{1+|y|}\le C\max(\tau,1)\underset{y_1,y_2\in H}\sup~\frac{|\phi(y_2)-\phi(y_1)|}{|y_2-y_1|}.
\end{equation}
\end{lemma}

\begin{proof}[Proof of Lemma~\ref{lem:Poisson}]
Observe that for all $\tau\in(0,\infty)$, one has
\[
\tau\sum_{m=0}^{\infty}\frac{1}{(1+\lambda_1\tau)^m}=\frac{\tau}{1-\frac{1}{1+\lambda_1\tau}}=\frac{1+\lambda_1\tau}{\lambda_1}\le C\max(\tau,1).
\]
In addition, using the centering condition on $\phi$ and the fact that the Gaussian distribution $\nu$ is invariant for the modified Euler scheme~\eqref{eq:schemePoisson} for any value of $\tau\in(0,\infty)$, for all $m\ge 0$, one has
\begin{align*}
|\IP_\tau^m\phi(y)|&=\big|\IP_\tau^m\phi(y)-\int\phi d\nu|\\
&=\big|\IP_\tau^m\phi(y)-\int\IP_\tau^m\phi(z)d\nu(z)\big|\\
&\le \underset{y_1,y_2\in H}\sup~\frac{|\phi(y_2)-\phi(y_1)|}{|y_2-y_1|}\int |\IA_\tau^m(y-z)|d\nu(z)\\
&\le \frac{1}{(1+\lambda_1\tau)^m}\underset{y_1,y_2\in H}\sup~\frac{|\phi(y_2)-\phi(y_1)|}{|y_2-y_1|}(\int|z|d\nu(z)+|y|).
\end{align*}
This proves that $\psi^\tau$ is well-defined for all $\tau\in(0,\infty)$. It is then straightforward to check that the identity~\eqref{eq:Poisson} and the inequality~\eqref{eq:boundPoisson}. The proof of Lemma~\ref{lem:Poisson} is thus completed.
\end{proof}

Let $\Delta t=T/N\in(0,\Delta t_0)$. For all $n\in\{0,\ldots,N-1\}$, define the auxiliary function $\phi_n^{\Delta t}:H\times H\to\R$ as follows: for all $x,y\in H$, set
\begin{equation}\label{eq:phi_n}
\phi_n^{\Delta t}(x,y)=\langle D\overline{u}_{N-n-1}^{\Delta t}(\IA_{\Delta t}x),\IA_{\Delta t}\bigl(F(x,y)-\overline{F}(x)\bigr)\rangle.
\end{equation}
Note that the centering condition
\[
\int \phi_n^{\Delta t}(x,\cdot)d\nu=0
\]
is satisfied, owing to the definition~\eqref{eq:Fbar} of $\overline{F}(x)$. Therefore one can define the auxiliary functions $\psi_n^{\Delta t}:H\times H\to\R$ as follows:
\begin{equation}\label{eq:psi_n}
\psi_n^{\Delta t,\epsilon}(x,y)=\tau\sum_{m=0}^{\infty}\E_y[\phi_n^{\Delta t}(x,\YY_m^\tau)],
\end{equation}
using the definition~\eqref{eq:schemePoisson} for the auxiliary scheme with time-step size $\tau=\Delta t/\epsilon$. Owing to Lemma~\ref{lem:Poisson}, $\psi_n^{\Delta t,\epsilon}(x,\cdot)$ is solution of the discrete Poisson equation~\eqref{eq:Poisson}:
\[
\IP_\tau \psi_n^{\Delta t,\epsilon}(x,\cdot)-\psi_n^{\Delta t,\epsilon}(x,\cdot)=-\tau\phi_n^{\Delta t}(x,\cdot).
\]
One has the following regularity estimates on the functions $\psi_n^{\Delta t,\epsilon}(x,\cdot)$, with constants independent of $\Delta t\in(0,\Delta t_0)$ and $\epsilon\in(0,\epsilon_0)$.
\begin{lemma}\label{lem:phin_psin}
For all $T\in(0,\infty)$ and $\kappa\in(0,1]$, there exists $C_\kappa(T)\in(0,\infty)$ such that for all $\epsilon\in(0,\epsilon_0)$, $\Delta t=T/N\in(0,\Delta t_0)$, all $x,y\in H$ and all $n\in\{0,\ldots,N-1\}$, one has
\begin{align}
|\psi_n^{\Delta t,\epsilon}(x,y)|&\le C_1(T)\vvvert\varphi\vvvert_1\max(\tau,1)(1+|y|)\label{eq:psin-1}\\
\underset{h\in H;|h|\le 1}\sup~|\langle D_x\psi_n^{\Delta t,\epsilon}(x,y),h\rangle|&\le C_1(T)\vvvert\varphi\vvvert_2\max(\tau,1)(1+|y|)\label{eq:psin-2}
\end{align}
and for all $n\in\{0,\ldots,N-3\}$, one has
\begin{equation}\label{eq:psin-3}
|\psi_{n+1}^{\Delta t,\epsilon}(x,y)-\psi_n^{\Delta t,\epsilon}(x,y)|\le \frac{C_\kappa(T)\Delta t^{1-\kappa}}{\bigl((N-n-2)\Delta t\bigr)^{1-\kappa}}\vvvert\varphi\vvvert_2\max(\tau,1)(1+|x|)(1+|y|),
\end{equation}
with $\tau=\Delta t/\epsilon$.
\end{lemma}

The proof of Lemma~\ref{lem:phin_psin} consists in the application of Lemma~\ref{lem:Poisson} for three auxiliary mappings, combined with the regularity results on $\overline{u}_{N-n-1}^{\Delta t}$ from Lemma~\ref{lem:ubarDeltat}. The application of Lemma~\ref{lem:Poisson} explains the presence of the factor $\max(\tau,1)$ on the right-hand sides of the inequalities, see the inequality~\eqref{eq:boundPoisson} from Lemma~\ref{lem:Poisson}.

\begin{proof}[Proof of Lemma~\ref{lem:momentboundscheme}]
Let us first prove the inequality~\eqref{eq:psin-1}. For all $n\in\{0,\ldots,N-1\}$, $x,y_1,y_2\in H$, one has
\begin{align*}
|\phi_n^{\Delta t}(x,y_2)-\phi_n^{\Delta t}(x,y_1)|&=|\langle D\overline{u}_{N-n-1}^{\Delta t}(\IA_{\Delta t}x),\IA_{\Delta t}(F(x,y_2)-F(x,y_1))\rangle|\\
&\le C(T)\vvvert\varphi\vvvert_1|F(x,y_2)-F(x,y_1)|\\
&\le C(T)\vvvert\varphi\vvvert_1|y_2-y_1|,
\end{align*}
owing to the inequality~\eqref{eq:lemubarDeltat_1} (see Lemma~\ref{lem:ubarDeltat}) and to the global Lipschitz continuity of $F$ (Assumption~\ref{ass:F}). Since $\phi_n^{\Delta t}(x,\cdot)$ satisfies the centering condition $\int\phi_n^{\Delta t}(x,\cdot)d\nu=0$, the inequality~\eqref{eq:psin-1} is then a straightforward consequence of Lemma~\ref{lem:Poisson}. 

Let us now prove the inequality~\eqref{eq:psin-2}. Since the mappings $\overline{u}_{N-n-1}^{\Delta t}$, $F$ and $\overline{F}$ are of class $\mathcal{C}^2$ (see Lemma~\ref{lem:ubarDeltat} and Assumption~\ref{ass:F}), $x\mapsto \phi_n^{\Delta t}(x,y)$ is of class $\mathcal{C}^1$, and one has
\begin{align*}
\langle D_x\phi_n^{\Delta t}(x,y),h\rangle&=\langle D\overline{u}_{N-n-1}^{\Delta t}(\IA_{\Delta t}x),\IA_{\Delta t}\bigl(D_xF(x,y).h-D\overline{F}(x).h\bigr)\rangle\\
&+D^2\overline{u}_{N-n-1}^{\Delta t}(\IA_{\Delta t}x).\bigl(\IA_{\Delta t}(F(x,y)-\overline{F}(x)),\IA_{\Delta t}h\bigr).
\end{align*}
In particular, the centering condition
\[
\int\langle D_x\phi_n^{\Delta t}(x,\cdot),h\rangle d\nu=0
\]
is satisfied. It is straightforward to check that $x\mapsto \psi_n^{\Delta t,\epsilon}(x,y)$ is of class $\mathcal{C}^1$, and that one has
\[
\langle D_x\psi_n^{\Delta t,\epsilon}(x,y),h\rangle=
\tau\sum_{k=0}^{\infty}\E_y[\langle D_x\phi_{n}^{\Delta t}(x,\YY_{k}^{\tau}),h\rangle.
\]
This means that the mapping $\langle D_x\psi_{n}^{\Delta t,\epsilon}(x,\cdot),h\rangle$ solves the Poisson equation
\[
(\IP_\tau-I)\langle D_x\psi_{n}^{\Delta t,\epsilon}(x,\cdot),h\rangle=\langle D_x\phi_{n}^{\Delta t}(x,\cdot),h\rangle.
\]
In order to apply Lemma~\ref{lem:Poisson}, it suffices to check tha the following property holds: for all $n\in\{0,\ldots,N-1\}$, $x,y_1,y_2\in H$, one has
\begin{align*}
\big|\langle D_x\phi_{n+1}^{\Delta t}\tau(x,y_2),h\rangle&-\langle D_x\psi_{n+1}^{\Delta t}(x,y_1),h\rangle\big|\\
&\le \big|\langle D\overline{u}_{N-n-2}^{\Delta t}(\IA_{\Delta t}x),\IA_{\Delta t}(D_xF(x,y_2).h-D_xF(x,y_2).h)\rangle|\\
&+\big|D^2\overline{u}_{N-n-2}^{\Delta t}(\IA_{\Delta t}x).\bigl(\IA_{\Delta t}h,\IA_{\Delta t}(F(x,y_2)-F(x,y_1))\bigr)\big|\\
&\le C(T)\vvvert\varphi\vvvert_2|h||y_2-y_1|,
\end{align*}
owing to the inequality~\eqref{eq:lemubarDeltat_4} (see Lemma~\ref{lem:ubarDeltat}) and to the regularity conditions on $F$ ($F$ is of class $\mathcal{C}^2$ with bounded first and second order derivatives, see Assumption~\ref{ass:F}). As a consequence, the inequality~\eqref{eq:psin-2} is obtained as an application of Lemma~\ref{lem:Poisson}.

It finally remains to prove the inequality~\eqref{eq:psin-3}. Set $\delta\psi_n^{\Delta t,\epsilon}(x,y)=\psi_{n+1}^{\Delta t,\epsilon}(x,y)-\psi_{n}^{\Delta t,\epsilon}(x,y)$. The mapping $\delta\psi_n^{\Delta t,\epsilon}(x,\cdot)$ is solution of the Poisson equation
\[
(\IP_\tau-I)\delta\psi_n^{\Delta t,\epsilon}(x,\cdot)=\tau\delta\phi_n^{\Delta t}(x,\cdot),
\]
with the auxiliary function $\delta\phi_n^{\Delta t}(x,\cdot)$ defined by
\[
\delta\phi_n^{\Delta t}(x,y)=\langle D\overline{u}_{N-n-2}^{\Delta t}(\IA_{\Delta t}x)-D\overline{u}_{N-n-1}^{\Delta t}(\IA_{\Delta t}x),\IA_{\Delta t}(F(x,y)-\overline{F}(x))\rangle.
\]
The centering condition
\[
\int\delta\phi_n^{\Delta t}(x,\cdot)d\nu=0
\]
is satisfied, therefore the application of Lemma~\ref{lem:Poisson} requires to upper bound the Lipschitz constant of $\delta\phi_n^{\Delta t}(x,\cdot)$.

For all $x\in H$, $n\in\{0,\ldots,N-3\}$, and $y_1,y_2\in H$, one has
\begin{align*}
\big|\delta_n\phi^{\tau}(x,y_2)-&\delta_n\phi^{\tau}(x,y_1)\big|\le \big|\langle D\overline{u}_{N-n-2}^{\Delta t}(\IA_{\Delta t}x)-D\overline{u}_{N-n-1}^{\Delta t}(\IA_{\Delta t}x),\IA_{\Delta t}(F(x,y_2)-F(x,y_1))\rangle\big|\\
&\le \Delta t^{1-\kappa}\frac{C_\kappa(T)}{\bigl((N-n-2)\Delta t\bigr)^{1-\kappa}}\vvvert\varphi\vvvert_2(1+|x|)\big|\IA_{\Delta t}(F(x,y_2)-F(x,y_1))\big|\\
&\le \Delta t^{1-\kappa}\frac{C_\kappa(T)}{\bigl((N-n-2)\Delta t\bigr)^{1-\kappa}}\vvvert\varphi\vvvert_2(1+|x|)|y_2-y_1|,
\end{align*}
owing to the inequality~\eqref{eq:lemubarDeltat_3} (see Lemma~\ref{lem:ubarDeltat}) and to the global Lipschitz continuity of $F$ (Assumption~\ref{ass:F}). Applying Lemma~\ref{lem:Poisson} then yields the inequality~\eqref{eq:psin-3}.

The proof of Lemma~\ref{lem:phin_psin} is thus completed.
\end{proof}

We are now in position to provide the proof of Proposition~\ref{propo:error_scheme-limitingscheme}.

\begin{proof}[Proof of Proposition~\ref{propo:error_scheme-limitingscheme}]
Let $T\in(0,\infty)$, $\varphi:H\to\R$ be of class $\mathcal{C}^2$, with bounded first and second order derivatives, $\epsilon\in(0,\epsilon_0)$, and $\Delta t=T/N\in(0,\Delta t_0)$, with $N\in\N$. Recall the notation $\tau=\Delta t/\epsilon$.

The error in the left-hand side of~\eqref{eq:error_scheme-limitingscheme} can written as follows:
\[
\E[\varphi(\XX_N^{\epsilon,\Delta t})]-\E[\varphi(\XX_N^{\Delta t})]=\E[\varphi(\XX_N^{\epsilon,\Delta t})]-\varphi(\overline{\XX}_N^{\Delta t})+\varphi(\overline{\XX}_N^{\Delta t})-\E[\varphi(\XX_N^{\Delta t})].
\]
It suffices to focus on the first error term on the left-hand side: indeed
\[
\big|\varphi(\overline{\XX}_N^{\Delta t})-\E[\varphi(\XX_N^{\Delta t})]\big|\le C(T)\Delta t\vvvert\varphi\vvvert_2(1+|x_0|^2),
\]
see~\eqref{eq:errorXbarX} from the proof of Proposition~\ref{propo:error_limitingscheme-averagedequation}.

The error term which remain to be studied can be decomposed as follows:
\begin{align*}
\E[\varphi(\XX_N^{\epsilon,\Delta t})]-\varphi(\overline{\XX}_N^{\Delta t})&=\E[\overline{u}_0^{\Delta t}(\XX_N^{\epsilon,\Delta t})]-\overline{u}_N^{\Delta t}(x_0)\\
&=\E[\overline{u}_0^{\Delta t}(\XX_N^{\epsilon,\Delta t})]-\E[\overline{u}_N^{\Delta t}(\XX_0^{\epsilon,\Delta t})]+\overline{u}_N^{\Delta t}(x_0^\epsilon)-\overline{u}_N^{\Delta t}(x_0),
\end{align*}
where the mapping $\overline{u}_{N}^{\Delta t}$ is given by~\eqref{eq:ubar}.

The mapping $\overline{u}_{N}^{\Delta t}$ is globally Lipschitz continuous, more precisely one has the inequality~\eqref{eq:lemubarDeltat_4}, and using Assumption~\ref{ass:init} one obtains
\[
\big|\overline{u}_N^{\Delta t}(x_0^\epsilon)-\overline{u}_N^{\Delta t}(x_0)\big|\le C(T)|x_0^\epsilon-x_0|\le C(T)\epsilon(1+|x_0|).
\]
Let us now study the remaining error term. Using a telescoping sum argument, one has
\begin{align*}
\E[\overline{u}_0^{\Delta t}&(\XX_N^{\epsilon,\Delta t})]-\E[\overline{u}_N^{\Delta t}(\XX_0^{\epsilon,\Delta t})]=\sum_{n=0}^{N-1}\bigl(\E[\overline{u}_{N-n-1}^{\Delta t}(\XX_{n+1}^{\epsilon,\Delta t})]-\E[\overline{u}_{N-n}^{\Delta t}(\XX_n^{\epsilon,\Delta t})]\bigr)\\
&=\sum_{n=0}^{N-1}\Bigl(\E[\overline{u}_{N-n-1}^{\Delta t}(\IA_{\Delta t}\XX_n^{\epsilon,\Delta t}+\Delta t\IA_{\Delta t}F(\XX_n^{\epsilon,\Delta t},\YY_{n+1}^{\epsilon,\Delta t}))]\\
&\hspace{2cm}-\E[\overline{u}_{N-n-1}^{\Delta t}(\IA_{\Delta t}\XX_n^{\epsilon,\Delta t}+\Delta t\overline{F}(\XX_n^{\epsilon,\Delta t}))]\Bigr).
\end{align*}

Since $\overline{u}_n^{\Delta t}$ is of class $\mathcal{C}^2$ owing to Lemma~\ref{lem:ubarDeltat}, by a Taylor expansion, one obtains the equality
\[
\E[\overline{u}_0^{\Delta t}(\XX_N^{\epsilon,\Delta t})]-\E[\overline{u}_N^{\Delta t}(\XX_N^{\epsilon,\Delta t})]=\Delta t\sum_{n=0}^{N-1}\E[\phi_{n}^{\Delta t}(\XX_n^{\epsilon,\Delta t},\YY_{n+1}^{\epsilon,\Delta t})]+R_n^{\epsilon,\Delta t},
\]
with the function $\phi_n^{\Delta t}$ defined by~\eqref{eq:phi_n}, and where one has
\begin{align*}
\E[|R_n^{\epsilon,\Delta t}|]&\le \Delta t^2\vvvert\overline{u}_{N-n-1}^{\Delta t}\vvvert_2\E[|F(\XX_n^{\epsilon,\Delta t},\YY_{n+1}^{\epsilon,\Delta t})-\overline{F}(\XX_n^{\epsilon,\Delta t})|^2]\\
&\le C(T)\Delta t^2\vvvert\varphi\vvvert_2\bigl(1+\E[|\XX_n^{\epsilon,\Delta t}|^2]+\E[|\YY_{n+1}^{\epsilon,\Delta t}|^2]\bigr),
\end{align*}
using the inequality~\eqref{eq:lemubarDeltat_4}.

Using the equality $\Delta t=\tau\epsilon$, the error term can be written as
\begin{equation}\label{eq:expressionerror}
\begin{aligned}
\Delta t\sum_{n=0}^{N-1}\E[\phi_{n}^{\Delta t}(\XX_n^{\epsilon,\Delta t},\YY_{n+1}^{\epsilon,\Delta t})]&=\epsilon\sum_{n=0}^{N-1}\bigl(\E[\psi_{n}^{\Delta t,\epsilon}(\XX_n^{\epsilon,\Delta t},\YY_{n+1}^{\epsilon,\Delta t})]-\E[\IP_\tau\psi_{n}^{\Delta t,\epsilon}(\XX_n^{\epsilon,\Delta t},\YY_{n+1}^{\epsilon,\Delta t})]\bigr)\\
&=\epsilon\sum_{n=0}^{N-1}\bigl(\E[\psi_{n}^{\Delta t,\epsilon}(\XX_n^{\epsilon,\Delta t},\YY_{n+1}^{\epsilon,\Delta t})]-\E[\psi_{n}^{\Delta t,\epsilon}(\XX_n^{\epsilon,\Delta t},\YY_{n+2}^{\epsilon,\Delta t})]\bigr)\\
&=\epsilon\sum_{n\in\{0,N-1,N-2\}}\bigl(\E[\psi_{n}^{\Delta t,\epsilon}(\XX_n^{\epsilon,\Delta t},\YY_{n+1}^{\epsilon,\Delta t})]-\E[\psi_{n}^{\Delta t,\epsilon}(\XX_n^{\epsilon,\Delta t},\YY_{n+2}^{\epsilon,\Delta t})]\bigr)\\
&+\epsilon\sum_{n=1}^{N-3}\bigl(\E[\psi_{n}^{\Delta t,\epsilon}(\XX_n^{\epsilon,\Delta t},\YY_{n+1}^{\epsilon,\Delta t})]-\E[\psi_{n}^{\Delta t,\epsilon}(\XX_n^{\epsilon,\Delta t},\YY_{n+2}^{\epsilon,\Delta t})]\bigr)
\end{aligned}
\end{equation}
where the second equality is a consequence of the Markov property and of the definition of the scheme~\eqref{eq:scheme}. For technical reasons, it is necessary to treat differently the terms $n\in\{0,N-2,N-1\}$ and $n\in\{1,\ldots,N-3\}$.

On the one hand, for if $n\in\{0,N-2,N-1\}$, owing to Lemma~\ref{lem:phin_psin} one has
\begin{align*}
\epsilon\Big|\E[\psi_{n}^{\Delta t,\epsilon}(\XX_n^{\epsilon,\Delta t},\YY_{n+1}^{\epsilon,\Delta t})]&-\E[\psi_{n}^{\Delta t,\epsilon}(\XX_n^{\epsilon,\Delta t},\YY_{n+2}^{\epsilon,\Delta t})]\Big|\\
&\le C(T)\vvvert\varphi\vvvert_1\epsilon\max(\tau,1)\bigl(1+\E[|\YY_{n+1}^{\epsilon,\Delta t}|]+\E[|\YY_{n+2}^{\epsilon,\Delta t}|]\bigr)\\
&\le C(T)\vvvert\varphi\vvvert_1\max(\Delta t,\epsilon),
\end{align*}
using the moment bound~\eqref{eq:boundYnepsilonDeltat} from Lemma~\ref{lem:momentboundsscheme-epsilonDeltat}.

On the other hand, using a telescoping sum argument and the Markov property, one obtains the auxiliary identities
\begin{align*}
\E[\IP_\tau\psi_{N-2}^{\Delta t,\epsilon}(\XX_{N-2}^{\epsilon,\Delta t},\YY_{N-2}^{\epsilon,\Delta t})]&-\E[\IP_\tau\psi_1^{\Delta t,\epsilon}(\XX_1^{\epsilon,\Delta t},\YY_1^{\epsilon,\Delta t})]\\
&=\sum_{n=1}^{N-3}\bigl(\E[\IP_\tau\psi_{n+1}^{\Delta t,\epsilon}(\XX_{n+1}^{\epsilon,\Delta t},\YY_{n+1}^{\epsilon,\Delta t})]-\E[\IP_\tau\psi_n^{\Delta t,\epsilon}(\XX_n^{\epsilon,\Delta t},\YY_n^{\epsilon,\Delta t})]\bigr)\\
&=\sum_{n=1}^{N-3}\bigl(\E[\IP_\tau\psi_{n+1}^{\Delta t,\epsilon}(\XX_{n+1}^{\epsilon,\Delta t},\YY_{n+1}^{\epsilon,\Delta t})]-\E[\IP_\tau\psi_{n}^{\Delta t,\epsilon}(\XX_{n}^{\epsilon,\Delta t},\YY_{n+1}^{\epsilon,\Delta t})]\\
&+\sum_{n=1}^{N-3}\bigl(\E[\IP_\tau\psi_{n}^{\Delta t,\epsilon}(\XX_{n}^{\epsilon,\Delta t},\YY_{n+1}^{\epsilon,\Delta t})]-\E[\IP_\tau\psi_{n}^{\Delta t,\epsilon}(\XX_{n}^{\epsilon,\Delta t},\YY_{n}^{\epsilon,\Delta t})]\bigr)\\
&=\sum_{n=1}^{N-3}\bigl(\E[\psi_{n+1}^{\Delta t,\epsilon}(\XX_{n+1}^{\epsilon,\Delta t},\YY_{n+2}^{\epsilon,\Delta t})]-\E[\psi_{n}^{\Delta t,\epsilon}(\XX_{n}^{\epsilon,\Delta t},\YY_{n+2}^{\epsilon,\Delta t})]\\
&+\sum_{n=1}^{N-3}\bigl(\E[\psi_{n}^{\Delta t,\epsilon}(\XX_{n}^{\epsilon,\Delta t},\YY_{n+2}^{\epsilon,\Delta t})]-\E[\psi_{n}^{\Delta t,\epsilon}(\XX_{n}^{\epsilon,\Delta t},\YY_{n+1}^{\epsilon,\Delta t})]\bigr).
\end{align*}
Observe that the expression appearing in the last line above corresponds to the expression appearing in the last line of~\eqref{eq:expressionerror}. One then obtains
\begin{align*}
\Delta t\sum_{n=1}^{N-3}\E[\phi_{n}^{\Delta t}(\XX_n^{\epsilon,\Delta t},\YY_{n+1}^{\epsilon,\Delta t})]&=\epsilon \E[\psi_{1}^{\Delta t,\epsilon}(\XX_1^\epsilon,\YY_2^{\epsilon,\Delta t})]-\epsilon\E[\psi_{N-2}^{\Delta t,\epsilon}(\XX_{N-2}^{\epsilon,\Delta t},\YY_{N-1}^{\epsilon,\Delta t})]\\
&+\epsilon\sum_{n=1}^{N-3}\bigl(\E[\psi_{n+1}^{\Delta t,\epsilon}(\XX_{n+1}^{\epsilon,\Delta t},\YY_{n+2}^{\epsilon,\Delta t})]-\E[\psi_{n+1}^{\Delta t,\epsilon}(\XX_n^{\epsilon,\Delta t},\YY_{n+2}^{\epsilon,\Delta t})]\bigr)\\
&+\epsilon\sum_{n=1}^{N-3}\bigl(\E[\psi_{n+1}^{\Delta t,\epsilon}(\XX_{n}^{\epsilon,\Delta t},\YY_{n+2}^{\epsilon,\Delta t})]-\E[\psi_{n}^{\Delta t,\epsilon}(\XX_n^{\epsilon,\Delta t},\YY_{n+2}^{\epsilon,\Delta t})]\bigr).
\end{align*}
To prove upper bounds for the three terms on the right-hand side above, the properties of the mappings $\psi_n^{\Delta t,\epsilon}$ provided by Lemma~\ref{lem:phin_psin} and moment bounds for the random variables $\XX_n^{\Delta t,\epsilon}$ and $\YY_{n}^{\Delta t,\epsilon}$ are employed. Recall also that $\epsilon\max(\tau,1)=\max(\Delta t,\epsilon)$.

$\bullet$ Using the inequality~\eqref{eq:psin-1}, one has
\begin{align*}
\epsilon\big|\E[\psi_{1}^{\Delta t,\epsilon}(\XX_1^\epsilon,\YY_2^{\epsilon,\Delta t})]\big|&\le C(T)\max(\Delta t,\epsilon)\vvvert\varphi\vvvert_1(1+\E[|\YY_2^{\epsilon,\Delta t}|])\\
&\le C(T)\max(\Delta t,\epsilon)\vvvert\varphi\vvvert_1,
\end{align*}
using the moment bound~\eqref{eq:boundYnepsilonDeltat} from Lemma~\ref{lem:momentboundsscheme-epsilonDeltat}.

Similarly, one has
\begin{align*}
\epsilon\big|\E[\psi_{N-2}^{\Delta t,\epsilon}(\XX_{N-2}^{\epsilon,\Delta t},\YY_{N-1}^{\epsilon,\Delta t})]\big|&\le C(T)\max(\Delta t,\epsilon)\vvvert\varphi\vvvert_1(1+\E[|\YY_{N-1}^{\epsilon,\Delta t}|])\\
&\le C(T)\max(\Delta t,\epsilon)\vvvert\varphi\vvvert_1.
\end{align*}

$\bullet$ Using the inequality~\eqref{eq:psin-2}, one has
\begin{align*}
\Big|\epsilon\sum_{n=1}^{N-3}\bigl(\E[\psi_{n+1}^{\Delta t,\epsilon}(\XX_{n+1}^{\epsilon,\Delta t},\YY_{n+2}^{\epsilon,\Delta t})]&-\E[\psi_{n+1}^{\Delta t,\epsilon}(\XX_n^{\epsilon,\Delta t},\YY_{n+2}^{\epsilon,\Delta t})]\bigr)\Big|\\
&\le C(T)\max(\Delta t,\epsilon)\vvvert\varphi\vvvert_2\sum_{n=1}^{N-3}\E\bigl[|\XX_{n+1}^{\epsilon,\Delta t}-\XX_n^{\epsilon,\Delta t}|(1+|\YY_{n+2}^{\epsilon,\Delta t}|)\bigr]\\
&\le C(T)\max(\Delta t,\epsilon)\vvvert\varphi\vvvert_2\sum_{n=1}^{N-3}\bigl(\E[|\XX_{n+1}^{\epsilon,\Delta t}-\XX_n^{\epsilon,\Delta t}|^2]\bigr)^{\frac12},
\end{align*}
using H\"older's inequality and the moment bound~\eqref{eq:boundYnepsilonDeltat} from Lemma~\ref{lem:momentboundsscheme-epsilonDeltat}. Applying Lemma~\ref{lem:incrementsscheme} then yields
\begin{align*}
\Big|\epsilon\sum_{n=1}^{N-3}\bigl(\E[\psi_{n+1}^{\Delta t,\epsilon}(\XX_{n+1}^{\epsilon,\Delta t},\YY_{n+2}^{\epsilon,\Delta t})]&-\E[\psi_{n+1}^{\Delta t,\epsilon}(\XX_n^{\epsilon,\Delta t},\YY_{n+2}^{\epsilon,\Delta t})]\bigr)\Big|\\
&\le C_\kappa(T)\max(\Delta t,\epsilon)\vvvert\varphi\vvvert_2\Delta t^{1-\kappa}(1+|x_0|)\sum_{n=1}^{N-3}\frac{1}{(n\Delta t)^{1-\kappa}}\\
&\le C_\kappa(T)\max(\Delta t,\epsilon)\vvvert\varphi\vvvert_2\Delta t^{-\kappa}(1+|x_0|).
\end{align*}

$\bullet$ Using the inequality~\eqref{eq:psin-3}, one has
\begin{align*}
\Big|\epsilon&\sum_{n=1}^{N-3}\bigl(\E[\psi_{n+1}^{\Delta t,\epsilon}(\XX_{n}^{\epsilon,\Delta t},\YY_{n+2}^{\epsilon,\Delta t})]-\E[\psi_{n}^{\Delta t,\epsilon}(\XX_n^{\epsilon,\Delta t},\YY_{n+2}^{\epsilon,\Delta t})]\bigr)\Big|\\
&\le C_\kappa(T)\max(\Delta t,\epsilon)\vvvert\varphi\vvvert_2\Delta t^{1-\kappa} \sum_{n=1}^{N-3}\frac{1}{\bigl((N-n-2)\Delta t\bigr)^{1-\kappa}}\E[(1+|\XX_n^{\epsilon,\Delta t}|)(1+|\YY_{n+2}^{\epsilon,\Delta t}|)]\\
&\le C_\kappa(T)\max(\Delta t,\epsilon)\vvvert\varphi\vvvert_2\Delta t^{-\kappa}(1+|x_0|).
\end{align*}

$\bullet$ Gathering the estimates, one obtains
\[
\big|\Delta t\sum_{n=1}^{N-3}\E[\phi_{n}^{\Delta t}(\XX_n^{\epsilon,\Delta t},\YY_{n+1}^{\epsilon,\Delta t})]\big|\le C_\kappa(T)\max(\Delta t,\epsilon)\vvvert\varphi\vvvert_2\Delta t^{-\kappa}(1+|x_0|),
\]
and finally
\begin{align*}
\big|\E[\varphi(\XX_N^{\epsilon,\Delta t})]-\varphi(\overline{\XX}_N^{\Delta t})\big|&\le C(T)\epsilon\vvvert\varphi\vvvert_1(1+|x_0|)+C(T)\Delta t\vvvert\varphi\vvvert_2(1+|x_0|^2)\\
&+C(T)\vvvert\varphi\vvvert_1(\epsilon+\Delta t)\\
&+C_\kappa(T)\bigl(\frac{\epsilon}{\Delta t^\kappa}+\Delta t^{1-\kappa}\bigr)\vvvert\varphi\vvvert_2(1+|x_0|).
\end{align*}
This concludes the proof of the inequality~\eqref{eq:error_scheme-limitingscheme} and of Proposition~\ref{propo:error_scheme-limitingscheme}.
\end{proof}

\section*{Acknowledgments}
This work is partially supported by the following projects operated by the French National Research Agency: ADA (ANR-19-CE40-0019-02) and SIMALIN (ANR-19-CE40-0016).

%\bibliographystyle{abbrv}
%\bibliography{biblio}

\end{document}